\documentclass{amsart}
\usepackage{amssymb,mathtools}
\usepackage{stmaryrd}
\usepackage[british]{babel}
\usepackage{enumitem}
\usepackage[latin1]{inputenc}
\usepackage{hyperref}

\allowdisplaybreaks

\usepackage{bussproofs}
\newenvironment{scprooftree}[1]%
  {\gdef\scalefactor{#1}\begin{center}\proofSkipAmount \leavevmode}%
  {\scalebox{\scalefactor}{\DisplayProof}\proofSkipAmount \end{center} }

\makeatletter
\newcommand{\dotminus}{\mathbin{\text{\@dotminus}}}
\newcommand{\@dotminus}{%
  \ooalign{\hidewidth\raise1ex\hbox{.}\hidewidth\cr$\m@th-$\cr}%
}
\makeatother

\newtheorem{theorem}{Theorem}
\numberwithin{theorem}{subsection}
\newtheorem{corollary}[theorem]{Corollary}
\newtheorem{lemma}[theorem]{Lemma}
\newtheorem{proposition}[theorem]{Proposition}
\theoremstyle{definition}
\newtheorem{definition}[theorem]{Definition}
\newtheorem{remark}[theorem]{Remark}

\newtheorem{exercise}[theorem]{Exercise}
\newtheorem{convention}[theorem]{Convention}

\newcommand{\nega}{{\sim}}
\newcommand{\fv}{\operatorname{FV}}
\newcommand{\len}{\operatorname{l}}
\newcommand{\gn}[1]{\left\ulcorner#1\right\urcorner}
\newcommand{\sub}{\operatorname{sub}}
\newcommand{\prf}{\operatorname{Prf}}
\newcommand{\pr}{\operatorname{Pr}}
\newcommand{\vdashir}{\vdash^{\mathsf{IR}}}
\newcommand{\nf}[1]{=_{#1\text{-}\textsf{NF}}}

\title[Introduction to Mathematical Logic]{An Introduction to Mathematical Logic}
\author{Anton Freund}

\address{Anton Freund, University of W\"urzburg, Institute of Mathematics, Emil-Fischer-Stra{\ss}e~40, 97074~W\"urzburg, Germany}
\email{anton.freund@uni-wuerzburg.de}

\begin{document}

\begin{abstract}
This introduction begins with a section on fundamental notions of mathematical logic, including propositional logic, predicate or first-order logic, completeness, compactness, the L\"owenheim-Skolem theorem, Craig interpolation, Beth's definability theorem and Herbrand's theorem. It continues with a section on G\"odel's incompleteness theorems, which includes a discussion of first-order arithmetic and primitive recursive functions. This is followed by three sections that are devoted, respectively, to proof theory (provably total recursive functions and Goodstein sequences for~$\mathsf{I\Sigma}_1$), computability (fundamental notions and an analysis of K\H{o}nig's lemma in terms of the low basis theorem) and model theory (ultra\-products, chains and the Ax-Grothendieck theorem). We conclude with some brief introductory remarks about set theory (with more details reserved for a separate lecture). The author uses these notes for a first logic course for undergraduates in mathematics, which consists of 28 lectures and 14 exercise sessions of 90 minutes each. In such a course, it may be necessary to omit some material, which is straightforward since all sections except for the first two are independent of each other.
\end{abstract}


\maketitle

\tableofcontents

\section{Introduction}\label{sect:introduction}

Mathematical logic investigates what can be proved, defined and computed in vari\-ous mathematical settings. It does so by rigorous mathematical methods (`metamathematics'). Sometimes, one is interested in a logical analysis for its own sake. For example, one might be curious whether some axiom is indispensable for the proof of a given theorem (`foundations'). At other times, the focus is on applications in different~areas of mathematics or in computer science. One can, e.\,g., use tools from logic to extract quantitative bounds from proofs in analysis or to prove results in algebra. There are close connections between foundations and \mbox{applications}. For example, one can show that all proofs from some set of axioms give rise to algorithms of a certain complexity.

These lecture notes are intended as a first introduction to mathematical logic. It is assumed that the reader is familiar with the mathematical results and methods that are typically covered in a first year of university studies. Only the last subsection of these notes assumes additional knowledge from algebra, but even here the central ideas should be transparent without many prerequisites.

The present introduction is Section~\ref{sect:introduction} of these lecture notes. Section~\ref{sect:fundamentals} is concerned with basic concepts and results that are fundamental across mathematical logic. In particular, we present propositional and first-order logic. We then discuss completeness, compactness, the downward L\"owenheim-Skolem theorem, Craig interpolation, Beth definability and Herbrand's theorem. In Section~\ref{sect:goedel}, we introduce first-order arithmetic and the concept of primitive recursion. We then prove Tarski's result on the so-called undefinability of truth as well as G\"odel's first incompleteness theorem. Assuming the Hilbert-Bernays conditions, we also derive G\"odel's second incompleteness theorem and L\"ob's theorem. We note that Sections~\ref{sect:fundamentals} and~\ref{sect:goedel} contain prerequisites that are needed in the rest of these notes.

Sections~\ref{sect:proof-theory} to~\ref{sect:set-theory} are devoted to four main subareas of mathematical logic, namely the theories of proofs, computations, models and sets. The four sections are largely independently of each other, so that the reader or lecturer can easily select according to time constraints. Our idea was to present one reasonably advanced application from each subarea, including those fundamental concepts that are relevant for it. In the section on proof theory, we determine the provably total recursive functions of the axiom system~$\mathsf{I\Sigma}_1$ and derive an independence result for non-hereditary Goodstein sequences. The section on computability theory (also known as recursion theory) provides an analysis of K\H{o}nig's lemma and its weak variant via the notion of Turing jump and the low basis theorem. In the section on model theory, we first use ultraproducts to prove the upward L\"owenheim-Skolem theorem and the uncountable case of completeness and compactness (as only the countable case is treated in Section~\ref{sect:fundamentals}). We then derive a theorem of Ax and Grothendieck, which says that any injective polynomial $\mathbb C^n\to\mathbb C^n$ is~surjective. In the case of set theory, we do not include a full section but only a brief overview~in the form of an epilogue. This is because the easiest substantial result that the author could think of (namely the relative consistency of the axiom of choice) is already too involved for the present introduction to logic. The author prefers to do set theory justice by giving a separate course on the subject.

As mentioned in the previous paragraph, our aim was to include one reasonably advanced result from each area that we present. This made it necessary to omit many important topics altogether. A particularly glaring omission concerns~in\-tu\-i\-tion\-istic logic, constructive mathematics and type theory. Computational complexity theory is missing except for the short Exercise~\ref{ex:cnf-poly}. We hope that our lecture notes provide a solid starting point from which to explore these and many other topics that make up the diverse area that mathematical logic is today.

\section{Fundamentals}\label{sect:fundamentals}

As stated in the introduction, mathematical logic investigates what one can prove, define and compute in a given setting. All three activities are concerned with symbolic representations. In particular, proofs and definitions are expressed in some language. This is usually a so-called natural language such as English, which has developed historically and connects mathematics with other fields of thought. At the same time, it becomes more and more common to implement proofs on a computer, where one uses a so-called formal language. For such a language, the notions of well-formed expression and correct proof can be defined with mathematical precision. This is crucial for the purposes of mathematical logic. Let us note that we will often be concerned with two levels of language: the so-called object language, which is the formal language that we investigate, and the so-called metalanguage, in which the investigation takes place.

Mathematical logic considers various formal languages, each of which is particularly suitable for some application. At the same time, there is one formal language of singular importance, which is known as predicate logic or first-order logic. The latter forms the basis for axiom systems such as Peano arithmetic and Zermelo-Fraenkel set theory (see Sections~\ref{subsect:PA} and~\ref{sect:set-theory}). It thus provides a framework in which essentially all results of contemporary mathematics can be accommodated (though mathematics does, of course, consist of more than formalized results).

In Subsection~\ref{subsect:PL}, which is the first part of the present section, we present propositional logic. This is a rather weak fragment of first-order logic. It plays an important role in areas such as computational complexity, but we mainly discuss it to introduce central ideas in a simplified setting. In particular, Subsection~\ref{subsect:complete-PL} presents two notions of formal proof for predicate logic, which are known as natural deduction and the sequent calculus. Both notions of proof are sound, which means that any statement that can be proved is valid in a certain precise sense. We will also show that they are complete, i.\,e., that any valid statement is provable. It follows that natural deduction and the sequent calculus are equivalent to each other and to any other system that is sound and complete. This means that there is a single canonical notion of proof in logic. In contrast, mathematical truth cannot be captured via provability in a single axiom system, as we will see when we study G\"odel's incompleteness theorems in Section~\ref{sect:goedel}.

First-order logic and a sound and complete proof system for it are presented in Subsection~\ref{subsect:FO}. Let us note that completeness is often proved by a so-called Henkin construction. We give a different proof via Sch\"utte's deduction chains. This proof has the disadvantage that it seems to cover countable languages only. Results for uncountable languages will be proved via ultraproducts in Section~\ref{subsect:ultraprod}. The advantage of Sch\"utte's approach is that it directly yields cut-free proofs. This will be exploited in Subsection~\ref{sect:cut-free-completeness}, where we derive not only compactness and the downward L\"owenheim-Skolem theorem but also Craig interpolation, Beth's definability theorem and Herbrand's theorem.

To conclude this introduction, we point out a few sources that are complementary to the ahistoric and formal approach of the present lecture notes. First, it is instructive to follow the genesis of first-order logic~\cite{fo-history}, which is less straightforward than the canonical status of first-order logic today might suggests (even though the latter has good mathematical reasons). Secondly, there are strong connections between logic and the philosophy of mathematics~\cite{phil-math-stanford}. At the same time, it is good to keep in mind that not all philosophical questions about mathematics ask for a formal treatment (cf.~\cite{phil-math-practice}).

\subsection{Propositional Logic}\label{subsect:PL}

This section presents a framework called propositional logic, which allows to analyze how statements or propositions can be built from simpler ones in a certain way. Specifically, the symbols~$P_i$ of the following definition represent propositions that are not being analyzed further (in contrast to the situation in first-order logic that we will consider later). They are combined via so-called connectives $\neg,\land$ and~$\lor$, which are to be read as `not' (negation), `and'~(conjunction) and non-exclusive `or' (disjunction). Our choice of connectives is discussed below. 

\begin{definition}\label{def:PL-formula}
The formulas of propositional logic ($\mathsf{PL}$-formulas or just formulas) are the expressions that are recursively generated by the following clauses:
\begin{enumerate}[label=(\roman*)]
\item For each $i\in\mathbb N$, we have a $\mathsf{PL}$-formula $P_i$.
\item If $\varphi$ is a $\mathsf{PL}$-formula, then so is $\neg\varphi$.
\item If $\varphi$ and $\psi$ are $\mathsf{PL}$-formulas, then so are $(\varphi\land\psi)$ and $(\varphi\lor\psi)$.
\end{enumerate}
The expressions~$P_i$ are called propositional variables or atomic formulas.
\end{definition}

We will omit the outermost parentheses in $\varphi\land\psi$ and $\varphi\lor\psi$, which only become relevant when we build more complex formulas like $(\varphi\land\psi)\lor\rho$. Let us stress that negation binds stronger than the binary connectives, which means that $\neg\varphi\lor\psi$ is the disjunction of $\neg\varphi$ and $\psi$ while the negation of $\varphi\lor\psi$ is written as $\neg(\varphi\lor\psi)$. We now define some further connectives in terms of the given ones.

\begin{definition}
We use $\varphi\to\psi$ (implication) and $\varphi\leftrightarrow\psi$ (biconditional) as abbreviations for $\neg\varphi\lor\psi$ and for $(\varphi\to\psi)\land(\psi\to\varphi)$, respectively. Let us also agree to write $\bot$~(falsity) and $\top$ (truth) for $P_0\land\neg P_0$ and $P_0\lor\neg P_0$.
\end{definition}

The intuitive meaning of the connectives is officially fixed via the following definition, in which $0$ and $1$ represent the truth values `false' and `true'.

\begin{definition}
By a valuation we mean a function from $\mathbb N$ to $\{0,1\}$. Given a valu\-ation~$v$, we use recursion over the $\mathsf{PL}$-formula~$\varphi$ to define $\llbracket\varphi\rrbracket_v\in\{0,1\}$ by
\begin{align*}
\llbracket P_i\rrbracket_v&=v(i),\quad&\llbracket\varphi\land\psi\rrbracket_v&=\min(\llbracket\varphi\rrbracket_v,\llbracket\psi\rrbracket_v),\\
\llbracket\neg\varphi\rrbracket_v&=1-\llbracket\varphi\rrbracket_v,\quad&\llbracket\varphi\lor\psi\rrbracket_v&=\max(\llbracket\varphi\rrbracket_v,\llbracket\psi\rrbracket_v).
\end{align*}
We say $\varphi$ is (logically) valid or a tautology if $\llbracket\varphi\rrbracket_v=1$ holds for every valuation~$v$. If we have $\llbracket\varphi\rrbracket_v=1$ for some~$v$, then $\varphi$ is called satisfiable. When $\llbracket\psi\rrbracket_v=1$ holds for any~$v$ with $\llbracket\varphi\rrbracket_v=1$, we say that $\psi$ is a logical consequence of~$\varphi$ and write $\varphi\Rightarrow\psi$. Formulas $\varphi$ and $\psi$ are called (logically) equivalent if both $\varphi\Rightarrow\psi$ and $\psi\Rightarrow\varphi$ holds.
\end{definition}

We stress that $\varphi\to\psi$ is a $\mathsf{PL}$-formula while $\varphi\Rightarrow\psi$ is no formula but notation for a relation between two formulas. One often says that $\to$ is part of the object language (our object of investigation) while $\Rightarrow$ belongs to the metalanguage (in which the investigation takes place). Keeping these apart is widely seen as a fundamental competence. The following shows that this does not exclude strong~connections.

\begin{exercise}\label{ex:taut-sat-reducible}
Show that we have $\varphi\Rightarrow\psi$ precisely if $\varphi\to\psi$ is logically valid. Conversely, prove that $\varphi$ is valid precisely if we have $\top\Rightarrow\varphi$, which means that validity and logical consequence can be reduced to each other. Also give reductions between validity, satisfiability and equivalence.
\end{exercise}

Let us note that $\llbracket\varphi\rrbracket_v$ does only depend on the values of~$v$ on those~$i$ for which $P_i$ occurs in~$\varphi$. Validity, satisfiability, logical consequence and equivalence can thus be algorithmically decided via truth tables of the following form:

\begin{equation*}
\begin{array}{c|c||c|c||c}
P_0 & P_1 & P_0\to P_1 & (P_0\to P_1)\to P_0 & ((P_0\to P_1)\to P_0)\to P_0\\ \hline
0 & 0 & 1 & 0 & 1\\
0 & 1 & 1 & 0 & 1\\
1 & 0 & 0 & 1 & 1\\
1 & 1 & 1 & 1 & 1 
\end{array}
\end{equation*}
\vspace*{.5\baselineskip}

Here the first two columns provide a list of all possibilities for the relevant values of a valuation $v:\mathbb N\to\{0,1\}$. In the column below each formula $\varphi$, one finds the resulting truth-values $\llbracket\varphi\rrbracket_v$. The columns between the double lines can be seen as auxiliary computations, which may be omitted. The given table shows that~the formula in the last column is a tautology (which is called Peirce's law). At least intuitively, the truth table remains meaningful if we use letters like~$P$ and $Q$ at the place of the propositional variables~$P_0$ and~$P_1$. We will sometimes do so to improve readability. Also, we can replace $P_0$ and~$P_1$ by complex formulas~$\varphi$ and~$\psi$, to see how the truth values of the latter determine the one of a formula such as~$\varphi\to\psi$.

Some readers may find it peculiar that $\varphi\to\psi$ is considered to be true when both $\varphi$ and $\psi$ are false, which has instances like the statement `if~all integrable functions are continuous then $4$ is a prime number'. While such a statement may appear non-sensical, the given truth conditions for the implication conform with its regular usage in mathematical proofs (cf.~the previous exercise).

\begin{exercise}\label{ex:de-morgan}
Use truth tables to prove the equivalences
\begin{equation*}
\neg(\varphi\land\psi)\,\Leftrightarrow\,\neg\varphi\lor\neg\psi\quad\text{and}\quad\neg(\varphi\lor\psi)\,\Leftrightarrow\,\neg\varphi\land\neg\psi,
\end{equation*}
which are known as de Morgan's laws. Also show that double negations can be~eliminated in the sense that $\neg\neg\varphi$ is logically equivalent to~$\varphi$.
\end{exercise}

We have already seen that implication can be defined in terms of negation and disjunction. In order to justify our choice of connectives, we now show that this is no coincidence: any connective with an arbitrary number of arguments can be defined in terms of the given ones. The following makes this precise.

\begin{definition}\label{def:n-formula}
By an $n$-valuation we mean a map from~$\{0,\ldots,n-1\}$ to~$\{0,1\}$. An $n$-connective is a map~$F$ that assigns a value $F(v)\in\{0,1\}$ to each $n$-valuation~$v$. An $n$-formula is a $\mathsf{PL}$-formula that involves variables~$P_i$ for~$i<n$ only.
\end{definition}

By a slight abuse of notation, we also write $\llbracket\varphi\rrbracket_v$ when $v$ is just an $n$-valuation (rather than a valuation $\mathbb N\to\{0,1\}$), provided that $\varphi$ is an $n$-formula. To understand the given notion of connective, it may help to compare the truth table for the implication with the function $v\mapsto\llbracket P_0\to P_1\rrbracket_v$ on~$2$-valuations. The following shows that all possible connectives can be defined by one of our formulas.

\begin{theorem}[Expressive completeness of the propositional connectives]\label{thm:PL-completeness}
For any $n$-connective~$F$ with $n>0$, there is an $n$-formula~$\varphi$ such that $F(v)=\llbracket\varphi\rrbracket_v$ holds for any $n$-valuation~$v$.
\end{theorem}
\begin{proof}
We argue by induction on~$n$. In the base case of $n=1$, the only two valuations are given by $v_i:\{0\}\to\{0,1\}$ with $v_i(0)=i$ for~$i\leq 1$. There are four possibilities for the $1$-connective $F:\{v_0,v_1\}\to\{0,1\}$. These correspond to the \mbox{$1$-}formulas $\bot$ and $\top$ as well as $P_0$ and $\neg P_0$. For the induction step, we observe that each $(n+1)$-valuation~$v$ corresponds to the pair of its value $v(n)\in\{0,1\}$ and its restriction $v\!\restriction\!n$ to arguments~$k<n$. Given an $(n+1)$-connective~$F$, we consider the binary function~$G$ that is determined by $F(v)=G(v(n),v\!\restriction\!n)$. For each $i\in\{0,1\}$, the function $G(i,{-})$ is an $n$-connective, so that the induction hypothesis provides an $n$-formula $\varphi_i$ with $G(i,w)=\llbracket\varphi_i\rrbracket_w$. Let us now define $\varphi$ as the $(n+1)$-formula $(\neg P_n\land\varphi_0)\lor(P_n\land\varphi_1)$. To verify the claim from the theorem, we first consider an $(n+1)$-valuation $v$ with $v(n)=0$. We then have
\begin{equation*}
\llbracket\varphi\rrbracket_v=\llbracket\varphi_0\rrbracket_v=\llbracket\varphi_0\rrbracket_{v\restriction n}=G(0,v\!\restriction\!n)=F(v).
\end{equation*}
The case of $v(n)=1$ is analogous.
\end{proof}

In the presence of negation, one can define conjunction from disjunction and vice versa, due to de Morgan's laws. This may suggest that our choice of connectives is not as economical as possible. However, both conjunction and disjunction are~required for the normal forms that we present next. A more intuitive~explanation of the following definition will be given below.

\begin{definition}
Atomic formulas~$P_i$ and their negations~$\neg P_i$ are called literals. By a formula in negation normal form ($\mathsf{NNF}$-formula), we mean a $\mathsf{PL}$-formula that can be generated by the following clauses:
\begin{enumerate}[label=(\roman*)]
\item Each literal is an $\mathsf{NNF}$-formula.
\item If $\varphi_0$ and $\varphi_1$ are $\mathsf{NNF}$-formulas, then so is $\varphi_0\lor\varphi_1$.
\item If $\varphi_0$ and $\varphi_1$ are $\mathsf{NNF}$-formulas, then so is $\varphi_0\land\varphi_1$.
\end{enumerate}
An $\mathsf{NNF}$-formula generated without the use of~(iii) is called a disjunctive clause. The formulas in conjunction normal form ($\mathsf{CNF}$-formulas) are generated as follows:
\begin{enumerate}
\item Each disjunctive clause is a $\mathsf{CNF}$-formula.
\item If $\varphi_0$ and $\varphi_1$ are $\mathsf{CNF}$-formulas, then so is $\varphi_0\land\varphi_1$.
\end{enumerate}
\end{definition}

More intuitively, a formula is in negation normal form if negations do only occur directly in front of propositional variables. Given that each of $\lor$ and $\land$ is associative, we can represent $\mathsf{CNF}$-formulas by expressions of the form
\begin{equation*}
\textstyle\bigwedge_{i\leq k}\textstyle\bigvee_{j\leq l(i)} L_{i,j}
\end{equation*}
for literals~$L_{i,j}$ (where $\bigvee_{j\leq l(i)} L_{i,j}$ is the disjunctive clause $L_{i,0}\lor\ldots\lor L_{i,l(i)}$). There is also a dual notion of disjunction normal form, in which conjunction and disjunction are interchanged. Exercise~\ref{ex:cnf} below will suggest a proof of the following theorem that is more intuitive but seems harder to make precise.

\begin{theorem}\label{thm:cnf}
Each $\mathsf{PL}$-formula is equivalent to a $\mathsf{CNF}$-formula and in particular to a formula in negation normal form.
\end{theorem}
\begin{proof}
We first show that any formula~$\varphi$ is equivalent to an $\mathsf{NNF}$-formula~$\varphi^+$. Let us argue by induction over the build-up of $\varphi$ according to Definition~\ref{def:PL-formula}. To make the induction go through, we simultaneously show that $\neg\varphi$ is also equivalent to an $\mathsf{NNF}$-formula, which we denote by~$\varphi^-$. In the base case, we are concerned with a formula $\varphi$ of the form~$P_i$. Here we can take $\varphi^+$ and $\varphi^-$ to be~$\varphi$ and $\neg\varphi$, respectively. For the induction step, we distinguish three cases. First assume $\varphi$ has the form~$\neg\varphi_0$. The induction hypothesis ensures that $\varphi_0$ and~$\neg\varphi_0$ are equivalent to formulas $\varphi_0^+$ and $\varphi_0^-$ in negation normal form. Given that $\neg\varphi$ is equivalent to~$\varphi_0$, we can take $\varphi^+$ and $\varphi^-$ to be $\varphi_0^-$ and $\varphi_0^+$, respectively. In the second case, we have a formula~$\varphi$ that is of the form~$\varphi_0\land\varphi_1$. Let us note that $\neg\varphi$ is equivalent to $\neg\varphi_0\lor\neg\varphi_1$, by one of de Morgan's laws from Exercise~\ref{ex:de-morgan}. Using the formulas that are provided by the induction hypothesis, we can take $\varphi^+$ and~$\varphi^-$ to be $\varphi_0^+\land\varphi_1^+$ and $\varphi_0^-\lor\varphi_1^-$. By a dual argument, where $\land$ is interchanged with $\lor$, we can cover the final case of a formula~$\varphi$ that has the form $\varphi_0\lor\varphi_1$.

To establish the theorem, it is now enough to show that each $\mathsf{NNF}$-formula~$\varphi$ is equivalent to a $\mathsf{CNF}$-formula~$\varphi'$. We argue by induction over the build-up of~$\varphi$. In the base case, where $\varphi$ is a literal $P_i$ or~$\neg P_i$, we can take $\varphi'$ to be~$\varphi$. For the induction step, we do not need to consider the case where $\varphi$ has the form~$\neg\varphi_0$ (except when~$\varphi_0$ is $P_i$), as $\varphi$ is assumed to be in negation normal form. Concerning the two remaining cases, we first assume that~$\varphi$ has the form~$\varphi_0\land\varphi_1$. Using the formulas from the induction hypothesis, we can then take $\varphi'$ to be $\varphi_0'\land\varphi_1'$. Finally, assume that $\varphi$ has the form~$\varphi_0\lor\varphi_1$. If $\varphi_0'$ and $\varphi_1'$ are both disjunctive clauses, we can take~$\varphi'$ to be $\varphi_0'\lor\varphi_1'$, which is a disjunctive clause and hence a $\mathsf{CNF}$-formula. Without loss of generality, we now assume that $\varphi_0'$ is no disjunctive clause. Given that it is a $\mathsf{CNF}$-formula, it must then be of the form $\psi_0\land\psi_1$ for $\mathsf{CNF}$-formulas~$\psi_i$, so that we get
\begin{equation*}
\varphi\,\Leftrightarrow\,(\psi_0\land\psi_1)\lor\varphi_1'\,\Leftrightarrow\,(\psi_0\lor\varphi_1')\land(\psi_1\lor\varphi_1').
\end{equation*}
Here the second equivalence can be verified via a truth table. It remains to show that the conjuncts $\psi_i\lor\varphi_1'$ are $\mathsf{CNF}$-formulas. By repeating the previous step, we can inductively reduce to the case where $\psi_i$ and then also $\varphi_1'$ is a disjunctive clause.
\end{proof}

Let us comment on the method of induction over formulas, which is used in the previous proof. Some readers may prefer to view it as an induction over~$\mathbb N$ or, more specifically, over the number of symbols in a formula. Others may argue that both the $\mathsf{PL}$-formulas and the natural numbers are recursively generated (we have an element $0$ and another element $n+1$ for each element~$n$ that is already constructed), so that both come with their own induction principle, neither of which is more fundamental than the other.

As promised, the following exercise provides a more intuitive approach to conjunction normal forms. We note that it suggests a new proof not only for Theorem~\ref{thm:cnf} but also for the completeness result from Theorem~\ref{thm:PL-completeness} (since it yields a $\mathsf{CNF}$-formula that corresponds to an arbitrary truth table).

\begin{exercise}\label{ex:cnf}
Show how a $\mathsf{CNF}$-formula equivalent to $(P_0\to P_1)\to P_0$ can be read off from the truth table above.

\emph{Hint:} It may help to consider the truth table for $\neg((P_0\to P_1)\to P_0)$ as well. How do the lines with final entry~`$1$' relate to an equivalent formula in disjunction normal form? Note that the $\mathsf{CNF}$-formula that you read off from a truth table does sometimes contain redundant clauses, which one will usually omit.
\end{exercise}

In general, $\neg\varphi$ will not be an $\mathsf{NNF}$-formula when $\varphi$ is one. For later use, we~introduce the following operation on $\mathsf{NNF}$-formulas, which behaves like negation.

\begin{definition}\label{def:nega}
For each $\mathsf{NNF}$-formula $\varphi$ we define another $\mathsf{NNF}$-formula~$\nega\varphi$ by the recursive clauses
\begin{align*}
\nega P_i&=\neg P_i,\quad&\nega(\varphi\land\psi)&=(\nega\varphi)\lor(\nega\psi)\\
\nega(\neg P_i)&=P_i,\quad&\nega(\varphi\lor\psi)&=(\nega\varphi)\land(\nega\psi).
\end{align*}
\end{definition}

Let us stress that $\nega$ is not a new connective but notation for a defined operation on formulas with connectives $\neg,\land$ and $\lor$. For example, the expression $\nega(\neg P\lor Q)$ denotes the formula $P\land\neg Q$. By induction over $\mathsf{NNF}$-formulas, one can show that $\nega\nega\varphi$ and~$\varphi$ are the same (not just equivalent). In view of de Morgan's laws, another such induction yields the following.

\begin{lemma}\label{lem:tilde-negation}
For each $\mathsf{NNF}$-formula $\varphi$, the formula $\nega\varphi$ is equivalent to~$\neg\varphi$.
\end{lemma}

In Exercise~\ref{ex:taut-sat-reducible} we have seen that questions of validity, satisfiability and equivalence can be reduced to each other. While they are algorithmically decidable in principle, the approach via truth tables does quickly become infeasible in practice, since a formula built from $P_0,\ldots,P_{n-1}$ leads to a table with $2^n$ lines. The question whether some alternative algorithm can yield decisions in polynomial time is equivalent to the famous \mbox{$\mathsf P$-versus-$\mathsf{NP}$} problem, which is one of the seven Millennium Prize Problems that carry an award of 1\,000\,000\,\$\ by the Clay Mathematics Institute. It is investigated in a deep research area known as complexity theory, which does unfortunately go beyond the scope of the present introduction. Practical solutions are offered by so-called $\mathsf{SAT}$-solvers. We conclude this section with an exercise that treats a rather restrictive special case.

\begin{exercise}\label{ex:cnf-poly}
A $\mathsf{CNF}$-formula $\textstyle\bigwedge_{i\leq k}\textstyle\bigvee_{j\leq l(i)} L_{i,j}$ is called a 2-$\mathsf{CNF}$-formula if we have $l(i)\leq 1$ for each~$i\leq k$ (note that $j\leq 1$ allows for the two values $j\in\{0,1\}$). Show how to decide the satisfiability of 2-$\mathsf{CNF}$-formulas in polynomial time.

\emph{Hint:} When $P$ does not occur in~$\varphi$ and $\psi$, then $(P\lor\varphi)\land(\neg P\lor\psi)$ is satisfiable precisely if the same holds for~$\varphi\lor\psi$. This points to an algorithm that is known as `resolution'. Why is it polynomial for 2-$\mathsf{CNF}$-formulas but not in general? A formal definition of polynomial time algorithms can be given in terms of Turing machines. Here you can use the concept in an informal way.
\end{exercise}

\subsection{Complete Proof Systems}\label{subsect:complete-PL}

In this section, we present two systems of formal proofs, which were introduced by Gerhard Gentzen~\cite{gentzen-schliessen}. While we work in the setting of propositional logic, the aim is to lay a foundation that can be extended later. Our first proof system is called natural deduction and has the advantage that it is, arguably, fairly close to proofs as familiar from mathematical practice. Just like the $\mathsf{PL}$-formulas of the previous section, formal proofs belong to the object language and can thus be investigated by mathematical means. The second proof system, which is known as the sequent calculus, is particularly well suited for such investigations. We show that our proof systems have two desirable properties, which are known as soundness and completeness. All proof systems with these properties are equivalent and thus canonical in a certain~sense.

The following definition explains provability rather than proofs, although we will see how a notion of proof is defined implicitly. When $A$ and $B$ denote sets of formulas while $\varphi$ and $\varphi_i$ stand for single formulas, we write $A,\varphi$ for $A\cup\{\varphi\}$ and $A,B$ for $A\cup B$ as well as $\varphi_1,\ldots,\varphi_n$ for $\{\varphi_1,\ldots,\varphi_n\}$.

\begin{definition}\label{def:nat-deduct}
We recursively define a relation $A\vdash\varphi$ between a set $A$ of formulas and a formula~$\varphi$ (to be read as `$\varphi$ can be proved from the assumptions~$A$ in the system of natural deduction'):
\begin{enumerate}[label=(\roman*)]
\item For any $\varphi\in A$ we have $A\vdash\varphi$.
\item If we have $A\vdash\varphi\to\psi$ and $A\vdash\varphi$, we get $A\vdash\psi$. From $A,\varphi\vdash\psi$ we can infer $A\vdash\varphi\to\psi$.
\item Given $A\vdash\varphi$ and $A\vdash\neg\varphi$, we obtain $A\vdash\bot$. From $A,\varphi\vdash\bot$ we get~$A\vdash\neg\varphi$. If we have $A,\neg\varphi\vdash\bot$, we obtain $A\vdash\varphi$ (`proof by~contradiction').
\item Given $A\vdash\varphi$ and $A\vdash\psi$, we get $A\vdash\varphi\land\psi$. From $A\vdash\varphi\land\psi$ we can infer both $A\vdash\varphi$ and $A\vdash\psi$.
\item If we have $A\vdash\varphi$ or $A\vdash\psi$, we can conclude $A\vdash\varphi\lor\psi$. Given $A\vdash\varphi\lor\psi$ as well as $A,\varphi\vdash\theta$ and $A,\psi\vdash\theta$, we obtain $A\vdash\theta$.
\end{enumerate}
\end{definition}

Let us note that each connective comes with a rule that introduces and one that eliminates it. When given $A\vdash\varphi$ and $B\vdash\psi$, we cannot directly infer $A,B\vdash\varphi\land\psi$ by clause~(iv), as the latter requires~$A=B$. However, this is no real restriction: The following lemma will allow us to introduce additional assumptions and weaken the premises to $A,B\vdash\varphi$ and $A,B\vdash\psi$, so that we can then infer $A,B\vdash\varphi\land\psi$ after all. It also follows that $A\vdash\bot$ entails $A\vdash\varphi$ for any~$\varphi$ (`ex falso quodlibet'), as we can weaken the premise to $A,\neg\varphi\vdash\bot$ and then conclude by contradiction. As preparation for the following proof, we note that it is allowed to have $\varphi\in A$ in the premise $A,\varphi\vdash\psi$ of clause~(ii) and in similar places. In fact, when given the premise~$A,\varphi\vdash\psi$, we can always replace $A$ by $A':=A\cup\{\varphi\}$, since the sets $A,\varphi$ and $A',\varphi$ will then coincide. If we now apply clause~(ii) with $A'$ instead of~$A$, this permits us to conclude~$A,\varphi\vdash\varphi\to\psi$ rather than $A\vdash\varphi\to\psi$.

\begin{lemma}[Weakening]\label{lem:weakening}
Given $A\vdash \varphi$ and $A\subseteq B$, we get $B\vdash \varphi$.
\end{lemma}
\begin{proof}
We argue by induction over the way in which $A\vdash\varphi$ is obtained according to the previous definition (or over the number of times that clauses~(i) to~(v) were applied). In the case of clause~(i), we have $A\vdash\varphi$ due to~$\varphi\in A$. Given $A\subseteq B$, we get $\varphi\in B$ and hence $B\vdash\varphi$ by the same clause~(i). Next, we consider the case where $\varphi$ has the form $\varphi_0\to\varphi_1$ and clause~(ii) was applied with premise $A,\varphi_0\vdash\varphi_1$. As $A\subseteq B$ entails $A\cup\{\varphi_0\}\subseteq B\cup\{\varphi_0\}$, the induction hypothesis yields $B,\varphi_0\vdash\varphi_1$, so that we get $B\vdash\varphi$ by the same clause~(ii). The other cases are similar.
\end{proof}

More intuitively, proofs in natural deduction are usually written in the following form, which is explained below:

\begin{scprooftree}{.8}
\AxiomC{$u_0:\varphi\land \psi $}
\RightLabel{(iv)}
\UnaryInfC{$\varphi$}
\AxiomC{$u_1:\neg \varphi$}
\RightLabel{(iii)}
\BinaryInfC{$\bot$}
\AxiomC{$u_0:\varphi\land \psi $}
\RightLabel{(iv)}
\UnaryInfC{$\psi $}
\AxiomC{$u_2:\neg \psi $}
\RightLabel{(iii)}
\BinaryInfC{$\bot$}
\AxiomC{$u_3:\neg \varphi\lor\neg \psi $}
\RightLabel{(v); $u_1,u_2$}
\TrinaryInfC{$\bot$}
\RightLabel{(iii); $u_0$}
\UnaryInfC{$\neg(\varphi\land \psi )$}
\RightLabel{(ii); $u_3$}
\UnaryInfC{$\neg \varphi\lor\neg \psi \to\neg(\varphi\land \psi )$}
\end{scprooftree}
\vspace*{.5\baselineskip}

\noindent Such a `proof tree' is to be read from the top down (from the `leaves' to the `root'). The leaf marked $u_0:\varphi\land \psi $ should be understood as an assumption: it represents a proof of $\varphi\land \psi \vdash \varphi\land \psi $ according to clause~(i) of Definition~\ref{def:nat-deduct}. The purpose of the label~$u_0$ will become clear later. Each horizontal line corresponds to an inference, i.\,e., to an application of clauses~(ii) to~(v). For example, the inference at the top left uses clause~(iv) to transform $\varphi\land \psi \vdash \varphi\land \psi $ into $\varphi\land \psi \vdash \varphi$. In the given proof tree, this is now combined with a leaf~$u_1:\neg \varphi$ that corresponds to $\neg \varphi\vdash\neg \varphi$. Thanks to weakening, we may use clause~(iii) to combine $\varphi\land \psi \vdash \varphi$ and $\neg \varphi\vdash\neg \varphi $ into $\varphi\land \psi ,\neg \varphi\vdash\bot$. Similarly, the second premise of the ternary inference in our proof tree corresponds to $\varphi\land \psi ,\neg \psi \vdash\bot$. The third premise is an assumption that amounts to~$\neg \varphi\lor\neg \psi \vdash\neg \varphi\lor\neg \psi $. Now the ternary inference uses the second part of clause~(v) to combine its premises into $\varphi\land \psi ,\neg \varphi\land\neg \psi \vdash\bot$. Here the assumptions $\neg \varphi$ and $\neg \psi $ from the first and second premise are dropped in favour of the disjunction~$\neg \varphi\lor\neg \psi $, in accordance with Definition~\ref{def:nat-deduct}. The labels $u_1$ and~$u_2$ next to the inference line are there to indicate that the corresponding leaves do no longer yield open assumptions. When writing the proof, it may also help to tick off the leaves at this point (traditionally by adding square parentheses as in $[u_0:\varphi\land \psi ]$). The remaining inferences correspond to clauses~(iii) and~(ii) and discharge the assumptions~$u_0$ and~$u_3$. Overall, the proof tree shows $\emptyset\vdash\neg \varphi\lor\neg \psi \to\neg(\varphi\land \psi )$.

To summarize, a proof tree with root $\varphi$ and open assumptions~$A$ is a certificate that we get $A\vdash\varphi$ in a certain way. Note that there may be different proofs of the same statement. At least intuitively, it should be clear that one can mechanically check whether a proof tree is correct according to Definition~\ref{def:nat-deduct}. On the other hand, it is not obvious if there is an algorithm that decides whether a given statement $A\vdash\varphi$ holds, i.\,e., whether $\varphi$ is provable from~$A$. One can show that such an algorithm exists for propositional logic (as we will see) but not for the more expressive framework of predicate logic, which is presented in the next section.

Upon reflection, the proof tree above is not too far from an informal proof of the same statement. Indeed, to establish $\neg \varphi\lor\neg \psi \to\neg(\varphi\land \psi )$, one would certainly assume $\neg \varphi\lor\neg \psi $ to derive $\neg(\varphi\land \psi )$, which corresponds to the final inference. Also, it seems natural to prove $\neg(\varphi\land \psi )$ by showing that $\varphi\land \psi $ leads to a contradiction (the penultimate inference) and to use $\neg \varphi\lor\neg \psi $ by considering what each of $\neg \varphi$ and~$\neg \psi $ entails (the ternary inference). Nevertheless, it requires practice to see how such informal considerations translate into a formal proof.

\begin{exercise}\label{ex:nat-deduct}
(a) The proof tree above yields an implication that corresponds to one part of the de Morgan laws from Exercise~\ref{ex:de-morgan}. Give proof trees for the other three implications that make up these laws. Also show that we have $\emptyset\vdash\neg\neg\varphi\leftrightarrow\varphi$.

(b) Show that each $\mathsf{PL}$-formula~$\varphi$ admits an $\mathsf{NNF}$-formula $\varphi^+$ with $\emptyset\vdash\varphi\leftrightarrow\varphi^+$. \emph{Hint:} What can you keep from the proof of Theorem~\ref{thm:cnf} and what do you need to adapt? If you are happy to put in some extra work, you can also adapt the entire theorem and prove an equivalence with a $\mathsf{CNF}$-formula.
\end{exercise}

As our second proof system, we present a version of the so-called sequent calculus. By a (Tait-style~\cite{tait-style}) sequent we mean a finite set of formulas in negation normal~form. The sequent $\Gamma=\{\varphi_1,\ldots,\varphi_n\}$ represents the disjunction $\bigvee\Gamma=\varphi_1\lor\ldots\lor\varphi_n$, so that $\Gamma$ is true precisely if one of the $\varphi_i$ is true. In particular, the empty sequent represents falsity. The following definition uses the same symbol~$\vdash$~(turnstile) as in the context of natural deduction. There is no danger of confusion, as the turnstile in natural deduction has a set of formulas (possibly~$\emptyset$) to its left, while the one in our sequent calculus does not.

\begin{definition}\label{def:sequent-calc-PL}
Recall that a sequent is a finite set of $\mathsf{NNF}$-formulas. We recur\-sively define $\vdash\Gamma$ for a sequent~$\Gamma$ (read `$\Gamma$ can be derived in sequent calculus'):
\begin{description}[labelwidth=5.5ex,labelindent=\parindent,leftmargin=!,before={\renewcommand\makelabel[1]{(##1)}}]
\item[$\mathsf{Ax}$] If $\Gamma$ contains some atomic formula and its negation, we have $\vdash\Gamma$.
\item[$\land$] Given $\vdash\Gamma,\varphi$ and $\vdash\Gamma,\psi$, we get $\vdash\Gamma,\varphi\land\psi$. 
\item[$\lor$] From $\vdash\Gamma,\psi$ we get both $\vdash\Gamma,\psi\lor\theta$ and $\vdash\Gamma,\varphi\lor\psi$.
\item[$\mathsf{Cut}$] If we have both $\vdash\Gamma,\varphi$ and $\vdash\Gamma,\nega\varphi$ for some $\mathsf{NNF}$-formula~$\varphi$ (cf.~Definition~\ref{def:nega}), then we get $\vdash\Gamma$.
\end{description}
The last clause is known as the cut rule. If $\vdash\Gamma$ can be derived without~($\mathsf{Cut}$), then we write $\vdash_0\Gamma$ and say that $\Gamma$ has a cut-free proof in sequent calculus.
\end{definition}

Let us stress that, for example, $\varphi\lor\psi\in\Gamma$ is permitted in clause~($\lor$), so that the latter allows to deduce $\vdash\Delta,\varphi\lor\psi$ from the premise~$\vdash\Delta,\varphi\lor\psi,\psi$ (take $\Gamma=\Delta,\varphi\lor\psi$ and observe $\Gamma,\varphi\lor\psi=\Gamma$). Intuitively, the following holds because a disjunction gets weaker when we add more disjuncts. For a formal proof one argues by induction.

\begin{lemma}[Weakening]
Consider sequents $\Gamma\subseteq\Delta$. If we have $\vdash\Gamma$ or~$\vdash_0\Gamma$, we get $\vdash\Delta$ or~$\vdash_0\Delta$, respectively.
\end{lemma}

In particular, when given premises $\vdash\Delta,\varphi$ and $\vdash\Gamma,\psi$, we can use weakening to get $\vdash\Delta,\Gamma,\varphi$ and $\vdash\Delta,\Gamma,\psi$, which may then be combined into $\vdash\Delta,\Gamma,\varphi\land\psi$ according to clause~($\land$). As in the case of natural deduction, proofs in sequent calculus~will often be written in the following form (note that the final formula below is equivalent to Peirce's law from the previous section):

\begin{prooftree}
\AxiomC{$\neg P,P$}
\RightLabel{($\lor$)}
\UnaryInfC{$\neg P\lor Q,P$}
\AxiomC{$\neg P,P$}
\RightLabel{($\land$)}
\BinaryInfC{$(\neg P\lor Q)\land\neg P,P$}
\RightLabel{($\lor$)}
\UnaryInfC{$((\neg P\lor Q)\land\neg P)\lor P,P$}
\RightLabel{($\lor$)}
\UnaryInfC{$((\neg P\lor Q)\land\neg P)\lor P$}
\end{prooftree}
\vspace*{.5\baselineskip}

\noindent It is worth noting that two applications of clause~($\lor$) are needed to combine the disjuncts at the end of the previous proof.

\begin{exercise}\label{ex:nega}
Prove that we have $\varphi,\nega\varphi\vdash\bot$ in natural deducation as well as $\vdash\nega\varphi,\varphi$ in sequence calculus, for any $\mathsf{NNF}$-formula~$\varphi$. \emph{Hint:} Use induction over the build-up of the formula~$\varphi$.
\end{exercise}

The following provides a first connection between sequent calculus and natural deduction. We will later strengthen the result in two respects: First, the assumption that the proof in sequent calculus is cut-free is not actually needed, even though it makes things somewhat easier. Secondly, the implication is actually an equivalence, i.\,e., a proof in natural deduction does also yield one in sequent calculus.

\begin{proposition}\label{prop:seq-to-nat-deduct}
If there is a cut-free proof $\vdash_0\varphi$ in sequent calculus, there is also a proof $\emptyset\vdash\varphi$ in natural deduction.
\end{proposition}
\begin{proof}
The idea is to use induction over the derivation of $\vdash_0\varphi$ according to the clauses from Definition~\ref{def:sequent-calc-PL}. However, the induction does not go through directly: even though the sequent $\varphi$ consists of a single formula, its proof may involve sequents with more elements, while only a single formula is allowed to the right of the turnstile in natural deduction. To circumvent this issue, let us agree to write $\Gamma^-$ for $\neg\varphi_1,\ldots,\neg\varphi_n$ when $\Gamma$ is the sequent $\varphi_1,\ldots,\varphi_n$. We use induction to show that $\vdash_0\Gamma$ entails $\Gamma^-\vdash\bot$. For the base case, we must establish $\neg P,\neg\neg P\vdash\bot$, which is straightforward. In the induction step, we consider the case where~$\vdash_0\Gamma$ was derived according to clause~($\lor$). More specifically, we assume that $\Gamma$ has the form $\Delta,\psi\lor\theta$ and was derived from $\vdash_0\Delta,\psi$. By the induction hypothesis, we get $\Delta^-,\neg\psi\vdash\bot$. This entails $\Delta^-\vdash\psi$ via proof by contradiction. We can conclude $\Delta^-\vdash\psi\lor\theta$, which readily yields $\Delta^-,\neg(\psi\lor\theta)\vdash\bot$, as required. The case where $\vdash_0\Gamma$ was derived by clause~($\land$) is treated in a similar way. Once the induction is completed, we know that $\vdash_0\varphi$ implies~$\neg\varphi\vdash\bot$. To get $\emptyset\vdash\varphi$, we again use proof by contradiction.
\end{proof}

We now come to the main result of this section, which shows that a formula is provable (in natural deduction or sequent calculus) precisely if it is logically valid. Let us begin with the easier direction.

\begin{proposition}[Soundness for propositional logic]\label{prop:soundness}\mbox{ }

(a) Assume we have $\Gamma\vdash\varphi$ in natural deduction. For any valuation~$v$ such that we have $\llbracket\psi\rrbracket_v=1$ for all~$\psi\in\Gamma$, we also have $\llbracket\varphi\rrbracket_v=1$.

(b) Assume we have $\vdash\Gamma$ in sequent calculus. For any valuation~$v$, the sequent~$\Gamma$ contains some formula $\varphi$ for which we have $\llbracket\varphi\rrbracket_v=1$.
\end{proposition}
\begin{proof}
Both parts of the proposition are established by induction over proofs. Since all cases are rather similar, we only consider a single one: assume that the proof of $\vdash\Gamma$ in~(b) was concluded by a cut rule with premises $\vdash\Gamma,\varphi$ and $\vdash\Gamma,\nega\varphi$. Given any valuation~$v$, the induction hypothesis tells us, first, that we have $\llbracket\psi\rrbracket_v=1$ for some~$\psi\in\Gamma\cup\{\varphi\}$. If we have $\psi\in\Gamma$, then we are done, so let us assume~$\llbracket\varphi\rrbracket_v=1$. From Lemma~\ref{lem:tilde-negation} we know that $\nega\varphi$ is equivalent to~$\neg\varphi$. Hence we get ~$\llbracket\nega\varphi\rrbracket_v=0$. But for the second premise of our cut, the induction hypotheses ensures that we have $\llbracket\psi\rrbracket_v=1$ for some~$\psi\in\Gamma\cup\{\nega\varphi\}$. We can only have $\psi\in\Gamma$, as required.
\end{proof}

To prove a result that is converse to the previous proposition, we need some preparations. Let us write $2^{<\omega}$ for the collection of finite sequences with~entries from~$\{0,1\}$. To refer to the length and entries of such sequences, we shall write them as $\sigma=\langle\sigma_0,\ldots,\sigma_{l(\sigma)-1}\rangle$. For $\sigma,\tau\in 2^{<\omega}$ we write $\sigma\sqsubset\tau$ to express that $\sigma$ is a proper initial segment of~$\tau$, i.\,e., that we have $l(\sigma)<l(\tau)$ and $\sigma_i=\tau_i$ for~$i<l(\sigma)$. For $\sigma\in 2^{<\omega}$ and $i<2$, we put $\sigma\star i:=\langle\sigma_0,\ldots,\sigma_{l(\sigma)-1},i\rangle\in 2^{<\omega}$. Also, we write~$2^\omega$ for the set of functions from~$\mathbb N$ to~$\{0,1\}$ (infinite sequences). Given $f\in 2^\omega$ and a number $n\in\mathbb N$, we put $f[n]:=\langle f(0),\ldots,f(n-1)\rangle\in 2^{<\omega}$. For $\sigma\in 2^{<\omega}$ and $f\in 2^\omega$ we write $\sigma\sqsubset f$ to express~$\sigma=f[l(\sigma)]$.

\begin{definition}\label{def:binary-tree}
A binary tree is a nonempty set $T\subseteq 2^{<\omega}$ such that $\sigma\sqsubset\tau\in T$ entails~$\sigma\in T$. We say that $f\in 2^\omega$ is a branch of~$T$ if we have $f[n]\in T$ for all~$n\in\mathbb N$. By a leaf of~$T$ we mean an element $\sigma\in T$ such that $\sigma\star i\notin T$ holds for both~$i<2$.
\end{definition}

In Section~\ref{subsect:KL} we will see that the following theorem is strictly weaker than the analogous result for trees with arbitrary finite branchings. Here we consider the weak version in order to show that it suffices for the purpose at hand.

\begin{theorem}[Weak K\H{o}nig's Lemma]\label{thm:wkl}
Any infinite binary tree has a branch.
\end{theorem}
\begin{proof}
In an infinite binary tree $T$, we can iteratively select $\sigma^0\sqsubset\sigma^1\sqsubset\ldots$ with $l(\sigma^i)=i$ such that there are infinitely many $\tau\in T$ with~$\sigma^i\sqsubset\tau$. Here the recursion step is secured by the fact that the union of two finite sets is finite, as $\sigma^i\sqsubset\tau$ yields either $\sigma^i\star 0\sqsubseteq\tau$ or $\sigma^i\star 1\sqsubseteq\tau$. To conclude, we define $f\in 2^\omega$ by $f(i):=\sigma^{i+1}_i$, which does indeed yield $f[n]=\sigma^n\in T$ for all~$n\in\mathbb N$.
\end{proof}

We will not need the result of the following exercise (or any prerequisites from topology), but it provides some explanation for the name of Theorem~\ref{thm:compactness} below. 

\begin{exercise}\label{ex:compactness}
Use K\H{on}ig's lemma to show that $2^\omega$ is compact when considered as a topological space with basic open sets~$\mathcal O_\sigma:=\{f\in 2^\omega\,|\,\sigma\sqsubset f\}$ for~$\sigma\in 2^{<\omega}$.
\end{exercise}

The following proposition is a version of the so-called completeness theorem, but we reserve this name for a stronger result that will be derived below.

\begin{proposition}\label{prop:completeness}
Consider a sequent $\Gamma=\{\varphi_1,\ldots,\varphi_n\}$ of propositional logic. If $\bigvee\Gamma=\varphi_1\lor\ldots\lor\varphi_n$ is valid, there is a cut-free proof $\vdash_0\Gamma$ in sequent calculus.
\end{proposition}
\begin{proof}
The idea is to search for a proof tree with root~$\Gamma$ by applying the rules of sequent calculus backwards in all possible ways. This will either result in a finite proof or in an infinite tree, which must thus have a branch. From such a branch we will be able to read off a valuation that witnesses that~$\bigvee\Gamma$ is not logically valid. In the following we make this idea precise.

To define a binary tree $T$ , we will decide $\sigma\in T$ by recursion over the length of the sequence~$\sigma\in 2^{<\omega}$. Simultaneously, we shall specify a sequent~$\Gamma(\sigma)$ for each~$\sigma\in T$. In fact, we will define $\Gamma(\sigma)$ as a finite sequence (rather than a set) of formulas in negation normal form, as an order on~$\Gamma(\sigma)$ is needed for our construction. When we wish to view~$\Gamma(\sigma)$ as a sequent in the sense of Definition~\ref{def:sequent-calc-PL}, we interpret it as the set of its entries. Extending our previous notation, we write $\Gamma,\varphi$ for~$\Gamma\star\varphi$ in the context of ordered sequents.

 In the base case of the aforementioned recursion, we declare that the empty sequence~$\langle\rangle$ is an element of~$T$ and that $\Gamma(\langle\rangle)$ is the sequent~$\Gamma$ that is given in the proposition (with an arbitrary order on its elements). In the recursion step, we assume that we already know~$\sigma\in T$ and that $\Gamma(\sigma)$ has been defined (note that $\sigma\notin T$ forces $\sigma\star i\notin T$ since~$T$ is a tree). We may assume that $\Gamma$ is nonempty, since the empty disjunction is false rather than valid. A glance at the following~construction shows that this property is preserved, i.\,e., that $\Gamma(\sigma)$ is always nonempty. Write
\begin{equation*}
\Gamma(\sigma)=\langle\varphi,\varphi_1,\ldots,\varphi_n\rangle\quad\text{and}\quad\Delta=\langle\varphi_1,\ldots,\varphi_n\rangle.
\end{equation*}
If $\Gamma(\sigma)$ contains some atomic formula and the negation of the same formula, we stipulate that $\sigma$ is a leaf of~$T$. Otherwise, we~declare $\sigma\star i\in T$ for both~$i<2$ (even though we could confine ourselves to $i=0$ in all but the second of the following cases). We then put
\begin{equation*}
\Gamma(\sigma\star i):=\begin{cases}
\Delta,\varphi & \text{if $\varphi$ is a negated or unnegated atomic formula},\\[.5ex]
\Delta,\varphi,\psi_i & \text{if $\varphi$ is the conjunctive formula~$\psi_0\land\psi_1$},\\[.5ex]
\Delta,\varphi,\psi_0 & \parbox[t]{.6\textwidth}{if $\varphi$ is the disjunctive formula~$\psi_0\lor\psi_1$ and $\psi_0$ does not occur in~$\Gamma(\sigma)$,}\\[3.2ex]
\Delta,\varphi,\psi_1 & \parbox[t]{.6\textwidth}{if $\varphi$ is the disjunctive formula~$\psi_0\lor\psi_1$ and $\psi_0$ does occur in~$\Gamma(\sigma)$.}
\end{cases}
\end{equation*}
Note that $\Gamma(\sigma)$ is merely permuted in the first case. In the last two cases, we could have defined $\Gamma(\sigma\star i)$ to be $\Delta,{\psi_0},\psi_1$, but the given definition is more instructive in the context of a generalization that we present later (think of $\exists x\,\varphi(x)$ as an infinite disjunction $\varphi(t_0)\lor\varphi(t_1)\lor\ldots$ in the proof of Proposition~\ref{prop:completeness-FO}).

Due to K\H{o}nig's lemma, the tree~$T$ will either be finite or have a branch. Let us first assume that it is finite. Then $T$ with the attached sequents~$\Gamma(\sigma)$ is essentially a proof in the sense of Definition~\ref{def:sequent-calc-PL}. To make this more precise, we consider the height function~$h:T\to\mathbb N$ that is given by $h(\sigma)=n$ when $n$ is the maximal length of a chain~$\sigma=\sigma_0\sqsubset\ldots\sqsubset\sigma_n$ in~$T$. We use induction over~$h(\sigma)$ to show~$\vdash_0\Gamma(\sigma)$. Note that we have $h(\sigma)=0$ precisely when $\sigma$ is a leaf of~$T$. By construction, this is the case precisely if $\Gamma(\sigma)$ contains an atomic formula and its negation. We then have $\vdash_0\Gamma(\sigma)$ according to clause~($\mathsf{Ax}$) of Definition~\ref{def:sequent-calc-PL}. In the remaining case, we write $\varphi$ for the first entry and $\Delta$ for the rest of~$\Gamma(\sigma)$, as above. First assume that $\varphi$ is a negated or unnegated atomic formula. Let us note that we have~$h(\sigma\star i)<h(\sigma)$, as any chain starting with $\sigma\star i$ can be extended into a longer one starting with~$\sigma$. The induction hypothesis will thus provide $\vdash_0\Delta,\varphi$. To complete the induction step in this case, we need only observe that $\Delta,\varphi$ and~$\Gamma(\sigma)$ are equal as sequents (sets rather than sequences). Let us now assume that $\varphi$ is a conjunction~$\psi_0\land\psi_1$. As in the previous case, we inductively obtain $\vdash_0\Gamma(\sigma),\psi_i$ for both~$i<2$. Due to clause~($\land$) of Definition~\ref{def:sequent-calc-PL}, we can conclude $\vdash_0\Gamma(\sigma),\varphi$. Since $\varphi$ occurs in~$\Gamma(\sigma)$, the sequents $\Gamma(\sigma),\varphi$ and $\Gamma(\sigma)$ coincide (as sets). So we get~$\vdash_0\Gamma(\sigma)$, as needed for the induction step. In the case where $\varphi$ is disjunctive, one argues similarly. Now that the induction is completed, we can apply the result to~$\sigma=\langle\rangle$. Since $\Gamma(\langle\rangle)$ was defined to be~$\Gamma$, we obtain $\vdash_0\Gamma$, which is the conclusion of the proposition.

To complete the proof, we consider the case where $T$ has a branch~$f:\mathbb N\to\{0,1\}$. Let $\mathcal F$ be the collection of all formulas that occur in~$\Gamma(f[n])$ for some~$n\in\mathbb N$. We show the following crucial facts:
\begin{enumerate}[label=(\roman*)]
\item For any atomic formula~$\theta\in\mathcal F$, we have $\neg\theta\notin\mathcal F$.
\item If $\psi_0\lor\psi_1$ lies in $\mathcal F$, so do both $\psi_0$ and $\psi_1$.
\item If $\psi_0\land\psi_1$ lies in~$\mathcal F$, then so does~$\psi_0$ or~$\psi_1$.
\end{enumerate}
To establish~(i) by contradiction, we assume that $\mathcal F$ contains both $\theta$ and $\neg\theta$. Consider $m,n\in\mathbb N$ such that $\theta$ occurs in $\Gamma(f[m])$ while~$\neg\theta$ lies in~$\Gamma(f[n])$. Without loss of generality, we may assume~$m\leq n$. By construction, formulas in $\Gamma(\sigma)$ are always contained in~$\Gamma(\sigma\star i)$ as well. It follows that $\theta$ does also lie in~$\Gamma(f[n])$, which thus contains an atomic formula and its negation. Again by construction, we can conclude that $f[n]\star i\notin T$ holds for both~$i<2$. In view of $f[n]\star f(n)=f[n+1]$, this contradicts the assumption that~$f$ is a branch. In order to establish~(ii), we assume that $\psi_0\lor\psi_1$ lies in~$\Gamma(f[k])$. Whenever we pass from $\Gamma(\sigma)$ to $\Gamma(\sigma\star i)$ in the construction above, each formula (except for the first) moves one position to the front. Hence $\varphi$ will be the first formula of~$\Gamma(f[m])$ for some~$m\geq k$. It now follows that $\Gamma(f[m+1])$ contains both $\varphi$ and~$\psi_0$. As before, we find $n>m$ such that $\varphi$ is the first formula in~$\Gamma(f[n])$. As the latter will still contain~$\psi_0$, the construction ensures that~$\Gamma(f[n+1])$ contains~$\psi_1$. The proof of~(iii) is similar but easier.

Now consider the valuation~$v$ that satisfies $v(i)=1$ precisely if we have~$P_i\notin\mathcal F$. We use induction over ~$\mathsf{NNF}$-formulas to show that $\varphi\in\mathcal F$ entails~$\llbracket\varphi\rrbracket_v=0$. This is immediate when~$\varphi$ has the form~$P_i$. In the case where~$\varphi$ has the form~$\neg P_i$, the claim holds since we get $P_i\notin\mathcal F$ due to~(i) above. Now assume that $\varphi$ is a disjunction~$\psi_0\lor\psi_1$. For each~$i<2$ we get $\psi_i\in\mathcal F$ by~(ii) above, so that the induction hypothesis yields~$\llbracket\psi_i\rrbracket_v=0$. We can thus conclude $\llbracket\varphi\rrbracket_v=\max_{i<2}\llbracket\psi_i\rrbracket_v=0$. When~$\varphi$ is conjunctive, one argues similarly. Finally, note that~$f[0]$ is the empty sequence, so that $\Gamma(f[0])$ coincides with~$\Gamma$. For any $\varphi\in\Gamma$ we thus have~$\llbracket\varphi\rrbracket_v=0$. So $\bigvee\Gamma$ is not logically valid, against the assumption of the proposition.
\end{proof}

The following exercise shows that we can extract additional information from the fact that we get a proof that is cut-free. More intricate information of this type will be obtained in Section~\ref{sect:cut-free-completeness} below.

\begin{exercise}
Show that any valid $n$-formula (cf.~Definition~\ref{def:n-formula}) has a sequent calculus proof in which only $n$-formulas occur. \emph{Hint:} Use induction over cut-free proofs, i.\,e., over the number of times a clause from Definition~\ref{def:sequent-calc-PL} has been applied.
\end{exercise}

In view of the following result, the reader may wonder why the cut rule has been introduced at all. A first answer can be found in the paragraph after Corollary~\ref{cor:sequent-nat-deduct} and in the proof of Theorem~\ref{thm:completeness-prop}. In Section~\ref{sect:cut-free-completeness}, we will see further evidence for the importance of the cut rule. We note that the word `semantic' refers to concepts such as satisfaction in a model, which are concerned with the \emph{meaning} of linguistic expressions. In contrast, `syntactic' refers to methods that investigate the structure of proofs and other linguistic objects as sequences of symbols, without taking their meaning into account (e.\,g., the proof of Theorem~\ref{thm:herbrand} can be called syntactic).

\begin{corollary}[Semantic Cut-Elimination]\label{cor:cut-elim-PL}
From $\vdash\Gamma$ we can infer $\vdash_0\Gamma$.
\end{corollary}
\begin{proof}
Given the premise, we can invoke Proposition~\ref{prop:soundness} to learn that $\bigvee\Gamma$ is valid. By the previous theorem, this yields the conclusion.
\end{proof}

Let us point out that cut-elimination can also be proved by syntactic methods, which provide an explicit procedure that transforms a given proof with cuts into one that is cut-free (see the paragraph before Exercise~\ref{ex:Herbrand-supexp}). Using the proposition above, we can also complete our proof that natural deduction and sequent calculus are equivalent proof systems.

\begin{corollary}\label{cor:sequent-nat-deduct}
For any formulas $\psi_1,\ldots,\psi_n$ and~$\varphi$ in negation normal form, the following are equivalent:
\begin{enumerate}[label=(\roman*)]
\item We have $\psi_1,\ldots,\psi_n\vdash\varphi$ in natural deducation.
\item We have $\vdash\nega\psi_1,\ldots,\nega\psi_n,\varphi$ in sequent calculus.
\end{enumerate}
\end{corollary}
\begin{proof}
Given~(i), soundness (Proposition~\ref{prop:soundness}) ensures that $\nega\psi_1\lor\ldots\lor\nega\psi_n\lor\varphi$ is logically valid (see~also Lemma~\ref{lem:tilde-negation}). We then obtain~(ii) by Proposition~\ref{prop:completeness}. Given~(ii), we invoke the previous corollary to get $\vdash_0\nega\psi_1,\ldots,\nega\psi_n,\varphi$ without cuts. By the proof of  Proposition~\ref{prop:seq-to-nat-deduct}, we obtain
\begin{equation*}
\neg\nega\psi_1,\ldots,\neg\nega\psi_n,\neg\varphi\vdash\bot
\end{equation*}
in natural deduction. Invoking proof by contradiction, we get
\begin{equation*}
\neg\nega\psi_2,\ldots,\neg\nega\psi_n,\neg\varphi\vdash\nega\psi_1.
\end{equation*}
Due to Exercise~\ref{ex:nega}, we also have $\psi_1\vdash\nega\psi_1\to\bot$. So we can conclude
\begin{equation*}
\psi_1,\neg\nega\psi_2,\ldots,\neg\nega\psi_n,\neg\varphi\vdash\bot.
\end{equation*}
By repeating the argument, we obtain $\psi_1,\ldots,\psi_n,\neg\varphi\vdash\bot$. To reach~(i), we apply proof by contradiction once again.
\end{proof}

The proof that~(i) implies~(ii) shows that all sound and complete proof systems are equivalent. At the same time, it provides no efficient procedure that would transform a given proof in natural deduction into a proof in sequent calculus. An obvious way to define such a procedure is by recursion over proofs. It is worth considering the case where $\emptyset\vdash\psi$ has been derived from $\emptyset\vdash\varphi\to\psi$ and~$\emptyset\vdash\varphi$. Given that we treat $\varphi\to\psi$ as an abbreviation of~$\neg\varphi\lor\psi$, it makes sense to assume that we get~$\vdash\nega\varphi\lor\psi$ and $\vdash\varphi$ in sequent calculus. To define a recursive procedure, we should show how to derive~$\vdash\psi$. Note that $\nega(\nega\varphi\lor\psi)$ coincides with $\varphi\land\nega\psi$, since $\nega\nega\varphi$ is the same as~$\varphi$ (see the paragraph after Definition~\ref{def:nega}). To conclude by an application of the cut rule, it is enough to have~$\vdash\varphi\land\nega\psi,\psi$. The latter can be obtained from $\vdash\varphi$ as given and $\vdash\nega\psi,\psi$ from Exercise~\ref{ex:nega}. The point is that we seem to need the cut rule in order to define an efficient procedure that witnesses the implication from~(i) to~(ii) in the previous corollary. Of course, Corollary~\ref{cor:cut-elim-PL} entails that the implication remains valid when we demand a cut-free proof in~(ii).

\begin{exercise}
Prove the converse of Proposition~\ref{prop:soundness}(a), i.\,e., show that the second sentence from that result implies the first. \emph{Remark:} This yields a natural deduction version of Proposition~\ref{prop:completeness}. In contrast to the latter, your result should apply to arbitrary $\mathsf{PL}$-formulas and not just to~$\mathsf{NNF}$-formulas. Exercise~\ref{ex:nat-deduct}(b) is helpful in that respect.
\end{exercise}

It will later be crucial to have the following generalization of Proposition~\ref{prop:completeness} with potentially infinite sets of assumptions. The fact that we have a proof from assumptions may be more transparent if one uses Corollary~\ref{cor:sequent-nat-deduct} to write the conclusion as $\theta_1,\ldots,\theta_n\vdash\varphi$.

\begin{theorem}[Completeness for propositional logic]\label{thm:completeness-prop}
Let $\varphi$ and $\Theta$ be an $\mathsf{NNF}$-formula and a set  of such formulas, respectively. Assume that~$\varphi$ is a logical consequence of~$\Theta$, in the sense that $\llbracket\varphi\rrbracket_v=1$ holds for any valuation~$v$ with $\llbracket\theta\rrbracket_v=1$ for all~$\theta\in\Theta$. Then we have $\vdash\nega\theta_1,\ldots,\nega\theta_n,\varphi$ for some finite collection of~$\theta_i\in\Theta$.
\end{theorem}
\begin{proof}
Let us fix an enumeration $\theta_1,\theta_2,\ldots$ of the set~$\Theta$, which is countable in view of Definition~\ref{def:PL-formula} (see also Exercise~\ref{ex:Cantor-pairing} below). As in the proof of Proposition~\ref{prop:completeness}, we recursively construct a binary tree~$T$ and assign an ordered sequent~$\Gamma(\sigma)$ to each~$\sigma\in T$. In the base case we declare $\langle\rangle\in T$ and $\Gamma(\langle\rangle)=\langle\varphi\rangle$. When $\tau\in T$ has even length $l(\tau)=2(k-1)$, we stipulate that $\tau\star i\in T$ holds for both $i<2$ and that we have
\begin{equation*}
\Gamma(\tau\star 0)=\Gamma(\tau),\theta_k\quad\text{and}\quad\Gamma(\tau\star 1)=\Gamma(\tau),\nega\theta_k.
\end{equation*}
Furthermore, we declare that $\tau\star 0$ is a leaf of~$T$. On the other hand, the recursion step for $\sigma=\tau\star 1$ (odd length) is defined exactly as in the proof of~Proposition~\ref{prop:completeness}.

Let us first assume that $T$ has a branch~$f$. Statements~(i) to~(iii) from the proof of Proposition~\ref{prop:completeness} remain valid with essentially the same proof (simply skip even stages and note that formulas are still rotated to the front). As before, we thus get a valuation~$v$ such that $\llbracket\psi\rrbracket_v=0$ holds for any~$\psi$ that occurs on~$f$. In particular we have $\llbracket\varphi\rrbracket_v=0$. Since $f[2(k-1)]\star 0$ was declared to be a leaf, we must have $f(2(k-1))=1$. This means that~$\nega\theta_k$ occurs on~$f$, which allows us to conclude~$\llbracket\neg\theta_k\rrbracket_v=\llbracket\nega\theta_k\rrbracket_v=0$ (cf.~Lemma~\ref{lem:tilde-negation}) and hence~$\llbracket\theta_k\rrbracket_v=1$. Since $\theta_k$ can be an arbitrary element of~$\Theta$, this is against the assumption of the theorem.

Now assume that~$T$ is finite. Consider the height function~$h:T\to\mathbb N$ from the proof of Proposition~\ref{prop:completeness} and choose~$n\in\mathbb N$ so large that we have $h(\langle\rangle)\leq 2n$. We show $\vdash\nega\theta_1,\ldots,\nega\theta_n,\Gamma(\sigma)$ by induction over~$h(\sigma)$. Note that the conclusion of the theorem is obtained for~$\sigma=\langle\rangle$. Compared with the proof of Proposition~\ref{prop:completeness}, there are two new cases. First, assume that $\sigma$ is a leaf~$\tau\star 0$ for a sequence~$\tau$ that has even length~$2(k-1)$. We get
\begin{equation*}
2(k-1)=l(\tau)<l(\sigma)\leq h(\langle\rangle)\leq 2n
\end{equation*}
and hence~$k\leq n$. In view of Exercise~\ref{ex:nega}, we thus obtain $\vdash\nega\theta_1,\ldots,\nega\theta_n,\Gamma(\tau),\theta_k$. This is as required, since $\Gamma(\tau),\theta_k$ coincides with~$\Gamma(\sigma)$. In the second new case, the sequence $\sigma$ itself has even length~$2(k-1)$. Here the induction hypothesis (applied to the sequences $\sigma\star 0$ and $\sigma\star 1$) yields $\vdash\Gamma(\sigma),\theta_k$ as well as $\vdash\Gamma(\sigma),\nega\theta_k$. We can conclude by an application of the cut rule.
\end{proof}

Proof systems for propositional logic are connected to deep questions in computational complexity. With this notable exception, however, their~importance is diminuished by the fact that validity can be decided via truth tables. In the following section, we prove completeness in a more general setting, where no decision method such as truth tables is available. Once this is accomplished, we explain how the completeness result relates to G\"odel's famous incompleteness~theorems.

\subsection{First-Order Logic}\label{subsect:FO}

In this subsection, we present a far-reaching extension of propositional logic, which is known as predicate logic or first-order logic. The latter is of singular importance in several areas of mathematical logic, not least because it provides a framework for axiom systems such as Zermelo-Fraenkel set theory, which allow for a formal development of most if not all contemporary mathematics.

The terms of the following definition stand for mathematical objects (such as the vertices of a graph) that are related in certain ways (e.\,g., by an edge between~them). We can form conjunctions and disjunctions as in propositional logic (which amounts to the special case where we have $0$-ary relation symbols only). A new kind of~formulas can be built thanks to the so-called quantifiers $\forall$ (`for all') and~$\exists$ (`exists').

\begin{definition}\label{def:signature}
A signature consists of a set of so-called function symbols and a set of relation symbols, each of which is assigned a \mbox{non-negative} integer called its arity. The pre-terms over a signature are generated as follows:
\begin{enumerate}[label=(\roman*)]
\item We have a countably infinite set of pre-terms $v_0,v_1,\ldots$ called variables.
\item When $f$ is an $n$-ary function symbol from our signature (i.\,e., one of arity~$n$) and we already have pre-terms $t_1,\ldots,t_n$, then $ft_1\ldots t_n$ is a pre-term.
\end{enumerate}
Note that~(ii) includes the case of $0$-ary function symbols, which we call constants. For a given signature, the pre-formulas are generated by the following clauses:
\begin{enumerate}[label=(\roman*')]
\item If $R$ is an $n$-ary relation symbol and $t_1,\ldots,t_n$ are pre-terms of our signature, then both $Rt_1\ldots t_n$ and $\neg Rt_1\ldots t_n$ are pre-formulas.
\item When $\varphi$ and $\psi$ are pre-formulas, then so are $(\varphi\land\psi)$ and $(\varphi\lor\psi)$.
\item If $\varphi$ is a pre-formula, then so are $\forall v_{2i}\,\varphi$ and $\exists v_{2i}\,\varphi$ for any $i\in\mathbb N$.
\end{enumerate}
Let us stress that only variables with even index are allowed in clause~(iii').
\end{definition}

There are several ways in which our definition differs from other presentations: First and most notably, the special role of even indices and the very notion of pre-formula relates to a technical issue around substitution, which will be clarified by the introduction of formulas in Definition~\ref{def_fv-formula} and the paragraph that follows the latter. Second, some authors demand that any signature contains a binary relation symbol~$=$ for equality, which will then receive a special treatment. The latter can be incorporated into our approach, as we will see at the end of this subsection. Third, our definition does only produce formulas in negation normal form. This is motivated by our previous treatment of propositional logic. It is no real restriction, since the following provides a negation for arbitrary formulas.

\begin{definition}\label{def:nega-pred-log}
The negation~$\nega\varphi$ of a pre-formula~$\varphi$ is recursively explained by
\begin{align*}
\nega Rt_1\ldots t_n\,&=\,\neg Rt_1\ldots t_n,&\enspace \nega(\varphi\land\psi)\,&=\,(\nega\varphi)\lor(\nega\psi),&\enspace\nega(\forall v_{2i}\,\varphi)\,&=\,\exists v_{2i}\,\nega\varphi,\\ 
\nega\neg Rt_1\ldots t_n\,&=\,Rt_1\ldots t_n,&\enspace \nega(\varphi\lor\psi)\,&=\,(\nega\varphi)\land(\nega\psi),&\enspace\nega(\exists v_{2i}\,\varphi)\,&=\,\forall v_{2i}\,\nega\varphi.
\end{align*}
We will usually write $\neg\varphi$ at the place of~$\nega\varphi$. Also, we will write $\varphi\to\psi$ for $\neg\varphi\lor\psi$ and use further abbreviations that are familiar from propositional logic.
\end{definition}

Concerning notation, we point out that quantifiers in pre-formulas bind stronger than propositional connectives, so that $\forall v_{2i}\,\varphi\land\psi$ is the conjunction of~$\forall v_{2i}\,\varphi$ and~$\psi$, which differs from~$\forall v_{2i}(\varphi\land\psi)$. The variables $v_i$ (with even or odd index $i$ depending on the context) will often be represented by other lower case letters, where $x,y,z$ and $u,w$ are typical choices. We shall sometimes contract like quantifiers and write $\forall x,y\,\varphi$ rather than $\forall x\forall y\,\varphi$. The following is an opportunity to see a few examples.

\begin{exercise}
Use pre-formulas over the signature $\{E\}$ with a binary relation symbol to state that $E$ represents an equivalence relation. \emph{Hint:} As a different but related example, note that $\forall x\forall y\,\neg(x<y\land y<x)$ expresses that the binary relation symbol~$<$ represents a relation that is antisymmetric. Note that we have used so-called infix notation by writing $x<y$ rather than ${<}xy$.
\end{exercise}

The intuitive meaning of formulas is made official by the following.

\begin{definition}\label{def:models-pred-log}
When $\sigma$ is a signature, a $\sigma$-structure~$\mathcal M$ consists of a non-empty set~$M$ (called universe), a subset $R^{\mathcal M}\subseteq M^n$ for each~$n$-ary relation symbol~$R$ in~$\sigma$ and a function $f^{\mathcal M}:M^n\to M$ for each $n$-ary function symbol~$f$ in~$\sigma$. By a variable assignment (for~$\mathcal M$), we mean a function~$\eta:\mathbb N\to M$. Given such an assignment, the interpretation~$t^{\mathcal M,\eta}$ of a pre-term over~$\sigma$ is recursively defined by
\begin{equation*}
v_i^{\mathcal M,\eta}=\eta(i)\quad\text{and}\quad (ft_1\ldots t_n)^{\mathcal M,\eta}=f^{\mathcal M}(t_1^{\mathcal M,\eta},\ldots,t_n^{\mathcal M,\eta}).
\end{equation*}
Given a variable assignment~$\eta$ and an element~$a\in M$, we define $\eta_i^a$ by $\eta_i^a(i)=a$ and $\eta_i^a(j)=\eta(j)$ for~$i\neq j$. For a pre-formula $\varphi$ over~$\sigma$, we determine $\mathcal M,\eta\vDash\varphi$ (say that $\mathcal M$ is a model of or satisfies~$\varphi$ under the assignment~$\eta$) by the following clauses:
\begin{alignat*}{3}
\mathcal M,\eta&\vDash Rt_1\ldots t_n\quad&&\Leftrightarrow\quad(t_1^{\mathcal M,\eta},\ldots,t_n^{\mathcal M,\eta})\in R^{\mathcal M},\\
\mathcal M,\eta&\vDash\neg Rt_1\ldots t_n\quad&&\Leftrightarrow\quad(t_1^{\mathcal M,\eta},\ldots,t_n^{\mathcal M,\eta})\notin R^{\mathcal M},\\
\mathcal M,\eta&\vDash\varphi_0\land\varphi_1\quad&&\Leftrightarrow\quad\mathcal M,\eta\vDash\varphi_i\text{ for both }i<2,\\
\mathcal M,\eta&\vDash\varphi_0\lor\varphi_1\quad&&\Leftrightarrow\quad\mathcal M,\eta\vDash\varphi_i\text{ for some }i<2,\\
\mathcal M,\eta&\vDash\forall v_i\,\varphi\quad&&\Leftrightarrow\quad\mathcal M,\eta_i^a\vDash\varphi\text{ for all }a\in M,\\
\mathcal M,\eta&\vDash\exists v_i\,\varphi\quad&&\Leftrightarrow\quad\mathcal M,\eta_i^a\vDash\varphi\text{ for some }a\in M.
\end{alignat*}
We write $\mathcal M\vDash\varphi$ when we have $\mathcal M,\eta\vDash\varphi$ for every~$\eta:\mathbb N\to M$, and we write $\vDash\varphi$ when $\mathcal M\vDash\varphi$ holds for every~$\mathcal M$. In the latter case, we say that~$\varphi$ is logically valid. If we have $\vDash\varphi\leftrightarrow\psi$, then $\varphi$ and $\psi$ are called logically equivalent. When $\Psi$ is a set of pre-formulas, we write $\Psi\vDash\varphi$ (say $\varphi$ is a logical consequence of~$\Psi$) when we have $\mathcal M,\eta\vDash\varphi$ for all~$\mathcal M$ and $\eta$ such that $\mathcal M,\eta\vDash\psi$ holds for every~$\psi\in\Psi$. If there is some~$\mathcal M$ and $\eta$ with $\mathcal M,\eta\vDash\psi$ for all~$\psi\in\Psi$, then we say that $\Psi$ is satisfiable.
\end{definition}

We will use the words `language' and `model' as synonyms for `signature' and `structure' (for some signature), respectively. This is common practice in some areas of logic, even though one can make meaningful distinctions between these terms. To justify Definition~\ref{def:nega-pred-log}, one can use induction over pre-formulas to check
\begin{equation*}
\mathcal M,\eta\vDash\neg\varphi\quad\Leftrightarrow\quad\mathcal M,\eta\nvDash\varphi,
\end{equation*}
where the right side expresses that $\mathcal M,\eta\vDash\varphi$ does not hold. Also note that $\neg\neg\varphi$ is literally the same pre-formula as~$\varphi$, like in the case of propositional logic.

\begin{definition}\label{def_fv-formula}
The free variables of pre-terms and pre-formulas are recursively defined by $\fv(v_i)=\{v_i\}$ and $\fv(ft_0\ldots t_{n-1})=\bigcup_{i<n}\fv(t_i)$ as well as
\begin{gather*}
\fv(Rt_0\ldots t_{n-1})=\fv(\neg Rt_0\ldots t_{n-1})=\textstyle\bigcup_{i<n}\fv(t_i),\\
\fv(\varphi_0\land\varphi_1)=\fv(\varphi_0\lor\varphi_1)=\textstyle\bigcup_{i<2}\fv(\varphi_i),\\
\fv(\forall v_{2i}\,\varphi)=\fv(\exists v_{2i}\,\varphi)=\fv(\varphi)\backslash\{v_{2i}\}.
\end{gather*}
Terms and formulas are defined as pre-terms and pre-formulas, respectively, in which no variables~$v_{2i}$ with even index are free. By a literal we mean a formula that has the form $Rt_1\ldots t_n$ or $\neg Rt_1\ldots t_n$. A formula without any free variables is called closed or a sentence. By a theory we mean a set of sentences.
\end{definition}

We now explain the point of pre-formulas and the special role of even indices. The variable~$y$ in formulas such as $\forall y\,\varphi$ and $\forall x\exists y\,\psi$ is called bound. At least~intuitively, the names of bound variables do not matter, since formulas such as $\forall x\, Px$ and $\forall y\, Py$ are equivalent. We distinguish between even and odd indices to ensure that the bound and free variables in terms and formulas are always distinct. This has the advantage that it simplifies substitution. Indeed, substituting~$y$ by~$fx$ in $\exists x\,Rxy$ should not yield $\exists x\,Rxfx$, since the intended interpretation would change if we did allow that the free variable~$x$ in~$fx$ is captured by a bound variable. One common solution is to rename the bound variable, so that the indicated substitution results in $\exists z\,Rzfx$ for some new~$z$. While this is a valid approach, it has the effect that substitution is no longer defined by recursion over formulas (as we substitute into $Rzy$ and not into the subformula $Rxy$ of $\exists x\,Rxy$), which creates some complications (in particular for the proof of G\"odel's incompleteness theorems in Section~\ref{subsect:goedel}).

\begin{definition}\label{def:substitution}
For pre-terms $s,t$ and a variable~$v_i$ with even or odd~$i\in\mathbb N$, the substitution $s[v_i/t]$ is recursively given by
\begin{equation*}
v_j[v_i/t]=\begin{cases}
t & \text{if $i=j$},\\
v_j & \text{otherwise},
\end{cases}\qquad
(fs_1\ldots s_n)[v_i/t]=fs_1[v_i/t]\ldots s_n[v_i/t].
\end{equation*}
When~$\varphi$ is a pre-formula and~$t$ is a term, we let $\varphi[v_i/t]$ be defined by
\begin{gather*}
\begin{aligned}
(Rs_1\ldots s_n)[v_i/t]&=Rs_1[v_i/t]\ldots s_n[v_i/t],\\
(\neg Rs_1\ldots s_n)[v_i/t]&=\neg Rs_1[v_i/t]\ldots s_n[v_i/t],\\
(\varphi_0\circ\varphi_1)[v_i/t]&=\varphi_0[v_i/t]\circ\varphi_1[v_i/t],
\end{aligned}\\
(\mathsf Q v_{2j}\,\varphi)[v_i/t]=\begin{cases}
\mathsf Q v_{2j}\,\varphi & \text{if $i=2j$},\\
\mathsf Q v_{2j}\,\varphi[v_i/t] & \text{otherwise},
\end{cases}
\end{gather*}
where $\circ$ can be $\land$ or~$\lor$ and $\mathsf Q$ can be $\forall$ or~$\exists$.
\end{definition}

Let us stress that we have left $\varphi[v_i/t]$ undefined when~$t$ is a pre-term but no~term. This makes sense as we could otherwise capture free variables once again. Arguing by induction over pre-formulas, one can show
\begin{equation*}
\fv(\varphi[x/t])\subseteq(\fv(\varphi)\backslash\{x\})\cup\fv(t),
\end{equation*}
so that $\varphi[x/t]$ is a formula (and not just a pre-formula) when the same holds~for~$\varphi$. Also note that $\neg\varphi$ will be a formula in this case, as we have $\fv(\neg\varphi)=\fv(\varphi)$. Another induction shows that $(\neg\varphi)[x/t]$ coincides with $\neg(\varphi[x/t])$, so that we may omit the parentheses. The following justifies our definition of substitution.

\begin{exercise}\label{ex:substitution}
Use induction over pre-formulas to show that we have
\begin{equation*}
\mathcal M,\eta\vDash\varphi[v_i/t]\quad\Leftrightarrow\quad\mathcal M,\eta_i^a\vDash\varphi\text{ with }a=t^{\mathcal M,\eta}.
\end{equation*}
\emph{Hint:} First use induction over pre-terms to prove $s[v_i/t]^{\mathcal M,\eta}=s^{\mathcal M,\xi}$ for $\xi=\eta_i^a$.
\end{exercise}

It is worth pointing out that $\mathcal M,\eta\vDash\varphi$ does only depend on the values of~$\eta$ on variables that are free in~$\varphi$. We will sometimes write $\varphi(v_{i(1)},\ldots,v_{i(n)})$ to refer to the free variables of~$\varphi$ (in other contexts also when not all $v_{i(j)}$ occur in~$\varphi$ and not all its free~variables are among the $v_{i(j)}$). In this case, we also write $\mathcal M\vDash\varphi(a_1,\ldots,a_n)$ in order to express that $\mathcal M,\eta\vDash\varphi$ holds for the assignment~$\eta$ given by $\eta(i(j))=a_j$. Analogous notation will be used for terms. The universal closure of a formula~$\varphi$ is defined as $\forall x_1,\ldots,x_n\,\varphi$ where $x_1,\ldots,x_n$ are the free variables of~$\varphi$ (in some given order). We note that $\varphi$ is logically valid precisely when the same holds for its universal closure.

Let us say that a formula is quantifier-free if its construction uses clauses~(i') and~(ii') of Definition~\ref{def:signature} only. We declare that any quantifier-free formula is in prenex normal form and that $\forall x\,\varphi$ and $\exists x\,\varphi$ are in prenex normal form whenever the same holds for~$\varphi$. In other words, a formula in prenex normal form begins with a block of quantifiers that is followed by a quantifier-free formula. For example, the formula $\exists x\forall y(Px\to Qy)$ is in prenex normal form while $\forall x\, Px\to\forall y\,Qy$ is not.

\begin{exercise}
Show that $x\notin\fv(\psi)$ entails
\begin{alignat*}{4}
&\vDash(\psi\to\forall x\,\rho)&&\leftrightarrow\forall x(\psi\to\rho),\qquad&&\vDash(\psi\to\exists x\,\rho)&&\leftrightarrow\exists x(\psi\to\rho),\\
&\vDash(\forall x\,\varphi\to\psi)&&\leftrightarrow\exists x(\varphi\to\psi),\qquad&&\vDash(\exists x\,\varphi\to\psi)&&\leftrightarrow\forall x(\varphi\to\psi).
\end{alignat*}
Provide examples to show that one cannot drop the assumption~$x\notin\fv(\psi)$. Derive that any formula is logically equivalent to one in prenex normal form.
\end{exercise}

We now extend our sequent calculus from propositional to predicate logic. Natural deduction could also be extended, but we will not do so in the present lecture notes. By a sequent we mean a finite set of formulas (which are in negation normal form by the comment before Definition~\ref{def:nega-pred-log}). We adopt notational conventions such as $\Gamma,\varphi=\Gamma\cup\{\varphi\}$ from the case of propositional logic. The free variables of a sequent are given by $\fv(\Gamma)=\bigcup_{i<n}\fv(\varphi_i)$ for $\Gamma=\varphi_0,\ldots,\varphi_{n-1}$.

\begin{definition}\label{def:sequent-calc-FO}
For a sequent~$\Gamma$ we define~$\vdash\Gamma$ by the following recursive clauses:
\begin{description}[labelwidth=5.5ex,labelindent=\parindent,leftmargin=!,before={\renewcommand\makelabel[1]{(##1)}}]
\item[$\mathsf{Ax}$] When $\Gamma$ contains some literal and its negation, we have $\vdash\Gamma$.
\item[$\land$] Given $\vdash\Gamma,\varphi$ and $\vdash\Gamma,\psi$, we get $\vdash\Gamma,\varphi\land\psi$.
\item[$\lor$] From $\vdash\Gamma,\psi$ we get both $\vdash\Gamma,\psi\lor\theta$ and $\vdash\Gamma,\varphi\lor\psi$.
\item[$\forall$] If we have $\vdash\Gamma',\varphi[x/y]$ for some $\Gamma'\subseteq\Gamma$ and some~$y$ that satisfies the so-called variable condition~$y\notin\fv(\Gamma',\forall x\,\varphi)$, then we get $\vdash\Gamma,\forall x\,\varphi$.
\item[$\exists$] Given $\vdash\Gamma,\varphi[x/t]$ for some term~$t$, we get $\vdash\Gamma,\exists x\,\varphi$.
\item[$\mathsf{Cut}$] When have $\vdash\Gamma,\varphi$ and $\vdash\Gamma,\neg\varphi$ for some formula~$\varphi$, we get $\vdash\Gamma$.
\end{description}
We write $\vdash_0\Gamma$ when $\vdash\Gamma$ can be derived without~($\mathsf{Cut}$). When $\Theta$ is a set of formulas and~$\varphi$ is a formula, we write $\Theta\vdash\varphi$ to express that $\vdash\neg\theta_1,\ldots,\neg\theta_n,\varphi$ holds for some finite collection of formulas~$\theta_i\in\Theta$. If we have $\Theta\nvdash\bot$, we say that $\Theta$ is consistent.
\end{definition}

In~($\forall$), the variable condition corresponds to the intuitive requirement that $y$ is arbitrary. Without it, we would get $\vdash\forall x\,Px,\forall x\,\neg Px$ and hence $\vdash\exists x\,Px\to\forall x\,Px$, which is clearly not valid (recall that $\psi\to\theta$ is notation for $\neg\psi\lor\theta$ and that $\neg\exists x\,Px$ coincides with $\forall x\,\neg Px$). By considering $\Gamma'\subseteq\Gamma$, we allow that $y$ is free in some additional formulas in~$\Gamma$. This is convenient (though not really necessary) in the context of weakening (cf.~Lemma~\ref{lem:weakening}): for $\Gamma\subseteq\Delta$, a straightforward induction shows that $\vdash\Gamma$ entails $\vdash\Delta$ (analogous for~$\vdash_0$). Still in~($\forall$), the variable~$y$ must have the form $v_{2i+1}$, so that $\varphi[x/y]$ is a formula. Note that we get to choose the name of the bound variable~$x$ when we infer $\vdash \forall x\,Px,\exists z\,\neg Pz$ from $\vdash Py,\exists z\,\neg Pz$. More generally, if $\varphi'$ results from~$\varphi$ by any permissible renaming of bound variables, then $\varphi'\lor\neg\varphi$ is logically valid, so that we will be able to conclude $\vdash_0\varphi',\neg\varphi$ once completeness is established.

\begin{exercise}\label{ex:nega-pred}
(a) Show that we have $\vdash_0\varphi,\neg\varphi$ for any formula of predicate logic (cf.~Exercise~\ref{ex:nega}).

(b) Derive $\vdash_0\exists y(Py\to\forall x\,Px)$. \emph{Remark:} The given conclusion is sometimes called the drinker formula: upon one reading, it asserts the existence of a person~$y$ such that everybody drinks when~$y$ does. It may be instructive to first consider why the formula is logically valid.
\end{exercise}

Let us now extend the soundness result (cf.~Proposition~\ref{prop:soundness}). We recall that~$\bigvee\Gamma$ denotes the disjunction over the formulas in the sequent~$\Gamma$.

\begin{proposition}[Soundness for predicate logic]\label{prop:soundness-FO}
From~$\vdash\Gamma$ we get $\vDash\bigvee\Gamma$.
\end{proposition}
\begin{proof}
As in the case of predicate logic, one argues by induction over proofs, i.\,e., over the number of times that some clause from Definition~\ref{def:sequent-calc-FO} has been applied in the derivation of~$\vdash\Gamma$. In the first new case, the sequent~$\Gamma$ contains a formula~$\forall v_i\,\varphi$ that has been deduced with premise~$\vdash\Gamma',\varphi[v_i/v_j]$ for $\Gamma'\subseteq\Gamma$ and $v_j\notin\fv(\Gamma',\forall v_i\,\varphi)$. Consider an arbitrary model~$\mathcal M$ and variable assignment~$\eta$. We are done if we have $\mathcal M,\eta\vDash\psi$ for some~$\psi\in\Gamma'$, so we assume that the latter fails. To conclude that we have~$\mathcal M,\eta\vDash\forall v_i\,\varphi$, we show $\mathcal M,\eta_i^a\vDash\varphi$ for an arbitrary~$a$ from the universe of~$\mathcal M$. Let us note that $i$ and~$j$ are different, since they are even and odd, respectively. In view of $v_j\notin\fv(\forall v_i\,\varphi)=\fv(\varphi)\backslash\{v_i\}$, we get $v_j\notin\fv(\varphi)$. Thus the open claim is equivalent to~$\mathcal M,(\eta_j^a)_i^a\vDash\varphi$. The latter is itself equivalent to $\mathcal M,\eta_j^a\vDash\varphi[v_i/v_j]$ by Exercise~\ref{ex:substitution}, as we have $a=v_j^{\mathcal M,\xi}$ with $\xi=\eta_j^a$. Now if we had $\mathcal M,\eta_j^a\nvDash\varphi[v_i/v_j]$, the induction hypothesis would yield $\mathcal M,\eta_j^a\vDash\psi$ for some~$\psi\in\Gamma'$. By $v_j\notin\fv(\Gamma')$ we would obtain $\mathcal M,\eta\vDash\psi$, against the assumption above. In the second new case, $\Gamma$~contains a formula~$\exists v_i\,\varphi$ that was derived with premise~$\vdash\Gamma,\varphi[v_i/t]$. For any~$\mathcal M$ and~$\eta$, we may inductively assume~$\mathcal M,\eta\vDash\varphi[v_i/t]$ as above. Now Exercise~\ref{ex:substitution} yields $\mathcal M,\eta_i^a\vDash\varphi$ for a suitable~$a$, so that we get $\mathcal M,\eta\vDash\exists v_i\,\varphi$ as desired.
\end{proof}

In order to extend Theorem~\ref{thm:completeness-prop} from propositional to first-order logic, we adopt the following convention. By countable we will always mean finite or countably infinite, so that a set is countable precisely if it admits an injection into~$\mathbb N$.

\begin{convention}\label{conv:countable}
From this point until the end of Section~\ref{sect:cut-free-completeness}, we tacitly assume that any signature is countable.
\end{convention}

Let us note that results on the uncountable case will be proved in Section~\ref{subsect:ultraprod}.

\begin{proposition}\label{prop:completeness-FO}
From $\vDash\bigvee\Gamma$ we get~$\vdash_0\Gamma$, for any sequent~$\Gamma$ of predicate logic.
\end{proposition}
\begin{proof}
We explain how the proof of Proposition~\ref{prop:completeness} can be adapted. As in that proof, we define a tree $T\subseteq 2^{<\omega}$ and an ordered sequent~$\Gamma(\sigma)$ for each~$\sigma\in T$. For the recursion step, we again write
\begin{equation*}
\Gamma(\sigma)=\langle\varphi,\varphi_1,\ldots,\varphi_n\rangle\quad\text{and}\quad\Delta=\langle\varphi_1,\ldots,\varphi_n\rangle.
\end{equation*}
In the case distinction that defines~$\Gamma(\sigma\star i)$, we must consider two new cases. First assume that~$\varphi$ is of the form~$\forall x\,\psi$. We then put $\Gamma(\sigma\star i):=\Delta,\varphi,\psi[x/y]$ for some variable $y\notin\fv(\Delta,\varphi)$. As preparation for the other case, we fix an enumeration that lists all terms $t_0,t_1,\ldots$ over the given signature (cf.~Convention~\ref{conv:countable}). Secondly, we now assume that $\varphi$ has the form~$\exists x\,\psi$. We then put $\Gamma(\sigma\star i):=\Delta,\varphi,\psi[x/t_j]$ for the minimal~$j$ such that $\psi[x/t_j]$ does not occur in~$\Gamma(\sigma)$.

When the resulting tree is finite, it witnesses that we have $\vdash_0\Gamma$, as in the proof of Proposition~\ref{prop:completeness}. We now assume that~$T$ has a branch~$f:\mathbb N\to\{0,1\}$. As before, we let $\mathcal F$ be the collection of formulas that occur in~$\Gamma(f[n])$ for some~$n\in\mathbb N$. Our task is to find a model~$\mathcal M$ and a variable assignement~$\eta$ such that we have~$\mathcal M,\eta\nvDash\bigvee\Gamma$. For the universe, we take the set~$M=\{t_i\,|\,i\in\mathbb N\}$ of terms. The interpretation of an $n$-ary function symbol~$f$ is given by
\begin{equation*}
f^{\mathcal M}:M^n\to M\quad\text{with}\quad f^{\mathcal M}(t_1,\ldots,t_n):=ft_1\ldots t_n.
\end{equation*}
To define our variable assignment~$\eta$, we put $\eta(i):=v_i$. A straightforward induction yields $t^{\mathcal M,\eta}=t$ for any term~$t$.
For an $n$-ary relation symbol~$R$, we define
\begin{equation*}
R^{\mathcal M}:=\{(t_1,\ldots,t_n)\in M^n\,|\,\neg Rt_1\ldots t_n\in\mathcal F\}.
\end{equation*}
To see that this yields the desired countermodel, we note that~$\mathcal F$ validates~(i) to~(iii) from the proof of Proposition~\ref{prop:completeness} as well as the following:
\begin{enumerate}[label=(\roman*)]\setcounter{enumi}{3}
\item If $\forall x\,\psi$ lies in~$\mathcal F$, then so does~$\psi[x/y]$ for some variable~$y$.
\item If $\exists x\,\psi$ lies in~$\mathcal F$, then so does~$\psi[x/t_j]$ for every~$j\in\mathbb N$.
\end{enumerate}
To establish~(v) by contradiction, we assume that we have $\exists x\,\psi\in\mathcal F$ and that $j$ is minimal with~$\psi[x/t_j]\notin\mathcal F$. Since we always have~$\Gamma(f[n])\subseteq\Gamma(f[n+1])$, we may pick a large $N\in\mathbb N$ such that $\Gamma(f[N])$ contains $\exists x\,\psi$ as well as $\psi[x/t_i]$ for all~$i<j$. As before, we may also assume that $\exists x\,\psi$ is the first formula in~$\Gamma(f[N])$. By construction, it follows that $\psi[x/t_j]$ occurs in~$\Gamma(f[N]\star f(N))=\Gamma(f[N+1])$ and thus in~$\mathcal F$, against our assumption. The argument for~(iv) is similar but easier.

We now use induction on the length of formulas to show that $\mathcal M,\eta\nvDash\varphi$ holds for any~$\varphi\in\mathcal F$. In view of $\Gamma\subseteq\mathcal F$ this will yield~$\nvDash\bigvee\Gamma$, as needed to complete the proof. First assume that $\varphi$ has the form $Rt_1\ldots t_n$. By clause~(i) from the proof of Proposition~\ref{prop:completeness}, we obtain $\neg Rt_1\ldots t_n\notin\mathcal F$, so that the definition of~$\mathcal M$ yields $(t_1,\ldots,t_n)\notin R^{\mathcal M}$. In view of $t^{\mathcal M,\eta}=t$ we can conclude~$\mathcal M,\eta\nvDash Rt_1\ldots t_n$, as desired. The case of a formula~$\neg Rt_1\dots t_n$ is similar, and conjunctions and disjunctions are treated like in the proof of Proposition~\ref{prop:completeness}. Now assume that $\varphi\in\mathcal F$ has the form~$\forall v_i\,\psi$. In view of clause~(iv) above, the induction hypothesis yields $\mathcal M,\eta\nvDash\psi[v_i/y]$ for some~$y$. Due to Exercise~\ref{ex:substitution} we get $\mathcal M,\eta_i^y\nvDash\psi$, which entails $\mathcal M,\eta\nvDash\varphi$. Finally, we assume that $\varphi$ has the form~$\exists v_i\,\psi$. Our task is to establish that $\mathcal M,\eta_i^a\nvDash\psi$ holds for any~$a\in M$. Consider $j\in\mathbb N$ with $a=t_j=t_j^{\mathcal M,\eta}$. Once again by Exercise~\ref{ex:substitution}, the open claim amounts to~$\mathcal M,\eta\nvDash\psi[v_i/t_j]$. We conclude by~(v) and the induction hypothesis.
\end{proof}

Soundness and cut-free completeness yield a semantic proof of the following result on cut elimination, just like for propositional logic.

\begin{corollary}
If we have $\vdash\Gamma$, we even get~$\vdash_0\Gamma$.
\end{corollary}

We now come to the main result of this subsection.

\begin{theorem}[Completeness for first-order logic~\cite{goedel-completeness}]\label{thm:completeness-AL}
From $\Theta\vDash\varphi$ we get $\Theta\vdash\varphi$, for any set~$\Theta$ of formulas and any formula~$\varphi$.
\end{theorem}
\begin{proof}
Given that~$\Theta$ is countable by the convention above, we can conclude as in the proof of Theorem~\ref{thm:completeness-prop}, with the modifications that we have made in the proof of Proposition~\ref{prop:completeness-FO}.
\end{proof}

The following provides an important reformulation of completeness.

\begin{exercise}\label{ex:completeness}
Show that completeness implies that any consistent set of formulas is satisfiable. Also give a direct proof of the converse implication.
\end{exercise}

Finally, we discuss the treatment of equality in predicate logic.

\begin{definition}\label{def:equality}
A signature with equality is a signature~$\sigma$ together with a choice of a binary relation symbol in~$\sigma$, which we always denote by~$=$ and write infix ($x=y$ rather than ${=}xy$). By the associated equality axioms, we mean the universal closures of the following formulas, where $f$ and~$R$ range over all $n$-ary function and relation symbols in the signature:
\begin{gather*}
x=x,\qquad x=y\to y=x,\qquad x=y\land y=z\to x=z,\\
\begin{aligned}
x_1=y_1\land\ldots\land x_n=y_n&\to fx_1\ldots x_n=fy_1\ldots y_n,\\
x_1=y_1\land\ldots\land x_n=y_n&\to(Rx_1\ldots x_n\leftrightarrow Ry_1\ldots y_n).
\end{aligned}
\end{gather*}
When $\sigma$ is a signature with equality, a strict $\sigma$-structure is a $\sigma$-structure~$\mathcal M$ such that $=^{\mathcal M}$ is the actual equality relation $\{(a,a)\,|\,a\in M\}$ on the universe~$M$.
\end{definition}

Some authors will only consider signatures with equality and demand that all structures are strict. By the following result and its proof, this is no real restriction.

\begin{theorem}\label{thm:strict-model}
Consider a set of formulas~$\Theta$ that includes the equality axioms for a signature~$\sigma$ with equality. If $\Theta$ can be satisfied by some $\sigma$-structure, it can be satisfied by a strict $\sigma$-structure.
\end{theorem}
\begin{proof}
Consider any $\sigma$-structure~$\mathcal M$ such that we have $\mathcal M,\xi\vDash\Theta$ for some variable assignment~$\xi$. Since $\mathcal M$ satisfies the equality axioms (which are closed formulas), $=^\mathcal M$ is an equivalence relation on the universe~$M$. Let~$N$ be the quotient $M/{=^{\mathcal M}}$, i.\,e., the set of equivalence classes $[a]=\{b\in M\,|\,a=^{\mathcal M}b\}$. For an $n$-ary function symbol~$f$, the corresponding equality axiom ensures that we get a well-defined function $f^{\mathcal N}:N^n\to N$ by setting
\begin{equation*}
f^{\mathcal N}([a_1],\ldots,[a_n])=[f^{\mathcal M}(a_1,\ldots,a_n)].
\end{equation*}
When~$R$ is an $n$-ary relation symbol, we may similarly put
\begin{equation*}
\mathcal R^{\mathcal N}=\{([a_1],\ldots,[a_n])\,|\,(a_1,\ldots,a_n)\in R^{\mathcal M}\}.
\end{equation*}
This yields a strict $\sigma$-structure~$\mathcal N$, as $[a]=^{\mathcal N}[b]$ is equivalent to $a=^{\mathcal M}b$, which means that $[a]$ and~$[b]$ coincide. Given a variable assignment~$\eta:\mathbb N\to M$, we define an assignment $[\eta]:\mathbb N\to\mathcal N$ by $[\eta](i):=[\eta(i)]$. A straightforward induction over terms shows that we have $t^{\mathcal N,[\eta]}=[t^{\mathcal M,\eta}]$. By induction over formulas, one can derive that $\mathcal M,\eta\vDash\varphi$ is equivalent to $\mathcal N,[\eta]\vDash\varphi$. Given that we have $\mathcal M,\xi\vDash\Theta$, we thus get $\mathcal N,[\xi]\vDash\Theta$, so that $\Theta$ is satisfied by a strict $\sigma$-structure.
\end{proof}

Let us also record the following fact.

\begin{exercise}\label{ex:eq-arb-forms}
For any formula~$\varphi(x)$ in a signature with equality, the associated equality axioms entail $x=y\land\varphi(x)\to\varphi(y)$. \emph{Hint:} Either use induction over terms and formulas or combine Exercise~\ref{ex:substitution} with completeness for strict models.
\end{exercise}

\subsection{Consequences of Cut-Free Completeness}\label{sect:cut-free-completeness}

In the present section, we use the completeness theorem to establish several classical results about first-order logic. The~fact that our version of completeness yields cut-free proofs will be relevant in the second part of the section.

The name of the following result is at least partially explained by Exercise~\ref{ex:compactness} and the way in which K\H{o}nig's lemma was used in the proof of completeness.

\begin{theorem}[Compactness]\label{thm:compactness}
Consider a set~$\Theta$ of formulas in some first-order language. If any finite subset of~$\Theta$ is satisfiable, then so is the entire set~$\Theta$.
\end{theorem}
\begin{proof}
If the set $\Theta$ is not satisfiable, then we have~$\Theta\vDash\bot$, so that Theorem~\ref{thm:completeness-AL} yields a proof $\vdash\neg\theta_1,\ldots,\neg\theta_n,\bot$ for a finite collection of formulas~$\theta_i\in\Theta$. But then soundness (see Theorem~\ref{prop:soundness-FO}) entails that $\{\theta_1,\ldots,\theta_n\}$ cannot be satisfiable.
\end{proof}

As an application of compactness, we show the following undefinability result. The latter should not be over-interpreted: one can still use first-order logic to express mathematical statements about finite and infinite objects (cf.~the broadly related discussion after Theorem~\ref{thm:loewenheim-skolem}). We note that the result would be trivial for non-strict models. To make these infinite, it suffices to replace some element of the universe by an infinite equivalence class (conversely to the proof of Theorem~\ref{thm:strict-model}).

\begin{corollary}
Consider a first-order language with equality. There is no set of formulas~$\Theta$ such that, for any strict model~$\mathcal M$, we have $\mathcal M\vDash\Theta$ precisely when the universe of~$\mathcal M$ is finite.
\end{corollary}
\begin{proof}
For each integer~$n\geq 2$, we let $\varphi_n$ be the formula
\begin{equation*}
\exists x_1\ldots\exists x_n\,\textstyle\bigwedge_{1\leq i<j\leq n}\neg x_i=x_j,
\end{equation*}
where the part after the quantifiers abbreviates $\neg x_1=x_2\land\neg x_1=x_3\land\neg x_2=x_3\ldots$. Let us assume that $\Theta$ is satisfied by all finite structures (or at least by arbitrarily large ones). We deduce that $\Theta\cup\{\varphi_n\,|\,n\geq 2\}$ has a strict model. By compactness and Theorem~\ref{thm:strict-model}, it is enough to show that $\Theta\cup\{\varphi_n\,|\,2\leq n\leq N\}$ is satisfiable for an arbitrary~$N\in\mathbb N$. Let us pick a finite~$\mathcal M\vDash\Theta$ with at least~$N$ elements. In~order to get $\mathcal M\vDash\varphi_n$ for~$n\leq N$, it suffices to witness the $x_i$ by $n$ different elements of the universe. Now let $\mathcal N$ be a strict model of $\Theta\cup\{\varphi_n\,|\,n\geq 2\}$. For arbitrary~$n\geq 2$, we can invoke $\mathcal N\vDash\varphi_n$ to conclude that the universe of~$\mathcal N$ has at least~$n$ elements. So $\mathcal N$ is infinite and $\Theta$ does not just have finite models.
\end{proof}

Part~(b) of the following exercise is another classical application of compactness.

\begin{exercise}
(a) Find a first-order formula~$\varphi$ in the language~$\{=\}$ such that we have $\mathcal M\vDash\varphi$ precisely when the universe of $\mathcal M$ has exactly three elements, for any strict model~$\mathcal M$. \emph{Remark:} More generally, you can prove that the same holds when three is replaced by any fixed positive integer.

(b) By a graph (directed with loops) we mean a strict model for the language with equality and another binary relation~$E$ (which determines which `points' from the universe are related by an edge). A graph~$G$ is called connected if any two points~$x,y$ admit a path between them, i.\,e., a sequence of points $x=z_0,\ldots,z_n=y$ (possibly with~$n=0$) such that $Ez_iz_{i+1}$ holds in $G$ for all~$i<n$. Show that there is no set of formulas that is satisfied precisely by the connected graphs. What happens if you focus on finite graphs?
\end{exercise}

The following is a preparation for our next result.

\begin{exercise}\label{ex:Cantor-pairing}
Show that the so-called Cantor pairing function $\pi:\mathbb N^2\to\mathbb N$ with
\begin{equation*}
\pi(\langle m,n\rangle)=n+\sum_{i=0}^{m+n} i=n+\frac{(m+n)\cdot(m+n+1)}{2}
\end{equation*}
is bijective. Use induction to derive that there is a bijection between $\mathbb N^n$ and $\mathbb N$ for each integer~$n\geq 1$. Infer that there is a bijection between $\mathbb N$ and the set $\bigcup_{n\in\mathbb N}\mathbb N^n$ of finite sequences of natural numbers. \emph{Hint:} The inverse of $\pi$ is an enumeration~of the pairs in $\mathbb N^2$ along diagonals in the plane, which starts with
\begin{equation*}
\langle 0,0\rangle,\quad\langle 0,1\rangle,\quad\langle 1,0\rangle,\quad\langle 2,0\rangle,\quad\langle 2,1\rangle,\quad\langle 2,2\rangle.
\end{equation*}
Consider how the lengths of the diagonals will add up.
\end{exercise}

The following result is complemented by an upward version that is proved in the last section of these lecture notes (see Theorem~\ref{thm:up-LS}). In that section we will also prove a downward version for uncountable signatures (see Theorem~\ref{thm:downward-LS-uncount}).

\begin{theorem}[Downward L\"owenheim-Skolem Theorem~\cite{loewenheim,skolem}]\label{thm:loewenheim-skolem}
Let~$\mathsf T$ be a theory in a first-order language that is countable (as ensured by Convention~\ref{conv:countable}). If $\mathsf T$ has some model, then it has a model with countable universe.
\end{theorem}
\begin{proof}
By the proofs of Proposition~\ref{prop:completeness-FO} and Theorem~\ref{thm:completeness-AL}, we get a model of~$\mathsf T$ where the universe consists of terms. The latter are built from variables $x_0,x_1,\ldots$ as well as function symbols $f_0,f_1,\ldots$, which we combine into a single set $S:=\{x_n\,|\,n\in\mathbb N\}\cup\{f_n\,|\,n\in\mathbb N\}$ that is countable due to $S\to\mathbb N$ with $x_n\mapsto 2n$ and $f_n\mapsto 2n+1$. Terms can be seen as finite sequences over~$S$, which means that the universe of our model may be identified with a subset of $\bigcup_{n\in\mathbb N}\mathbb N^n$. We can conclude by the previous exercise. Let us note that, for theories with equality, the proof of Theorem~\ref{thm:strict-model} yields a strict model that is still countable (but may now be finite).
\end{proof}

The previous theorem is sometimes referred to as the L\"owenheim-Skolem paradox. This~is because there are consistent axiom systems -- e.\,g.~in set theory -- that prove the existence of uncountable objects. By the theorem, these must occur inside some countable model~$\mathcal M$. This is possible because being countable is not an `absolute' property: It may be that an object is uncountable from the viewpoint of~$\mathcal M$ -- which means that $\mathcal M$ contains no injection from that object into the integers -- while it is actually countable because such an injection exists outside of~$\mathcal M$.

To prove the remaining results of this section, we exploit the fact that our proof of completeness entails cut elimination.

\begin{theorem}[Craig interpolation~\cite{craig-interpolation}]\label{thm:craig-interpolation}
Given an implication $\varphi_0\to\varphi_1$ that is logically valid, one can find a formula $\psi$ such that $\varphi_0\to\psi$ and $\psi\to\varphi_1$ are valid and any relation symbol in~$\psi$ occurs in both~$\varphi_0$ and~$\varphi_1$ (where we allow $\psi\in\{\bot,\top\}$ even when $\varphi_0$ and $\varphi_1$ share no relation symbols).
\end{theorem}
\begin{proof}
We will use induction over cut-free proofs to establish the following claim: Given $\vdash_0\Gamma,\Delta$, we get $\vdash_0\Gamma,\psi$ and $\vdash_0\Delta,\neg\psi$ for some formula~$\psi$ such that any relation symbol in~$\psi$ occurs in both~$\Gamma$ and~$\Delta$ (still allowing $\psi\in\{\bot,\top\})$. Before we carry out the induction, we show how to conclude: Under the assumptions from the theorem, cut-free completeness yields~$\vdash_0\neg\varphi_0,\varphi_1$. Once the claim above has been proved, we obtain $\vdash_0\neg\varphi_0,\psi$ and $\vdash_0\varphi_1,\neg\psi$ for a suitable formula~$\psi$. Soundness yields $\neg\varphi_0\lor\psi$ and $\varphi_1\lor\neg\psi$, which amounts to~$\varphi_0\to\psi$ and $\psi\to\varphi_1$.

In the base case of the indicated proof by induction, the sequent~$\Gamma,\Delta$ contains an atomic formula~$\theta$ and its negation. First assume that both these formulas are contained in~$\Gamma$. We can then take~$\psi$ to be~$\bot$, since we get $\vdash_0\Gamma,\bot$ by weakening while $\vdash_0\Delta,\top$ holds because $\top$ is logically valid. Similarly, we can take $\psi$ to be~$\top$ when both $\theta$ and $\neg\theta$ are contained in~$\Delta$. Now assume that $\theta$ occurs in~$\Gamma$ while $\neg\theta$~occurs in~$\Delta$. We then have $\vdash_0\Gamma,\neg\theta$ as well as $\vdash_0\Delta,\theta$. Also, the relation symbols in~$\theta$ occur in both $\Gamma$ and~$\Delta$. Hence we can take $\psi$ to be~$\neg\theta$ (recall that $\neg\neg\theta$ and $\theta$ coincide due to our treatment of formulas in negation normal form). When~$\neg\theta$ occurs in~$\Gamma$ while~$\theta$ occurs in~$\Delta$, we can take~$\psi$ to be~$\theta$.

Next, we assume that $\Gamma$ contains a formula $\rho_0\land\rho_1$ and that $\vdash_0\Gamma,\Delta$ was derived from the premises $\vdash_0\Gamma,\rho_i,\Delta$ for~$i\in\{0,1\}$. The induction hypothesis provides
\begin{equation*}
\vdash_0\Gamma,\rho_i,\psi_i\quad\text{and}\quad\vdash_0\Delta,\neg\psi_i,
\end{equation*}
where the relation symbols in $\psi_i$ occur in both~$\Gamma,\rho_i$ and~$\Delta$. We infer $\vdash_0\Delta,\neg\psi_0\land\neg\psi_1$ and $\vdash_0\Gamma,\rho_i,\psi_0\lor\psi_1$ for each~$i<2$. The latter yields $\vdash_0\Gamma,\psi_0\lor\psi_1$, since $\rho_0\land\rho_1$ occurs in~$\Gamma$. This does also ensure that $\Gamma$ contains any relation symbol in~$\rho_i$. We can thus take $\psi$ to be $\psi_0\lor\psi_1$ (recall that $\neg\psi_0\land\neg\psi_1$ and $\neg(\psi_0\lor\psi_1)$ denote the same formula in negation normal form). In~case the last rule has introduced a conjunction $\psi_0\land\psi_1$ in the sequent~$\Delta$, we can similarly choose $\psi$ of the form $\psi_0\land\psi_1$.

We continue with the case where $\Gamma$ contains a formula~$\rho_0\lor\rho_1$ that was introduced with premise $\vdash_0\Gamma,\rho_i,\Delta$ for some~$i\in\{0,1\}$. The induction hypothesis provides~$\psi$ with $\vdash_0\Gamma,\rho_i,\psi$ and $\vdash_0\Delta,\neg\psi$. We use the same rule to infer $\vdash_0\Gamma,\psi$. In this case, we can thus take the same $\psi$ to complete the induction step. The same applies when the last rule introduces a disjunction $\rho_0\lor\rho_1$ in~$\Delta$.

In the penultimate case, the sequent $\Gamma$ contains a formula $\forall x\,\rho$ that has been introduced with premise $\vdash_0\Gamma,\rho[x/y],\Delta$, where the variable $y$ is not free in~$\Gamma,\Delta$. Consider $\vdash_0\Gamma,\rho[x/y],\psi'$ and $\vdash_0\Delta,\neg\psi'$ as provided by the induction hypothesis. We cannot re-introduce $\forall x\,\rho$ by the same rule immediately, as $y$ may be free in~$\psi'$. Since $y$ is not free in~$\Delta$, however, we may infer $\vdash_0\Delta,\forall y\,\neg\psi'$ and also $\vdash_0\Gamma,\rho[x/y],\exists y\,\psi'$. As $\forall x\,\rho$ occurs in~$\Gamma$ and $y$ is not free in $\Gamma,\exists y\,\psi'$, we can now conclude $\vdash_0\Gamma,\exists y\,\psi'$. So in this case, we have completed the induction step with $\exists y\,\psi'$ at the place of~$\psi$. When $\forall x\,\rho$ is contained in~$\Delta$, we can conclude similarly, with $\psi$ of the form $\forall y\,\psi'$.

Finally, assume that the last rule has introduced a formula~$\exists x\,\rho$ that occurs in~$\Gamma$. This rule has premise $\vdash_0\Gamma,\rho[x/t],\Delta$ for some term~$t$. We note that $\Gamma$ contains all relation symbols (but not necessarily all function symbols) that occur in~$\rho[x/t]$. The induction hypothesis yields a suitable $\psi$ with $\vdash_0\Delta,\neg\psi$ and $\vdash_0\Gamma,\rho[x/t],\psi$, from which we get $\vdash\Gamma,\psi$. When $\exists x\,\rho$ lies in~$\Delta$, one concludes in the same way.
\end{proof}

The previous result is the first that relies on proofs being cut-free. To explain why the argument would break down in the presence of cuts, we note that the cut rule is the only one where the premises contain a formula that is unrelated to any formula in the conclusion. We have no control over the relation symbols in this formula in the premise. The following is an important consequence of Craig~interpolation.

\begin{theorem}[Beth definability~\cite{beth-definability}]
Let~$\mathsf T$ be a theory in a language $\mathcal L_P=\mathcal L\cup\{P\}$ with an $n$-ary predicate symbol~$P$ that does not occur in~$\mathcal L$. The following notions of explicit and implicit definability are equivalent:
\begin{enumerate}[label=(\roman*)]
\item There is an $\mathcal L$-formula~$\psi$ such that $\mathsf T$ entails $\forall x_1\ldots\forall x_n(Px_1\ldots x_n\leftrightarrow\psi)$.
\item For any $\mathcal L$-structure, there is at most one interpretation of~$P$ that extends it into an $\mathcal L_P$-structure that validates~$\mathsf T$.
\end{enumerate}
\end{theorem}
\begin{proof}
When statement~(i) holds, the only option in~(ii) is to interpret~$P$ as the set of tuples~$(a_1,\ldots,a_n)$ such that $\psi$ is validated when we assign~$a_i$ to $x_i$. The crucial task is to prove the converse implication from~(ii) to~(i). Let us consider the languages $\mathcal L_i=\mathcal L\cup\{P_i\}$ for~$i\in\{0,1\}$, where $P_0$ and~$P_1$ are different \mbox{$n$-ary} predicate symbols that do not occur on~$\mathcal L$. We write $\mathsf T_i$ for the~$\mathcal L_i$-theory that~results from~$\mathsf T$ when we replace each occurrence of~$P$ by~$P_i$. Given~(ii), we see that $P_0x_1\ldots x_n\leftrightarrow P_1x_1\ldots x_n$ is a logical consequence of the theory~$\mathsf T_0\cup\mathsf T_1$ (over the language $\mathcal L_0\cup\mathcal L_1$). Here we may replace $\mathsf T_0\cup\mathsf T_1$ by $\varphi_0\land\varphi_1$ for $\mathcal L_i$-sentences~$\varphi_i$ that are entailed by the~$\mathsf T_i$, due to compactness (cf.~Exercise~\ref{ex:completeness}). It follows that
\begin{equation*}
\varphi_0\land P_0x_1\ldots x_n\to(\varphi_1\to P_1x_1\ldots x_n)
\end{equation*}
is logically valid. The point is that $P_0$ and $P_1$ only occur in the premise and conclusion, respectively. We can thus invoke Craig interpolation to get an $\mathcal L$-formula~$\psi$ such that $\varphi_0\land P_0x_1\ldots x_n\to\psi$ and $\psi\to(\varphi_1\to P_1x_1\ldots x_n)$ is valid. It follows that $\mathsf T_0$ entails $P_0x_1\ldots x_n\to\psi$ while $\mathsf T_1$ entails $\psi\to P_1x_1\ldots x_n$. In each case, we may change~$P_i$ back into~$P$. We then learn that $\mathsf T$ entails $Px_1\ldots x_n\leftrightarrow\psi$, as in~(i).
\end{proof}

For the final result of this section, we introduce the following notion.

\begin{definition}
A formula in negation normal form is called a $\Pi$-formula if $\exists$ does not occur in it. It is called a $\Sigma$-formula if it does not contain~$\forall$ (but possibly~$\exists$).
\end{definition}

The formula $\theta(t_0)\lor\ldots\lor\theta(t_n)$ in the conclusion of the following theorem is known as a Herbrand disjunction.

\begin{theorem}[Herbrand's theorem]\label{thm:herbrand}
Assume that we have $\mathsf T\vDash\exists x\,\theta(x)$ where~$\mathsf T$ is a set of $\Pi$-formulas and $\theta$ is a $\Sigma$-formula, all in some language~$\mathcal L$. Then there is a finite collection of~$\mathcal L$-terms~$t_0,\ldots,t_n$ such that we have $\mathsf T\vDash\theta(t_0)\lor\ldots\lor\theta(t_n)$.
\end{theorem}
\begin{proof}
We will use induction over cut-free proofs to show the following claim: When we have $\vdash_0\Gamma,\exists x\,\theta(x)$ for a set $\Gamma,\exists x\,\theta(x)$ that consists of~$\Sigma$-formulas, then we obtain $\vdash_0\Gamma,\theta(t_1),\ldots,\theta(t_n)$ for suitable terms~$t_1,\ldots,t_n$. To see that this yields the theorem, we note that $\mathsf T\vDash\exists x\,\theta(x)$ entails $\vdash_0\neg\varphi_1,\ldots,\neg\varphi_m,\exists x\,\theta(x)$ for some formulas $\varphi_i$ in~$\mathsf T$, due to cut-free completeness. Here $\neg\varphi_i$ is a $\Sigma$-formula since $\varphi_i$ is a $\Pi$-formula (cf.~Definition~\ref{def:nega-pred-log}). In view of soundness, the claim above will now yield the desired result.

For the induction step, we distinguish cases according to the last rule that was used to derive $\vdash_0\Gamma,\exists x\,\theta(x)$, similarly to the proof of Theorem~\ref{thm:craig-interpolation}. We will only provide details for some representative cases and leave the remaining ones to the reader. In the crucial case, the last rule was used to introduced the formula~$\exists x\,\theta(x)$ with premise~$\vdash_0\Gamma,\exists x\,\theta(x),\theta(t)$. Given that $\theta(t)$ is a $\Sigma$-formula, we can use the induction hypothesis to get~$\vdash_0\Gamma,\theta(t_1),\ldots,\theta(t_n),\theta(t)$. This completes the induction step with one additional term~$t$ added to our list.

Next, consider the case where we have $\vdash_0\Gamma,\exists x\,\theta(x)$ since the sequent $\Gamma,\exists x\,\theta(x)$ contains some atomic formula and its negation. These two formulas must lie in~$\Gamma$ (because $\exists x\,\theta(x)$ is no literal). Hence we obtain $\vdash_0\Gamma$, which means that the claim holds for the empty list of terms.

In any other case, the last rule was used to introduce some formula in~$\Gamma$. Crucially, the assumption that $\Gamma$ consists of $\Sigma$-formulas ensures that this formula cannot be of the form $\forall y\,\rho$, where we would face a conflict with the variable condition if the relevant variable occurs in one of the terms~$t_i$. As an example, we consider the case where $\Gamma$ contains a formula $\rho_0\lor\rho_1$ that was introduced with premise~$\vdash_0\Gamma,\rho_i,\exists x\,\theta(x)$ for some~$i<2$. Given that $\rho_0\lor\rho_1$ is a $\Sigma$-formula, the same holds for~$\rho_i$. We can thus use the induction hypothesis to obtain~$\vdash_0\Gamma,\rho_i,\theta(t_1),\ldots,\theta(t_n)$. An application of the same rule will now yield $\vdash_0\Gamma,\theta(t_1),\ldots,\theta(t_n)$.
\end{proof}

Note that it is, once again, crucial to work with proofs that are cut-free, since the premise of a cut might not consist of~$\Sigma$-formulas.

\begin{exercise}
(a) Give an example to show that we cannot always take~$n=0$ in Theorem~\ref{thm:herbrand}, i.\,e., that it is necessary to admit Herbrand disjunctions that involve more than one term.

(b) Provide details for the remaining cases in the proof of Herbrand's theorem.
\end{exercise}

To extend Herbrand's theorem beyond the realm of $\Pi$- and $\Sigma$-formulas, one can exploit the fact that any prenex formula~$\forall x_1\exists y_1\ldots\forall x_n\exists y_n\,\theta(x_1,y_1,\ldots,x_n,y_n)$ is equivalent to a so-called second-order formula
\begin{equation*}
\exists f_1\ldots\exists f_n\forall x_1\ldots\forall x_n\,\theta\big(x_1,f_1(x_1),\ldots,x_n,f_n(x_1,\ldots,x_n)\big).
\end{equation*}
The latter is called the Skolem normal form of the given first-order formula and the~$f_i$ are called Skolem functions. We can avoid the quantification over Skolem functions if we introduce function symbols that represent them. In this way, an arbitrary theory can be replaced by a collection of $\Pi$-formulas in an extended language. Also, a prenex formula~$\varphi$ of the form~$\exists x_1\forall y_1\ldots\exists x_n\forall y_n\,\theta(x_1,y_1,\ldots,x_n,y_n)$ is equivalent to
\begin{equation*}
\forall f_1\ldots\forall f_n\exists x_1\ldots\exists x_n\,\theta\big(x_1,f_1(x_1),\ldots,x_n,f_n(x_1,\ldots,x_n)\big),
\end{equation*}
which is a so-called Herbrand normal form. When $\varphi$ appears as a conclusion~$\mathsf T\vDash\varphi$ like in Herbrand's theorem, we can omit the universally quantified Herbrand functions in favour of free variables or fresh function symbols of an extended language.

\begin{exercise}
(a) Consider functions $f,g:M\to M$ such that $f\circ g\circ f$ is bijective. Show that $g$ must be surjective (in fact bijective).

(b) Conclude that the formula $\forall x\,g(x)\neq 0\to\exists x\, f(g(f(x)))\neq x$ is logically valid (where $s\neq t$ abbreviates $\neg s=t$ in the language with equality and a constant symbol~$0$ as well as unary function symbols~$f$ and~$g$).

(c) Find terms~$s_0,\ldots,s_m$ and $t_0,\ldots,t_n$ (in the language~$\{0,f,g\}$) such that
\begin{equation*}
\textstyle\bigwedge_{i\leq m} g(s_i)\neq 0\to\textstyle\bigvee_{j\leq n}f(g(f(t_j)))\neq t_j
\end{equation*}
is logically valid. Note that the existence of such terms is guaranteed by Herbrand's theorem. Try to make it transparent how the terms $s_i$ and~$t_j$ that you have found can be extracted from the proof given in~(a) and~(b).

\emph{Remark:} The example in this exercise is due to Ulrich Berger (as noted in~\cite{kohlenbach-gazette}).
\end{exercise}

Our semantic proof of cut-elimination (see Corollary~\ref{cor:cut-elim-PL}) provides no information on complexity. More syntactic arguments (see e.\,g.~\cite{buss-introduction-98}) show that one can obtain a cut-free proof that is at most super-exponentially longer than a proof with cuts, where super-exponentiation is the function $n\mapsto 2_n$ with $2_0=1$ and $2_{n+1}=2^{2_n}$. In the proof of Theorem~\ref{thm:herbrand}, the length of the Herbrand disjunction is bounded by the size of the cut-free proof, so that it also admits a super-exponential upper bound. The following exercise shows that these extremely large bounds are essentially best possible. This means that cuts are crucial to get feasible proofs in practice, even though they can be eliminated in principle. In part~(a) of the following exercise, the reader will see that cuts are connected to auxiliary notions or lemmata (the formulas~$\varphi_n$). One might say that a creative idea is needed to find the appropriate lemmata, since these can contain notions that do not appear in the theorem that one aims to prove. This is related to the fact that the cut rule involves a premise that has no particular connection with the conclusion.

\begin{exercise}\label{ex:Herbrand-supexp}
Consider the signature with equality that includes constant symbols~$0$ and $1$, a unary function symbol~$2^{-}$, a binary function symbol~$+$ (written infix) and a unary relation symbol~$I$. Let $T_0$ be the theory that consists of the equality axioms and the universal closures of the following formulas:
\begin{equation*}
x+(y+z)=(x+y)+z,\quad y+0=y,\quad 2^0=1,\quad 2^x+2^x=2^{x+1},\quad I(0).
\end{equation*}
Let $T$ be the extension of~$T_0$ by the formula~$\forall x(I(x)\to I(x+1))$. We write $2_n$ for the terms defined by $2_0=1$ and $2_{n+1}=2^{2_n}$.

(a) Show that we have $T\vdash I(2_n)$ with proofs of size linear in~$n$. \emph{Hint:} Consider the formulas $\varphi_n$ given by $\varphi_0(x)=I(x)$ and $\varphi_{n+1}(x)=\forall y(\varphi_n(y)\to\varphi_n(y+2^x))$. Use recursion on~$m$ to construct proofs of
\begin{equation*}
\varphi_m(0)\land\forall x\big(\varphi_m(x)\to\varphi_m(x+1)\big).
\end{equation*}
Then derive $I(2_n)$ from the formulas~$\varphi_m(0)$ for~$m\leq n+1$. It suffices if you give a rough description of the formal proofs, based on an intuitive idea of proof size (but do consider where the cut rule is used).

(b) Since the result of (a) is equivalent to $T_0\vdash\exists x\big((I(x)\to I(x+1))\to I(2_n)\big)$, Herbrand's theorem guarantees the existence of terms~$t_0,\ldots,t_{N-1}$ with
\begin{equation*}
T_0\cup\{I(t_i)\to I(t_i+1)\,|\,i<N\}\vdash I(2_n).
\end{equation*}
Show that the latter can only hold if we have $N\geq 2_n$.

\emph{Remark:} Our formulation of this exercise is adopted from a course by Ulrich~Kohlenbach, while the original result is due to Vladimir~Orevkov~\cite{orevkov} and Richard Statman~\cite{statman}.
\end{exercise}

\section{G\"odel's Incompleteness Theorems}\label{sect:goedel}

In this section, we present the famous incompleteness theorems of Kurt G\"odel. These theorems are concerned with axiom systems that are supposed to provide a foundation for all of mathematics or at least for a substantial part of it. One example for such an axiom system is Zermelo-Fraenkel set theory, in which most of contemporary mathematics can be formalized (see Section~\ref{sect:set-theory}). We choose a~different example for our proof of G\"odel's theorems, namely Peano arithmetic. This axiom system captures essential properties of the natural numbers, with induction as the central proof principle. Since other finite objects such as graphs can be encoded by natural numbers (as we shall see), one could say that Peano arithmetic provides a foundation for finite mathematics. It now seems natural to ask: can any question of finite mathematics be decided in Peano arithmetic, i.\,e., does the latter prove~$\varphi$ or~$\neg\varphi$ for any sentence over its signature? By G\"odel's first incompleteness theorem, the answer is negative. The proof of this theorem is quite general and transfers to other axiom systems in which the natural numbers can be represented. In particular, Zermelo-Fraenkel set theory cannot decide all questions of mathematics and -- which may be more surprising -- not even all questions about the natural numbers.

Next, one might ask whether there are statements of finite mathematics that can be proved by infinitary but not by finitary means, or more specifically, whether some concrete statement about the natural numbers is provable in Zermelo-Fraenkel set theory but not in Peano arithmetic. G\"odel's second incompleteness theorem entails that the answer to this question is positive. More specifically, it asserts that foundational theories such as Peano arithmetic cannot prove their own consistency. Given that proofs are finite objects, a consistency statement can indeed be seen as an assertion about the natural numbers. The significance of G\"odel's second incompleteness theorem becomes much clearer when it is seen in the context of Hilbert's program, which we discuss at the very end of the present section.

In the following Subsection~\ref{subsect:PA}, we introduce Peano arithmetic and related axiom systems. We place particular emphasis on a result that is known as $\Sigma$-completeness. Subsection~\ref{sect:prim-rec} is devoted to the notion of primitive recursive function. The latter plays a crucial role in Subsection~\ref{subsect:goedel}, where we prove Tarski's result on the so-called undefinability of truth as well as G\"odel's incompleteness theorems. In the case of the second incompleteness theorem, we show that the result holds for any axiom system that satisfies some reasonable assumptions known as the Hilbert-Bernays conditions. We will try to make it plausible that the latter hold for Peano arithmetic but do not give a detailed proof of this fact.

As in the previous section, we conclude the introduction with a pointer to complementary reading. G\"odel's theorems were proved in a historical context where concrete computations were more and more replaced by abstract methods. The reader can follow this process in a fascinating case study on Dirichlet's work about primes in arithmetic progressions, due to Jeremy Avigad and Rebecca Morris~\cite{avigad-morris}.

\subsection{First-Order Arithmetic}\label{subsect:PA}

In this subsection, we present the language and some important theories of first-order arithmetic.

\begin{definition}\label{def:signature-fa}
The signature of first-order arithmetic consists of a constant~$0$, a function symbol~$S$ (successor) of arity one (unary), two function symbols~$+$ and~$\times$ of arity two (binary) as well as binary relation symbols~$=$ and~$\leq$ (all written infix). The numerals are the terms~$\overline n$ that are recursively given by $\overline 0:=0$ and $\overline{n+1}:=S\overline n$. We write $\mathbb N$ for the so-called standard model with the natural numbers as universe and the expected interpretations of the given signature (where $\times^{\mathbb N}$ is multiplication and we have $S^{\mathbb N}(n)=n+1$). A formula~$\varphi$ is called true if we have $\mathbb N\vDash\varphi$.
\end{definition}

The following restrictions on the quantifier complexity will play a crucial role.

\begin{definition}\label{def:Sigma-formula}
We say that a formula is~$\Sigma$ or a $\Sigma$-formula if it can be generated as follows (recall that we only consider formulas in negation normal form):
\begin{enumerate}[label=(\roman*)]
\item Any literal is a $\Sigma$-formula.
\item If $\varphi$ and $\psi$ are $\Sigma$-formulas, then so are $\varphi\land\psi$ and $\varphi\lor\psi$.
\item Given a $\Sigma$-formula $\varphi$, a term~$t$ and a variable~$x\notin\fv(t)$, we obtain further $\Sigma$-formulas $\forall x(x\leq t\to\varphi)$ and $\exists x(x\leq t\land\varphi)$, which will be abbreviated as $\forall x\leq t\,\varphi$ and $\exists x\leq t\,\varphi$, respectively (bounded quantification).
\item If $\varphi$ is a $\Sigma$-formula, then so is~$\exists x\,\varphi$.
\end{enumerate}
By a $\Delta_0$-formula or bounded formula, we mean a $\Sigma$-formula that can be generated without the use of clause~(iv). A formula $\exists x\,\theta$ with bounded~$\theta$ is called $\Sigma_1$. We say that $\psi$ is $\Pi$ or $\Pi_1$, respectively, if $\neg\psi$ is $\Sigma$ or $\Sigma_1$ (see Definition~\ref{def:nega-pred-log} for our treatment of negation). By a $\Delta$-formula (or $\Delta_1$-formula), we mean a formula that is equivalent to both a $\Sigma$-formula and a $\Pi$-formula (or a $\Sigma_1$-formula and a \mbox{$\Pi_1$-formula}, respectively). We say that a formula is $\Delta$ or $\Delta_1$ in some theory~$\mathsf T$ if the indicated equivalences can be proved in~$\mathsf T$. Finally, we speak of, e.\,g., a $\Sigma$-sentence in order to refer to a $\Sigma$-formula without free variables.
\end{definition}

More intuitively, a formula is $\Sigma$ or~$\Pi$ if it only contains bounded occurrences of universal or existential quantifiers, respectively. The terminology can be extended to formulas with alternating occurrences of quantifiers (e.\,g., one says that $\forall x\exists y\,\theta$ is $\Pi_2$ when $\theta$ is bounded), but this is not needed for the following. Intuitively, a $\Sigma$-formula is one that can be verified (but not falsified) by a computation that searches for witnesses to the existential quantifiers (and splits into finitely many subcomputations when a bounded universal quantifier is encountered). This viewpoint will be made official in Section~\ref{sect:computability}. Since such a computation is finite, one should be able to check it in a rather weak object theory. We make this idea precise in Theorem~\ref{thm:Sigma-completeness} below, for which we now introduce the required ingredients.

\begin{definition}\label{def:Q-PA}
Robinson arithmetic, denoted~$\mathsf Q$, is the theory in the signature of first-order arithmetic that consists of the equality axioms (see Definition~\ref{def:equality}) and the universal closures of the following formulas:
\begin{gather*}
Sx=Sy\to x=y,\qquad x\neq 0\leftrightarrow\exists y\,x=Sy,\qquad x\leq y\leftrightarrow\exists z\,z+x=y,\\
x+0=x,\qquad x+Sy=S(x+y),\qquad x\times 0=0,\qquad x\times Sy=(x\times y)+x.
\end{gather*}
The induction axiom~$\mathsf I\varphi$ for a formula~$\varphi(x)$, which may contain further free variables (as so-called parameters), is the universal closure of
\begin{equation*}
\varphi(0)\land\forall x\big(\varphi(x)\to\varphi(Sx)\big)\to\forall x\,\varphi(x).
\end{equation*}
Peano arithmetic, denoted by~$\mathsf{PA}$, is the extension of~$\mathsf Q$ by the sentences~$\mathsf I\varphi$ for all formulas~$\varphi(x)$ in the signature of first-order arithmetic (axiom schema of induction). The theory~$\mathsf{I\Sigma}_1$ consists of~$\mathsf Q$ and the axioms $\mathsf I\varphi$ for all $\Sigma_1$-formulas~$\varphi(x)$.
\end{definition}

In the following exercise, part~(a) confirms that $\mathsf Q$ is rather weak in certain~respects. This is a desirable feature, as a weaker base theory will yield a more general version of Theorem~\ref{thm:Sigma-completeness} and later of the first incompleteness theorem. Part~(b) of the exercise is crucial in connection with the computational verification mentioned above: it will allow us to view a bounded quantifier as a finite case distinction.

\begin{exercise}\label{ex:Q}
(a) Show that $\mathsf{I\Sigma}_1$ but not~$\mathsf Q$ proves the formula~$x+y=y+x$, which asserts that addition is commutative. \emph{Hint:} For the unprovability, show that $\mathsf Q$ has a model with universe $\mathbb N\cup\{\infty_0,\infty_1\}$ and $\infty_0+\infty_1=\infty_0\neq\infty_1=\infty_1+\infty_0$.

(b) Prove that, for an arbitrary $n\in\mathbb N$, we have
\begin{equation*}
\mathsf Q\vdash x\leq\overline n\rightarrow x=\overline 0\lor\ldots\lor x=\overline n.
\end{equation*}
\emph{Hint:} First show that we have $\mathsf Q\vdash Sx\leq Sy\to x\leq y$. Then prove the claim by induction on $n$. Note that the induction takes place in the meta-theory and not within the theory~$\mathsf Q$ (just like the quantification over~$n$), so that no induction axiom is required. In other words, the task is to find a separate $\mathsf Q$-proof for each~$n$, not a single proof that would cover all numerals.

(c) Show that we have $\mathsf Q\vdash x\leq\overline n\lor\overline n\leq x$ for any~$n\in\mathbb N$. \emph{Hint:} First prove that $\overline n\leq x$ entails the disjunction of $x=\overline n$ and $\overline{n+1}\leq x$. To establish the latter, show that we have $\mathsf Q\vdash Sx+\overline n=x+\overline{n+1}$. 
\end{exercise}

The following reveals a characteristic feature of the theory~$\mathsf{I\Sigma}_1$. Let us note that the stronger induction principles that are available in Peano arithmetic will not play a role in the present section. In Section~\ref{sect:goodstein} we will see a theorem that is provable in~$\mathsf{PA}$ but not in~$\mathsf{I\Sigma}_1$. 

\begin{exercise}\label{ex:sigma-sigma_1}
(a) Show that, for any $\Delta_0$-formula~$\theta$, we have
\begin{equation*}
\mathsf{I\Sigma}_1\vdash\forall x\leq v\exists y\,\theta\to\exists w\forall x\leq v\exists y\leq w\,\theta.
\end{equation*}

(b) Derive that any $\Sigma$-formula~$\varphi$ admits a $\Sigma_1$-formula~$\varphi'$ with $\mathsf{I\Sigma}_1\vdash\varphi\leftrightarrow\varphi'$. \emph{Remark:} You can now infer that~(a) remains valid when~$\theta$ is a $\Sigma$-formula.
\end{exercise}

For a term~$t$ and a variable assignment~$\eta$, the value of $t^{\mathbb N,\eta}$ does only depend on values of~$\eta$ for variables that are free in~$t$. When~$t$ is closed (i.\,e., has no free variables), we define $t^{\mathbb N}$ by stipulating that~$t^{\mathbb N}=t^{\mathbb N,\eta}$ holds for any assignment~$\eta$.

\begin{lemma}\label{lem:term-value-Q}
We have $\mathsf Q\vdash t=\overline{t^\mathbb N}$ for any closed term~$t$.
\end{lemma}
\begin{proof}
We argue by induction on the build-up of~$t$. When~$t$ is the constant~$0$, the numeral $\overline{t^{\mathbb N}}$ coincides with~$t$, so that we can conclude by an equality axiom. Next, we consider the case where $t$ has the form~$St_0$. Here we have $\overline{t^{\mathbb N}}=\overline{t_0^{\mathbb N}+1}=S\overline{t_0^{\mathbb N}}$. As the induction~hypothesis yields~$\mathsf Q\vdash t_0=\overline{t_0^\mathbb N}$, we can again use an equality axiom to conclude. Similarly, the cases where $t$ has the form~$t_0+t_1$ and $t_0\times t_1$ reduce to
\begin{equation*}
\mathsf Q\vdash\overline{m+n}=\overline m+\overline n\qquad\text{and}\qquad\mathsf Q\vdash\overline{m\times n}=\overline m\times\overline n.
\end{equation*}
We prove the former by induction on~$n$ (again in the meta-theory). The base case and induction step are covered by the following equalities in~$\mathsf Q$:
\begin{gather*}
\overline{m+0}=\overline m=\overline m+0=\overline m+\overline 0,\\
\overline{m+(n+1)}=\overline{(m+n)+1}=S\overline{m+n}=S(\overline m+\overline n)=\overline m+S\overline n=\overline m+\overline{n+1}.
\end{gather*}
Here the penultimate equation in each line relies on an axiom of~$\mathsf Q$, while the other equations are covered by the definitions, the induction hypothesis and the equality axioms. The claim about multiplication can now be derived in a similar way.
\end{proof}

As a final ingredient for the aforementioned verification of~$\Sigma$-formulas, we consider the special case of a literal.

\begin{lemma}
For any~$m,n\in\mathbb N$ we have
\begin{align*}
m=n\,&\Rightarrow\,\mathsf Q\vdash\overline m=\overline n,\,& m\neq n\,&\Rightarrow\,\mathsf Q\vdash\overline m\neq\overline n,\\
m\leq n\,&\Rightarrow\,\mathsf Q\vdash\overline m\leq\overline n,\,& m\not\leq n\,&\Rightarrow\,\mathsf Q\vdash\overline m\not\leq\overline n.
\end{align*}
\end{lemma}
\begin{proof}
The first claim is covered by the equality axioms. Since the latter include symmetry, we may assume $m>n$ rather than $m\neq n$ in the second claim. Let us write $m=m_0+1$, so that $\overline m$ is the same term as $S\overline{m_0}$. We get $\mathsf Q\vdash\exists y\,\overline m=Sy$ and hence $\mathsf Q\vdash\overline m\neq 0$, due to the second axiom in Definition~\ref{def:Q-PA}. This covers the case where~$n$ is zero. In the remaining case, we write $n=n_0+1$ to get $\overline n=S\overline{n_0}$. Arguing by induction, we may assume $\mathsf Q\vdash\overline{m_0}\neq\overline{n_0}$. We can now conclude by the first axiom listed in Definition~\ref{def:Q-PA}. For the third claim, we note that $m\leq n$ entails $k+m=n$ for some~$k\in\mathbb N$. Using the previous lemma, we can infer $\mathsf Q\vdash\overline k+\overline m=\overline n$. Now the third axiom from Definition~\ref{def:Q-PA} yields $\mathsf Q\vdash\overline m\leq\overline n$. For the final claim, we note that $m\not\leq n$ entails $\mathsf Q\vdash\overline m\neq\overline 0\land\ldots\land\overline m\neq\overline n$, thanks to the second claim. We now get $\mathsf Q\vdash\overline m\not\leq\overline n$ by part~(b) of Exercise~\ref{ex:Q}.
\end{proof}

Let us now prove the result that was promised after Definition~\ref{def:Sigma-formula}.

\begin{theorem}[$\Sigma$-completeness]\label{thm:Sigma-completeness}
We have $\mathsf Q\vdash\varphi$ for any true $\Sigma$-sentence~$\varphi$.
\end{theorem}
\begin{proof}
We use induction over the number of connectives and quantifiers in~$\varphi$. When the latter is a literal, the claim holds by the two previous lemmata. Now assume that~$\varphi$ has the form~$\varphi_0\land\varphi_1$. Given that~$\varphi$ is true, the same holds for each~$\varphi_i$. We inductively get $\mathsf Q\vdash\varphi_i$ and thus $\mathsf Q\vdash\varphi_0\land\varphi_1$. The case of a disjunction is treated in a similar way. Now assume that~$\varphi$ has the form~$\forall x\leq t\,\theta(x)$. Working in~$\mathsf Q$, we consider an arbitrary~$x\leq t$. By Lemma~\ref{lem:term-value-Q}, we may assume that~$t$ is a numeral~$\overline n$. Due to part~(b) of Exercise~\ref{ex:Q}, we get $x=\overline m$ for some~$m\leq n$. To conclude~$\theta(x)$, we use the induction hypothesis and Exercise~\ref{ex:eq-arb-forms}. Finally, let us assume that $\varphi$ has the form~$\exists x\,\theta(x)$, where the quantifier can be bounded or unbounded. Using the induction hypothesis, we find an $n\in\mathbb N$ with $\mathsf Q\vdash\theta(\overline n)$, which yields~$\mathsf Q\vdash\varphi$.
\end{proof}

From G\"odel's incompleteness theorems (see Section~\ref{subsect:goedel}), it will follow that the theorem above cannot be extended to~$\Pi$-sentences. The notion of consistency will play a central role in the second incompleteness theorem. In the following, we show that it is not just a minimal condition but has substantial positive implications for the truth of $\Pi$-formulas.

\begin{definition}
A theory $\mathsf T$ in the signature of first-order arithmetic is $\Sigma$-sound or $\Pi$-sound, respectively, if $T\vdash\varphi$ entails $\mathbb N\vDash\varphi$ for any $\Sigma$-sentence or~$\Pi$-sentence.
\end{definition}

As a consequence of $\Sigma$-completeness, we get the following.

\begin{corollary}\label{cor:con-pi-sound}
A theory~$\mathsf T\supseteq\mathsf Q$ in the signature of first-order arithmetic is~consistent precisely if it is~$\Pi$-sound.
\end{corollary}
\begin{proof}
If $\mathsf T$ is inconsistent, it proves any false $\Pi$-formula (such as $\overline 0=\overline 1$) by \emph{ex falso}. Conversely, assume that we have $\mathsf T\vdash\psi$ for a $\Pi$-sentence~$\psi$. If $\psi$ was false, $\neg\psi$ would be a true~$\Sigma$-sentence. By the previous theorem we would get~$\mathsf T\vdash\neg\psi$, against the consistency of~$\mathsf T$.
\end{proof}

As another consequence of $\Sigma$-completeness, we show that Robinson arihmetic `knows' about fixed values of certain functions.

\begin{definition}\label{def:sigma-definable}
We say that $f:\mathbb N^n\to\mathbb N$ is defined by a formula $\varphi(x_1,\ldots,x_n,y)$ if all natural numbers $a_i$ and $b$ validate
\begin{equation*}
f(a_1,\ldots,a_n)=b\quad\Leftrightarrow\quad\mathbb N\vDash\varphi(a_1,\ldots,a_n,b).
\end{equation*}
When $f$ can be defined by a $\Sigma$-formula, it is called $\Sigma$-definable. A relation $R\subseteq\mathbb N^n$ is called $\Delta$-definable if its characteristic function $\chi_R:\mathbb N^n\to\mathbb N$ is $\Sigma$-definable, where we have $\chi_R(\mathbf a)=1$ for $\mathbf a\in R$ and $\chi_R(\mathbf a)=0$ otherwise.
\end{definition}

The terminology for relations is justified as the $\Sigma$-condition~$\chi_R(\mathbf a)=1$ is equivalent to the $\Pi$-condition $\chi_R(\mathbf a)\neq 0$. More generally, we note that $f(\mathbf a)=b$ is equivalent to $\forall y(y\neq b\to f(\mathbf a)\neq y)$ when $f$ is total (i.\,e., defined on all of~$\mathbb N^n)$. However, this observation breaks down when one generalizes to functions that are partial (defined on a subset of~$\mathbb N^n$ that may not be $\Pi$-definable). With this generalization in mind, we speak of $\Sigma$-definable rather than $\Delta$-definable functions. Let us point out that there are $\Sigma$-formulas~$\varphi$ such that the relation $\{\mathbf a\,|\,\mathbb N\vDash\varphi(\mathbf a)\}$ is not $\Delta$-definable, as we will see in Exercise~\ref{ex:Sigma-Delta-diff}.

When~$f$ is defined by a $\Sigma$-formula~$\varphi$, the $\Sigma$-completeness and soundness of Robinson arithmetic entail that $f(\mathbf a)=b$ is equivalent to $\mathsf Q\vdash\varphi(\overline{\mathbf a},\overline b)$. The following yields an even stronger form of representability.

\begin{proposition}\label{prop:sigma-def-representable}
For any $\Sigma$-definable $f:\mathbb N^n\to\mathbb N$ there is a $\Sigma$-formula~$\varphi$ that represents $f$ in the sense that any $a_i\in\mathbb N$ validate
\begin{equation*}
\mathsf Q\vdash\forall y\left(\varphi\left(\overline{a_1},\ldots,\overline{a_n},y\right)\leftrightarrow y=\overline{f(a_1,\ldots,a_n)}\right).
\end{equation*}
\end{proposition}
\begin{proof}
Due to Exercise~\ref{ex:sigma-sigma_1}, we get a bounded formula~$\theta(\mathbf x,y,z)$ with the property that $f(\mathbf a)=b$ is equivalent to $\mathbb N\vDash\exists z\,\theta(\overline{\mathbf a},b,z)$. Let $\varphi_0(\mathbf x,y,z)$ be the formula
\begin{equation*}
\theta(\mathbf x,y,z)\land\forall y',z'\leq y\big(y'\neq y\to\neg\theta(\mathbf x,y',z')\big)\land\forall y',z'\leq z\big(y'\neq y\to\neg\theta(\mathbf x,y',z')\big).
\end{equation*}
Note that, intuituvely, the second and third conjunct together quantify over all~$y',z'$ below~$\max(y,z)$. We show that the conclusion of the proposition holds when~$\varphi(\mathbf x,y)$ is the formula $\exists z\,\varphi_0(\mathbf x,y,z)$. First observe that $\varphi$ does still define~$f$. For arbitrary~$\mathbf a$ and $b=f(\mathbf a)$, we can thus invoke $\Sigma$-completeness to get $\mathsf Q\vdash\varphi\left(\overline{\mathbf a},\overline b\right)$. This yields the direction from right to left in the desired equivalence.

For the other direction, we first set $m:=\max(b,c)$ for some $c$ with $\mathbb N\vDash\varphi_0(\mathbf a,b,c)$. Observe that $\Sigma$-completeness yields $\mathsf Q\vdash\varphi_0(\overline{\mathbf a},\overline b,\overline c)$  and in particular $\mathsf Q\vdash\theta(\overline{\mathbf a},\overline b,\overline c)$. We now argue in~$\mathsf Q$ and assume, towards a contradiction, that we have $\varphi(\overline{\mathbf a},y)$~for some $y\neq\overline b$. Consider some~$z$ with $\varphi_0(\overline{\mathbf a},y,z)$. In view of Exercise~\ref{ex:Q}, we may distinguish the following two cases: First assume that we have $\overline m\leq y$ or $\overline m\leq z$. By the second or third disjunct of $\varphi_0(\overline{\mathbf a},y,z)$, we then get $\neg\theta(\overline a,\overline b,\overline c)$, against the above. Now assume that we have $y,z\leq\overline m$ and hence $y,z\leq\overline b$ or $y,z\leq\overline c$. In this case, the second or third disjunct of $\varphi_0(\overline{\mathbf a},\overline b,\overline c)$ yields $\neg\theta(\overline a,x,y)$, against the assumption that we have $\varphi_0(\overline{\mathbf a},y,z)$.
\end{proof}

In the case of relations, we obtain the following.

\begin{corollary}
For any $\Delta$-definable relation~$R\subseteq\mathbb N^n$ there is a $\Sigma$-formula~$\varphi$ that represents~$R$ in the sense that any $a_i\in\mathbb N$ validate
\begin{align*}
(a_1,\ldots,a_n)\in R\quad&\Leftrightarrow\quad\mathsf Q\vdash\varphi\left(\overline{a_1},\ldots,\overline{a_n}\right),\\
(a_1,\ldots,a_n)\notin R\quad&\Leftrightarrow\quad\mathsf Q\vdash\neg\varphi\left(\overline{a_1},\ldots,\overline{a_n}\right).
\end{align*}
\end{corollary}
\begin{proof}
Assume that $\psi$ represents the characteristic function $\chi_R$ in the sense of the previous proposition. We declare that $\varphi(\mathbf x)$ is the formula $\psi(\mathbf x,\overline 1)$. The only part of the present corollary that does not directly follow from $\Sigma$-completeness and soundness is the direction from left to right of the second equivalence. To establish the latter, we note that $\mathbf a\notin R$ amounts to $\chi_R(\mathbf a)=0$. Given that $\mathsf Q$ proves~$\overline 0\neq\overline 1$, we get $\mathsf Q\vdash\neg\varphi(\overline{\mathbf a})$ due to the proposition.
\end{proof}

To conclude this section, we consider some variants of $\Sigma_1$-induction.

\begin{exercise}
(a) Show that $\mathsf{I\Sigma}_1$ proves~$\mathsf I\psi$ for any~$\Pi$-formula~$\psi$. \emph{Hint:} To establish $\psi(x)$ for an arbitrary~$x$, assume $\neg\psi(x)$ and derive~$\neg\psi(x-y)$ by induction on~$y\leq x$. The same argument proves the converse result that $\Pi_1$-induction entails $\Sigma_1$-induction.

(b) Conclude that, for any formula~$\varphi(x)$ that is $\Pi$ or $\Sigma$, we have
\begin{equation*}
\mathsf{I\Sigma}_1\vdash\forall x\big(\forall y<x\,\varphi(y)\to\varphi(x)\big)\to\forall x\,\varphi(x).
\end{equation*}
Here $\forall y<x\,\varphi(y)$ abbreviates $\forall y(y\leq x\land y\neq x\to\varphi(y))$. \emph{Remark:} The converse of this form of induction is a least element principle. Given $\exists x\,\psi$, it allows us to infer that there is an~$x$ with $\forall y<x\,\neg\psi(y)$ but $\psi(x)$.
\end{exercise}

\subsection{Primitive Recursion}\label{sect:prim-rec}

In this section, we introduce the primitive recursive functions, prove some of their closure properties and show that they are $\Sigma$-definable.

\begin{definition}\label{def:prim-rec}
A function $\mathbb N^n\to\mathbb N$ for some~$n\in\mathbb N$ is primitive recursive if it can be generated by the following recursive clauses:
\begin{enumerate}[label=(\roman*)]
\item The constant function~$C^n_k:\mathbb N^n\to\mathbb N$ with value~$C^n_k(a_1,\ldots,a_n)=k$ is primitive recursive for all~$k,n\in\mathbb N$.
\item The projection $I^n_i:\mathbb N^n\to\mathbb N$ with $I^n_i(a_1,\ldots,a_n)=a_i$ is primitive recursive for all intergers $1\leq i\leq n$.
\item The successor function $S:\mathbb N\to\mathbb N$ with $S(a)=a+1$ is primitive recursive.
\item Given $m,n\geq 1$ and primitive recursive $h:\mathbb N^n\to\mathbb N$ and $g_i:\mathbb N^m\to\mathbb N$ for all $i=1,\ldots,n$, the composition $f:\mathbb N^m\to\mathbb N$ with
\begin{equation*}
f(a_1,\ldots,a_m)=h\big(g_1(a_1,\ldots,a_n),\ldots,g_m(a_1,\ldots,a_m)\big)
\end{equation*}
is primitive recursive.
\item If $g:\mathbb N^n\to\mathbb N$ and $h:\mathbb N^{n+2}\to\mathbb N$ are primitive recursive, then so is the result $f:\mathbb N^{n+1}\to\mathbb N$ of the recursion with clauses
\begin{align*}
f(a_1,\ldots,a_n,0)&=g(a_1,\ldots,a_n),\\
f(a_1,\ldots,a_n,b+1)&=h\big(a_1,\ldots,a_n,b,f(a_1,\ldots,a_n,b)\big).
\end{align*}
\end{enumerate}
A relation $R\subseteq\mathbb N^n$ is called primitive recursive if the characteristic function~$\chi_R$ is a primitive recursive function.
\end{definition}

Let us consider the truncated subtraction function that is defined by
\begin{equation*}
a\dotminus b:=\begin{cases}
a-b & \text{if $a\geq b$},\\
0 & \text{otherwise}.
\end{cases}
\end{equation*}
To see that $a\mapsto a\dotminus 1$ is primitive recursive, use clause~(v) with $C^0_0$ and $I^2_1$ at the place of $g$ and~$h$, respectively. In view of $a\dotminus 0=a$ and $a\dotminus(b+1)=(a\dotminus b)\dotminus 1$, another application of~(v) shows that $(a,b)\mapsto a\dotminus b$ is primitive recursive. 

\begin{exercise}\label{ex:add-prim-rec}
Show that addition and multiplication are primitive recursive. Prove that $(\mathbf a,b)\mapsto\sum_{i<b}f(\mathbf a,i)$ is primitive recursive when the same holds~for~$f$.
\end{exercise}

Let us note that $\{\mathbf a\in\mathbb N^n\,|\,f(\mathbf a)=0\}$ is primitive recursive when the same holds for~$f$, since the characteristic function can be given as $\mathbf a\mapsto 1\dotminus f(\mathbf a)$. More generally, $\{\mathbf a\in\mathbb N^n\,|\,f(\mathbf a)\leq g(\mathbf a)\}$ is primitive recursive when this holds for~$f$ and~$g$, since $f(\mathbf a)\leq g(\mathbf a)$ is equivalent to $f(\mathbf a)\dotminus g(\mathbf a)=0$. The next proposition will also allow us to replace $\leq$ by equality (as $a=b$ is the conjunction of $a\leq b$ and $b\leq a$).

Given that the primitive recursive functions are closed under composition, the relation $\{\mathbf a\in\mathbb N^n\,|\,(g_1(\mathbf a),\ldots,g_m(\mathbf a))\in R\}$ is primitive recursive if the same holds for the relation~$R\subseteq\mathbb N^m$ and the functions $g_i:\mathbb N^n\to\mathbb N$. The case where each~$g_i$ is a projection~$I^n_j$ is of particular interest, as it shows that the arguments in a primitive recursive relation may be permuted. Together with the following, we see that the primitive recursive relations are closed under definitions by bounded formulas.

\begin{proposition}\label{prop:bounded-form-prim-rec}
We have the following closure properties:
\begin{enumerate}[label=(\alph*)]
\item If $R\subseteq\mathbb N^n$ is primitive recursive, then so is $\overline R:=\mathbb N^n\backslash R$.
\item If $R_i\subseteq\mathbb N^n$ is primitive recursive for each $i<2$, so are $R_0\cap R_1$ and $R_0\cup R_1$.
\item If $R\subseteq\mathbb N^{n+1}$ is primitive recursive, then the relations~$R_i\subseteq\mathbb N^{n+1}$ with
\begin{align*}
(\mathbf a,b)\in R_0\quad&\Leftrightarrow\quad(\mathbf a,i)\in R\text{ for all }i<b,\\
(\mathbf a,b)\in R_1\quad&\Leftrightarrow\quad(\mathbf a,i)\in R\text{ for some }i<b
\end{align*}
are primitive recursive as well.
\end{enumerate}
\end{proposition}
\begin{proof}
(a) In view of the remark before Exercise~\ref{ex:add-prim-rec}, it suffices to note that we have $\overline R=\{\mathbf a\in\mathbb N^n\,|\,\chi_R(\mathbf a)=0\}$.

(b) If $\chi_i$ is the characteristic function of~$R_i$, the characteristic function of~$R_0\cap R_1$ can be given as $\mathbf a\mapsto\chi_0(\mathbf a)\cdot\chi_1(\mathbf a)$. The remaining claim follows by duality, i.\,e., since $R_0\cup R_1$ is the complement of $\overline{R_0}\cap\overline{R_1}$.

(c) We have $\mathbf a\in R_1$ precisely when $\sum_{i<b}\chi_R(\mathbf a,i)$ is non-zero. By part~(a) and (the remark before) Exercise~\ref{ex:add-prim-rec}, this shows that $R_1$ is primitive recursive. The claim about~$R_1$ follows by duality once again.
\end{proof}

Since the primitive recursive functions are closed under composition, part~(c) of the previous proposition entails that we get primitive recursive relations $R_i'\subseteq\mathbb N^n$ if we stipulate that $\mathbf a\in R_i'$ holds when we have $(\mathbf a,i)\in R$ for some or all~$i<h(\mathbf a)$, provided that $h:\mathbb N^n\to\mathbb N$ is primitive recursive. Next, we note that the primitive recursive functions are closed under definitions by case distinction. In order to accommodate a definition with more than two cases, it suffices to apply the following result in an iterative fashion. 

\begin{lemma}
If $f,g:\mathbb N^n\to\mathbb N$ and $R\subseteq\mathbb N^n$ are primitive recursive, then the function $h:\mathbb N^n\to\mathbb N$ that is given by
\begin{equation*}
h(\mathbf a)=\begin{cases}
f(\mathbf a) & \text{if $\mathbf a\in R$},\\
g(\mathbf a) & \text{otherwise}
\end{cases}
\end{equation*}
is primitive recursive as well.
\end{lemma}
\begin{proof}
We have $h(\mathbf a)=f(\mathbf a)\cdot\chi_R(\mathbf a)+g(\mathbf a)\cdot(1\dotminus\chi_R(\mathbf a))$.
\end{proof}

As the following exercise shows, the assumption that $R$ is primitive recursive cannot be dropped in the previous lemma. More informative results on the limitations of primitive recursion will be proved in Section~\ref{sect:goodstein}.

\begin{exercise}
Show that not all subsets of~$\mathbb N$ are primitive recursive. \emph{Hint:}~Note that there are only countably many primitive recursive functions.
\end{exercise}

We now establish closure under an operator that is known as bounded minimization or bounded search.

\begin{proposition}\label{prop:bounded-min}
If the function $f:\mathbb N^{n+1}\to\mathbb N$ is primitive recursive, then the function $\mu_0f:\mathbb N^{n+1}\to\mathbb N$ that is given by
\begin{equation*}
\mu_0f(\mathbf a,b)=\min\big(\{i<b\,|\,f(\mathbf a,i)=0\}\cup\{b\}\big)
\end{equation*}
is primitive recursive as well.
\end{proposition}
\begin{proof}
Proposition~\ref{prop:bounded-form-prim-rec} yields a primitive recursive function $g:\mathbb N^{n+1}\to\mathbb N$ with
\begin{equation*}
g(\mathbf a,i)=\begin{cases}
1 & \text{if $f(\mathbf a,j)\neq 0$ for all~$j<i$},\\
0 & \text{otherwise}.
\end{cases}
\end{equation*}
Now it suffices to observe that we have $\mu_0f(\mathbf a,b)=\sum_{i<b}g(\mathbf a,i)$.
\end{proof}

Given a primitive recursive relation $R\subseteq\mathbb N^{n+1}$ and function $b:\mathbb N^n\to\mathbb N$ such that each tuple $\mathbf a\in\mathbb N^n$ admits an~$i<b(\mathbf a)$ with $(\mathbf a,i)\in R$, we can conclude that the function $\mathbf a\mapsto\min\{i\in\mathbb N\,|\,(\mathbf a,i)\in R\}$ is primitive recursive. Here the bounding function~$b$ is essential: in Section~\ref{sect:computability} we will see that unbounded search leads out of the class of primitive recursive functions.

In a recursive definition according to clause~(v) of Definition~\ref{def:prim-rec}, the successor value $f(\mathbf a,b+1)$ may depend on the previous value~$f(\mathbf a,b)$. Our next goal is to show that the primitive recursive functions are closed under a more liberal recursion principle, in which $f(\mathbf a,b)$ may depend on the sequence $f(\mathbf a,0),\ldots,f(\mathbf a,b-1)$ of all previous values. For this purpose, we now discuss how finite sequences can be coded by single numbers. Let us first recall the Cantor pairing function $\pi:\mathbb N^2\to\mathbb N$ from Exercise~\ref{ex:Cantor-pairing}. It is straightforward to see that we have $a,b\leq\pi(a,b)$ and that $\pi$ is increasing in each argument. Based on Exercise~\ref{ex:add-prim-rec}, we see that $\pi$ is primitive recursive. As $\pi$ is bijective, we have functions $\pi_i:\mathbb N\to\mathbb N$ for~$i<2$ that validate
\begin{equation*}
\pi_i\big(\pi(a_0,a_1)\big)=a_i\quad\text{and}\quad\pi\big(\pi_0(c),\pi_1(c)\big)=c.
\end{equation*}
We note that the functions $\pi_i$ are primitive recursive due to
\begin{equation*}
\pi_0(c)=\min\{a\leq c\,|\,\pi(a,b)=c\text{ for some }b\leq c\}.
\end{equation*}
To code finite sequences of arbitrary length, we iterate the Cantor pairing function.

\begin{definition}\label{def:seq-code}
With each finite sequence $a_0,\ldots,a_{n-1}$ of natural numbers, we associate a code $\langle a_0,\ldots,a_{n-1}\rangle\in\mathbb N$ that is recursively given by
\begin{equation*}
\langle\rangle:=0\quad\text{and}\quad\langle a_0,\ldots,a_n\rangle:=\pi(\langle a_0,\ldots,a_{n-1}\rangle,a_n)+1.
\end{equation*}
\end{definition}

Note that one should not define $\langle a_0,\ldots,a_n\rangle$ to be $\pi(\langle a_0,\ldots,a_{n-1}\rangle,a_n)$, for otherwise $\langle 0\rangle$ would be the same $\pi(\langle\rangle,0)=\pi(0,0)=0=\langle\rangle$.

\begin{exercise}
Show that each $c\in\mathbb N$ admits unique natural numbers $n$ and $a_i$ such that we have $c=\langle a_0,\ldots,a_{n-1}\rangle$. \emph{Hint:} For uniqueness, use induction on~$n$ to prove that $\langle a_0,\ldots,a_{m-1}\rangle=\langle b_0,\ldots,b_{n-1}\rangle$ entails $m=n$ and $a_i=b_i$ for all~$i<n$.
\end{exercise}

Due to the exercise, we have well-defined functions that provide the lenght and the entries of coded sequences.

\begin{definition}\label{def:sequence-encoding}
For $c=\langle a_0,...,a_{n-1}\rangle$ we declare~$\len(c)=n$ as well as $(c)_i=a_i$ for $i<n$ and $(c)_i=0$ for~$i\geq n$.
\end{definition}

Let us note that $a,b\leq\pi(a,b)$ entails $\len(c)\leq c$ as well as $(c)_i<c$ for $i<\len(c)$.

\begin{lemma}\label{lem:seq-code}
The functions $c\mapsto\len(c)$ and $(c,i)\mapsto(c)_i$ are primitive recursive.
\end{lemma}
\begin{proof}
Consider the function defined by~$f(c)=\pi_0(c\dotminus 1)$. Given that $f$ is primitive recursive, the same holds for the iterated version $(c,i)\mapsto f^{(i)}(c)$ with $f^{(0)}(c)=c$ and $f^{(i+1)}(c)=f(f^{(i)}(c))$. For $c=\langle a_0,\ldots,a_n\rangle$, an induction on~$i\leq n$ shows that we have $f^{(i)}(c)=\langle a_0,\ldots,a_{n-i}\rangle$. This yields
\begin{align*}
\len(c)&=\min\left\{i\leq c\,\left|\,f^{(i)}(c)=0\right.\right\},\\
(c)_i&=\pi_1\left(f^{(\len(c)+1-i)}(c)-1\right)\quad\text{for }i<\len(c),
\end{align*}
as needed to establish the lemma. 
\end{proof}

We can now justify the recursion principle that was mentioned above.

\begin{proposition}[Course-of-values recursion]\label{prop:course-of-values}
If $h:\mathbb N^{n+1}\to\mathbb N$ is primitive recursive, then the function $f:\mathbb N^{n+1}\to\mathbb N$ that is determined by
\begin{equation*}
f(\mathbf a,b)=h(\mathbf a,\langle f(\mathbf a,0),\ldots,f(\mathbf a,b-1)\rangle)
\end{equation*}
is primitive recursive as well.
\end{proposition}
\begin{proof}
Let $\overline f:\mathbb N^{n+1}\to\mathbb N$ be given by $\overline f(\mathbf a,b)=\langle f(\mathbf a,0),\ldots,f(\mathbf a,b-1)\rangle$. We have
\begin{equation*}
\overline f(\mathbf a,b+1)=\pi\left(\overline f(\mathbf a,b),h\left(\mathbf a,\overline f(\mathbf a,b)\right)\right)+1,
\end{equation*}
so that $\overline f$ is primitive recursive by clause~(v) of Definition~\ref{def:prim-rec}. We can now~conclude by the previous lemma, as we have $f(\mathbf a,b)=(\overline f(\mathbf a,b+1))_b$.
\end{proof}

Note that the function~$h$ in the previous proposition does also have access to~$b$, since the latter is determined as the length of the argument~$\overline f(\mathbf a,b)$. In the case of relations, course-of-values recursion allows us to define $(\mathbf a,b)\in R$ by a condition that depends on $(\mathbf a,i)\in R$ for~$i<b$. More precisely, any relation $Q\subseteq\mathbb N^{n+2}$~determines another relation~$R\subseteq\mathbb N^{n+1}$ with
\begin{equation*}
(\mathbf a,b)\in R\quad\Leftrightarrow\quad(\mathbf a,b,\langle\chi_R(\mathbf a,0),\ldots,\chi_R(\mathbf a,b-1)\rangle)\in Q,
\end{equation*}
where $R$ is primitive recursive when the same holds for~$Q$.

\begin{exercise}
Show that the concatenation function $*:\mathbb N\to\mathbb N$ given by
\begin{equation*}
\langle a_0,\ldots,a_{m-1}\rangle*\langle b_0,\ldots,b_{n-1}\rangle=\langle a_0,\ldots,a_{m-1},b_0,\ldots,b_{n-1}\rangle
\end{equation*}
is primitive recursive. \emph{Remark:} It is instructive to give one solution that uses course-of-values recursion and one that does not. 
\end{exercise}

Together with Proposition~\ref{prop:sigma-def-representable} and its corollary, the following shows that primitive recursive functions and relations are representable in Robinson arithmetic.

\begin{theorem}\label{thm:prim-rec-sigma}
Each primitive recursive function is $\Sigma$-definable.
\end{theorem}
\begin{proof}
To specify a value $f(\mathbf a,b)$ of a function as in clause~(v) of Definition~\ref{def:prim-rec}, we will want to refer to the sequence $f(\mathbf a,0),\ldots,f(\mathbf a,b-1)$ of previous values. For this reason, we first present a $\Sigma$-definable coding of finite sequences. Note that the sequence encoding that we have used above is primitive recursive and will thus be $\Sigma$-definable by the present theorem. Since the latter is not yet established, however, we need to find a coding for which we can give a direct proof that it is $\Sigma$-definable, which is a more challenging task.

We employ a coding function $\beta:\mathbb N^3\to\mathbb N$ that was already considered by G\"odel and is based on the Chinese remainder theorem. As we will see, it has the property that any finite sequence $a_0,\ldots,a_{n-1}$ can be coded by numbers~$b$ and $c$ such that we have $\beta(b,c,i)=a_i$ for all~$i<n$. Using the Cantor pairing function, we set
\begin{equation*}
\beta(b,c,i)=\min\big(\{d\in\mathbb N\,|\,1+\pi(d,1+i)\cdot b\text{ divides }c\}\cup\{b\}\big).
\end{equation*}
To verify the indicated property, we consider an arbitrary sequence~$a_0,\ldots,a_{n-1}$. Pick an~$a$ that is strictly larger than all~$a_i$. Then set
\begin{equation*}
b:=\pi(a,n)!\quad\text{and}\quad c:=\prod_{i=0}^{n-1}1+\pi(a_i,1+i)\cdot b.
\end{equation*}
For each~$i<n$, we see that $a_i$ is one of the $d$ over which the minimum above is taken, so that we get $\beta(b,c,i)\leq a_i<a\leq\pi(a,n)\leq b$. For $d:=\beta(b,c,i)$ we can conclude that $1+\pi(d,1+i)\cdot b$ divides~$c$. Let us note that the numbers $1+k\cdot\pi(a,n)!$ for $0<k\leq\pi(a,n)$ are pairwise co-prime (since any prime factor of one of these numbers is greater than $\pi(a,n)$ while the prime factors of a difference $(k-l)\cdot\pi(a,n)!$ are smaller). Given that $a'<a$ and $i<n$ entail $0<\pi(a',1+i)<\pi(a,n)$, there must thus be a $j<n$ with $\pi(d,1+i)=\pi(a_j,1+j)$. The latter yields $i=j$ and then $d=a_i$, as desired.

We now show that $\beta$ is $\Sigma$-definable. First observe that $\pi(a,b)=c$ is equivalent to $\mathbb N\vDash\varphi_\pi(a,b,c)$ where the latter is
\begin{equation*}
\exists z\leq (a+b)\times S(a+b)\,\big(z+z=(a+b)\times S(a+b)\land b+z=c).
\end{equation*}
It follows that $\beta(b,c,i)$ is the minimal~$d$ such that $\mathbb N$ satisfied $\theta(b,c,i,d)$, where the latter is given by
\begin{equation*}
\exists x\leq(d+Si)\times S(d+Si)\,\exists y\leq c\,\big(\varphi_\pi(d,Si,x)\land(1+x\times b)\times y=c\big)\lor b=d.
\end{equation*}
In other words, $\beta(b,c,i)=d$ holds precisely if $\mathbb N$ satisfies $\varphi_\beta(b,c,i,d)$, where the latter is
\begin{equation*}
\theta(b,c,i,d)\land\forall d'<d\,\neg\theta(b,c,i,d').
\end{equation*}
Note that $\varphi_\beta$ is a $\Sigma$-formula since $\theta$ and thus also~$\neg\theta$ is~$\Delta_0$ (actually $\theta$ being~$\Sigma$ would suffice by the paragraph after Definition~\ref{def:sigma-definable}).

To establish the theorem, we now argue by induction over the recursive clauses from Definition~\ref{def:prim-rec}. For the first three clauses we need only observe
\begin{align*}
C^n_k(a_1,\ldots,a_n)=b\quad&\Leftrightarrow\quad\mathbb N\vDash b=\overline k,\\
I^n_i(a_1,\ldots,a_n)=b\quad&\Leftrightarrow\quad\mathbb N\vDash b=a_i,\\
S(a)=b\quad&\Leftrightarrow\quad\mathbb N\vDash b=Sa.
\end{align*}
Now assume that~$f$ is given as in clause~(iv). We inductively have $\Sigma$-formulas with
\begin{align*}
h(b_1,\ldots,b_n)=c\quad&\Leftrightarrow\quad\mathbb N\vDash\psi(b_1,\ldots,b_n,c),\\
g_i(a_1,\ldots,a_m)=b\quad&\Leftrightarrow\quad\mathbb N\vDash\varphi_i(a_1,\ldots,a_m,b).
\end{align*}
It follows that $f(a_1,\ldots,a_m)=c$ holds precisely if we have
\begin{equation*}
\mathbb N\vDash\exists y_1,\ldots,y_n\left(\psi(y_1,\ldots,y_n,c)\land\textstyle\bigwedge_{i=1,\ldots,n}\varphi_i(a_1,\ldots,a_m,y_i)\right).
\end{equation*}
Finally, assume that $f$ is given as in clause~(v). Consider the formulas $\varphi_g$ and $\varphi_h$ that the induction hypothesis provides for~$g$ and $h$ as well as the formula~$\varphi_\beta$ from the sequence encoding above. We then see that $f(\mathbf a,b)=c$ is equivalent to
\begin{multline*}
\mathbb N\vDash\exists x,x'\big(\varphi_\beta(x,x',b,c)\land\exists z\left(\varphi_\beta(x,x',0,z)\land\varphi_g(\mathbf a,z)\right)\land{}\\*
\forall y<b\exists z,z'\big(\varphi_\beta(x,x',y,z)\land\varphi_\beta(x,x',Sy,z')\land\varphi_h(\mathbf a,y,z,z')\big)\big).
\end{multline*}
Here $x$ and $x'$ witness the existential quantifier precisely if~$\beta(x,x',y)=f(\mathbf a,y)$ holds for all~$y\leq b$, as one can check by induction (note that $z$ and $z'$ are uniquely determined as values of~$\beta$).
\end{proof}

The following is immediate in view of Definition~\ref{def:sigma-definable}.

\begin{corollary}
Any primitive recursive relation is $\Delta$-definable.
\end{corollary}

Representability in the sense of Proposition~\ref{prop:sigma-def-representable} is a pointwise property, as we get a separate proof in $\mathsf Q$ for each tuple of externally fixed arguments. By the following exercise, the stronger theory~$\mathsf{I\Sigma}_1$ proves a result in which the quantification over the arguments is internal. Let us note that $\exists!$ denotes unique existence, so that $\exists!x\,\varphi(x)$ abbreviates the conjunction of $\exists x\,\varphi(x)$ and $\forall x,x'(\varphi(x)\land\varphi(x')\to x=x')$.

\begin{exercise}\label{ex:ISigma_1-to-primrec}
Give a proof that any primitive recursive $f$ admits a $\Sigma$-formula~$\varphi$ such that $f(\mathbf a)=b$ is equivalent to $\mathbb N\vDash\varphi(\mathbf a,b)$ and we have $\mathsf{I\Sigma}_1\vdash\forall\mathbf x\exists! y\,\varphi(\mathbf x,y)$. \emph{Remark:}~The formulas constructed in the proof of Theorem~\ref{thm:prim-rec-sigma} are as desired. You may assume that $\mathsf{I\Sigma}_1$ proves the following: for any $a,b,d,n$ there are $b',d'$ such that we have $\beta(b',d',n)=a$ and $\beta(b',d',i)=\beta(b,d,i)$ for all~$i<n$ (which inductively shows that $\beta$ can code any finite sequence). While this is true, it seems hard to establish for~$\beta$ directly. A very diligent proof that uses two different sequence~encodings can be found in~\cite{hajek91} (see in particular Theorem~I.1.43 and~Section~I.1(c)). 
\end{exercise}

Theorem~\ref{thm:prim-rec-sigma} is not sharp in the sense that there are $\Sigma$-definable functions that are not primitive recursive, as we will see in Section~\ref{sect:computability}. On the other hand, we will show in Section~\ref{subsect:prov-tot-ISigma_1}~that the previous exercise is sharp in a suitable sense, i.\,e., that primitive recursion~corresponds precisely to the strength of $\mathsf{I\Sigma}_1$.

\subsection{Diagonalization and Incompleteness}\label{subsect:goedel}

In the present section, we prove the so-called diagonal lemma and derive Tarski's theorem on the undefinability of truth as well as G\"odel's first incompleteness theorem. We also consider~the second incompleteness theorem and L\"ob's theorem.

A crucial tool is the so-called arithmetization of notions like formula and proof. While the following clearly leaves room for variation, the paragraph after the definition provides some motivation for our choices.

\begin{definition}\label{def:Goedel-numbers}
To each pre-term~$t$ over the signature of first-order arithmetic we assign a so-called G\"odel number~$\gn{t}\in\mathbb N$, which is given by
\begin{gather*}
\gn{v_i}=\langle 0,0,i\rangle,\qquad\gn{0}=\langle 0,1\rangle,\qquad\gn{St}=\langle 0,2,\gn{t}\rangle,\\
\gn{s+t}=\langle 0,3,\gn{s},\gn{t}\rangle,\qquad\gn{s\times t}=\langle 0,4,\gn{s},\gn{t}\rangle.
\end{gather*}
We also assign a G\"odel number $\gn{\varphi}\in\mathbb N$ to each pre-formula~$\varphi$ over the signature of first-order arithmetic, by stipulating
\begin{align*}
\gn{s=t}&=\langle 1,0,\gn{s},\gn{t}\rangle,& \gn{\neg\,s=t}&=\langle 1,1,\gn{s},\gn{t}\rangle,\\
\gn{s\leq t}&=\langle 1,2,\gn{s},\gn{t}\rangle,& \gn{\neg\,s\leq t}&=\langle 1,3,\gn{s},\gn{t}\rangle,\\
\gn{\varphi_0\land\varphi_1}&=\langle 1,4,\gn{\varphi_0},\gn{\varphi_1}\rangle,&\gn{\varphi_0\lor\varphi_1}&=\langle 1,5,\gn{\varphi_0},\gn{\varphi_1}\rangle,\\
\gn{\forall v_{2i}\,\varphi}&=\langle 1,6,2i,\gn{\varphi}\rangle,&\gn{\exists v_{2i}\,\varphi}&=\langle 1,7,2i,\gn{\varphi}\rangle.
\end{align*}
\end{definition}

Note that the first entry of a G\"odel number~$\gn{\sigma}$ tells us whether $\sigma$ is a term or a formula. The second entry determines the outermost constructor or, in other words, the last clause of Definition~\ref{def:signature} that was used in the construction of~$\sigma$. A~straightforward induction shows that $\sigma$ and $\tau$ are the same term or formula when we have $\gn{\sigma}=\gn{\tau}$ (recall $a_i<\langle a_0,\ldots,a_n\rangle$). In the rest of this section it is tacitly~understood that we work over the signature of first-order arithmetic. 

\begin{lemma}\label{lem:formula-pr}
The relation $\{\gn{\sigma}\,|\,\sigma\in\Sigma\}\subseteq\mathbb N$ is primitive recursive when $\Sigma$ is the collection of all pre-terms, pre-formulas, terms or formulas, respectively.
\end{lemma}
\begin{proof}
To convey the main idea, we show that $R:=\{\gn{\sigma}\,|\,\sigma\in\Sigma\}\subseteq\mathbb N$ is primitive recursive when $\Sigma$ is the collection of numerals. This can be accomplished by course-of-values recursion, given that we have $b\in R$ precisely if $b$ is equal to $\langle 0,1\rangle$ or of the form $\langle 0,2,c\rangle$, where we necessarily have~$c<b$. To make this more formal, we need to define a primitive recursive relation~$Q\subseteq\mathbb N^2$ as in the paragraph after Proposition~\ref{prop:course-of-values}. First note that $Q_0:=\{(\gn{0},s)\,|\,s\in\mathbb N\}$ is primitive recursive by the paragraph after Exercise~\ref{ex:add-prim-rec} (and since constant functions and projections are primitive recursive). Let us also define $Q_1$ as the set of all pairs $(b,c)$ such that we have $b=\langle 0,2,i\rangle$ for some $i<b=\len(c)$ with $(c)_i=1$ (cf.~Lemma~\ref{lem:seq-code}). This set is primitive recursive in view of Proposition~\ref{prop:bounded-form-prim-rec}. An induction on~$b$ confirms
\begin{equation*}
b\in R\quad\Leftrightarrow\quad\left(b,\langle\chi_R(0),\ldots,\chi_R(b-1)\rangle\right)\in Q_0\cup Q_1=:Q.
\end{equation*}
The proofs of the claims about pre-terms and pre-formulas are more tedious to write out (because there are more cases) but not more difficult on a conceptual level. Let us now consider the set
\begin{equation*}
\fv:=\{(i,\gn{\sigma})\,|\,\text{$\sigma$ is a pre-term or pre-formula and $v_i\in\fv(\sigma)$}\},
\end{equation*}
which is primitive recursive due to course-of-values recursion (cf.~Definition~\ref{def_fv-formula}). To get the claim about terms and formulas, note that $(i,\ulcorner\sigma\urcorner)\in\fv$ entails $i<\sigma$. Thus $b$ is the G\"odel number of a term or formula, respectively, if it is the G\"odel number of a pre-term or pre-formula and we have $(2i,b)\notin\fv$ for all $i<b$. We can conclude by closure under bounded quantification (see Proposition~\ref{prop:bounded-form-prim-rec}).
\end{proof}

It is clear that functions such as $(\gn{\varphi},\gn{\psi})\mapsto\gn{\varphi\land\psi}$ are primitive recursive. Part~(a) of the following exercise covers the case of negation, which is somewhat less immediate in view of~Definition~\ref{def:nega-pred-log}. Part~(b) of the exercise is facilitated by our treatment of substitution, in which it is not necessary to rename bound variables (see Definition~\ref{def:substitution} and the paragraph that precedes it).

\begin{exercise}\label{ex:neg-prim-rec}
(a) Show that $\gn{\varphi}\mapsto\gn{\neg\varphi}$ is a primitive recursive function on G\"odel numbers of pre-formulas. \emph{Remark:}~Provide full technical details for at least part of your solution (analogous to the case of numerals in the previous proof).

(b) Prove that there is a primitive recursive function $\sub:\mathbb N^3\to\mathbb N$ such that we have $\sub(\gn{\sigma},i,\gn{t})=\gn{\sigma[v_i/t]}$ whenever $\sigma$ is a term or formula and~$t$ is a term.
\end{exercise}

We continue with a main technical result of the present section.

\begin{proposition}[Diagonal lemma]
For any formula $\varphi(x)$ with no further free variables, there is a sentence~$\psi$ such that we have~$\mathsf Q\vdash\psi\leftrightarrow\varphi(\overline{\gn{\psi}})$.
\end{proposition}
\begin{proof}
The previous exercise ensures that $\sigma\mapsto\sub(\sigma,1,\gn{\overline\sigma})$ is primitive recursive (as the same holds for~$n\mapsto\gn{\overline n}$). Due to Proposition~\ref{prop:sigma-def-representable} and Theorem~\ref{thm:prim-rec-sigma}, we find a $\Sigma$-formula~$\theta(x,y)$ such that
\begin{equation*}
\mathsf Q\vdash\forall y\,\left(\theta\left(\overline{\gn{\sigma}},y\right)\leftrightarrow y=\overline{\gn{\sigma\left(\overline{\gn{\sigma}}\right)}}\right)
\end{equation*}
holds for any formula~$\sigma(v_1)$. Let $\psi_0(v_1)$ be the formula $\exists y\,(\theta(v_1,y)\land\varphi(y))$. We~show that the proposition holds when~$\psi$ is given as $\psi_0(\overline{\gn{\psi_0}})$. Arguing in~$\mathsf Q$, we first assume $\psi$, so that we have a~$y$ with $\theta(\overline{\gn{\psi_0}},y)\land\varphi(y)$. Here the first conjunct forces $y=\overline{\gn{\psi}}$, which turns the second conjunct into~$\varphi(\overline{\gn{\psi}})$, as required. For the converse direction, we assume $\varphi(\overline{\gn{\psi}})$. By the above, we also have $\theta(\overline{\gn{\psi_0}},\overline{\gn{\psi}})$. Together we get $\exists y\,(\theta(\overline{\gn{\psi_0}},y)\land\varphi(y))$, which coincides with~$\psi$.
\end{proof}

The following result says that no formula of first-order arithmetic can define truth for all such formulas. To avoid confusion, we note that truth for one language can be defined in a richer one. Indeed, Definitions~\ref{def:models-pred-log} and~\ref{def:signature-fa} provide a precise characterization of truth for first-order arithmetic.

\begin{theorem}[Tarski's undefinability of truth~\cite{tarski36}]
There is no formula $\theta(x)$ of first-order arithmetic such that a sentence~$\psi$ of first-order arithmetic is true precisely if we have $\mathbb N\vDash\theta(\gn{\psi})$.
\end{theorem}
\begin{proof}
Aiming at a contradiction, we assume that there is a $\theta$ with the indicated property, which amounts to $\mathbb N\vDash\psi\leftrightarrow\theta(\gn{\psi}))$ for arbitrary~$\psi$. By the diagonal lemma with $\neg\theta$ at the place of~$\varphi$, we find a sentence $\psi$ with $\mathbb N\vDash\psi\leftrightarrow\neg\theta(\gn{\psi}))$. But then we get $\mathbb N\vDash\theta(\gn{\psi}))\leftrightarrow\neg\theta(\gn{\psi}))$, which cannot be the case.
\end{proof}

The following is a crucial ingredient for G\"odel's first incompleteness theorem.

\begin{proposition}
There is a primitive recursive relation~$\prf\subseteq\mathbb N^2$ such that a formula~$\varphi$ of first-order arithmetic is provable in predicate logic precisely if there is a number $p\in\mathbb N$ with $(p,\gn{\varphi})\in\prf$.
\end{proposition}
\begin{proof}
The idea is that $(p,\gn{\varphi})\in\prf$ holds precisely when~$p$ codes a proof of~$\varphi$ according to Definition~\ref{def:sequent-calc-FO} (see also the explanation after Lemma~\ref{lem:weakening}). In~order to make this precise, we declare that a coded sequent is a coded sequence of the form~$\langle\gn{\psi_0},\ldots,\gn{\psi_{n-1}}\rangle$, where each~$\psi_i$ is a formula. In the following, we sometimes omit the word `coded' and write $\varphi$ at the place of~$\gn{\varphi}$. Furthermore, we use familiar notation such as $\Delta\subseteq\Gamma$, which expresses that each entry of the coded sequent~$\Delta$ is also an entry of~$\Gamma$. By a coded proof, we mean a sequence~$\langle\Gamma^0,\ldots,\Gamma^l\rangle$ of sequents such that at least one of the following is satisfied for each~$k\leq l$:
\begin{description}[labelwidth=5.5ex,labelindent=\parindent,leftmargin=!,before={\renewcommand\makelabel[1]{(##1)}}]
\item[$\mathsf{Ax}$] The sequent $\Gamma^k$ contains some literal and its negation.
\item[$\land$] There is a formula~$\varphi_0\land\varphi_1$ in~$\Gamma^k$ such that each~$j<2$ admits an $i(j)<k$ for which $\Gamma^{i(j)}$ has the form~$\Gamma^k,\varphi_j$.
\item[$\lor$] There is a formula~$\varphi_0\lor\varphi_1$ in~$\Gamma^k$ such that some~$j<2$ admits an $i(j)<k$ for which $\Gamma^{i(j)}$ has the form~$\Gamma^k,\varphi_j$.
\item[$\forall$] We have a formula~$\forall x\,\varphi$ in~$\Gamma^k$, a subsequent $\Gamma'$ of $\Gamma^k$, a variable $y$ that is not free in $\Gamma',\forall x\,\varphi$ and an index $i<k$ such that $\Gamma^i$ has the form~$\Gamma',\varphi[x/y]$.
\item[$\exists$] We have a formula~$\exists x\,\varphi$ in~$\Gamma^k$, a term~$t$ and an index~$i<k$ such that $\Gamma^i$ has the form~$\Gamma^k,\varphi[x/t]$.
\item[$\mathsf{Cut}$] For some formula~$\varphi$, we have indices $i(0),i(1)<k$ such that $\Gamma^{i(0)}$ and~$\Gamma^{i(1)}$ are of the form~$\Gamma^k,\varphi$ and $\Gamma^k,\neg\varphi$, respectively.
\end{description}
In Section~\ref{sect:fundamentals}, we have considered proofs as trees. The coded proofs that we have just described list the vertices of these trees in a suitable order (from the leafs to the root). More formally, we can use induction over the generation of $\vdash\Gamma$ according to Definition~\ref{def:sequent-calc-FO} to find a coded proof~$\langle\Gamma^0,\ldots,\Gamma^l\rangle$ with $\Gamma=\Gamma^l$. Concersely, when $\langle\Gamma^0,\ldots,\Gamma^l\rangle$ is a coded proof, a straightforward induction on~$i\leq l$ yields~$\vdash\Gamma^i$.

We now declare that $(p,\gn{\varphi})\in\prf$ holds when $p$ is a coded proof~$\langle\Gamma^0,\ldots,\Gamma^l\rangle$ such that $\Gamma^l$ has the form~$\langle\gn{\varphi}\rangle$. The notion of coded proof is primitive recursive by Lemma~\ref{lem:formula-pr}, Exercise~\ref{ex:neg-prim-rec} and the results of Section~\ref{sect:prim-rec}. We provide some details for the case of clause~($\exists$). If we write $\Gamma^k=\langle\psi_0,\ldots,\psi_{n-1}\rangle$, the clause first asserts that there is an index $m<n$ such that $\varphi_m$ is of the form~$\exists x\,\varphi$. In view of Proposition~\ref{prop:bounded-form-prim-rec} (closure under bounded quantification) and Definition~\ref{def:Goedel-numbers}, this is a primitive recursive condition, and we can read off~$x$ and~$\varphi$ from the G\"odel number of~$\exists x\,\varphi$. We now determine whether~$x$ is actually free in~$\varphi$. If it is not, then $\varphi[x/t]$ coincides with~$\varphi$, so that the quantification over~$t$ can be avoided. If~$x$ is free in~$\varphi$, then $t$ will occur in~$\varphi[x/t]$ and hence in~$\Gamma^i$ for some~$i<k$, which allows us to bound the quantification over~$t$. Similarly, we see that the remaining conditions can be expressed in terms of bounded quantification and functions that we already know to be primitive recursive.
\end{proof}

Let us show that the result of the proposition extends to proofs from axioms.

\begin{exercise}\label{ex:prov-Sigma}
(a) Prove that $\{\gn{\varphi}\,|\,\mathsf T\vdash\varphi\}$ is $\Sigma$-definable if the same holds for the set~$\{\gn{\theta}\,|\,\theta\in\mathsf T\}$ of axioms of the theory~$\mathsf T$. \emph{Hint:} We have $\mathsf T\vdash\varphi$ precisely if there are~$\theta_i\in\mathsf T$ such that $\theta_1\land\ldots\land\theta_n\to\varphi$ is provable in predicate logic.

(b) Show that $\{\gn{\theta}\,|\,\theta\in\mathsf{PA}\}$ is primitive recursive and in particular $\Sigma$-definable. \emph{Remark:} This claim remains true when Peano arithmetic is replaced by Robinson arithmetic or by the theory~$\mathsf{I\Sigma}_1$.
\end{exercise}

We can now confirm that some of the previous constructions are optimal.

\begin{exercise}\label{ex:Sigma-Delta-diff}
(a) Show that there is a $\Sigma$-formula~$\theta(x)$ such that an arbitrary $\Sigma$-sentence~$\varphi$ is true precisely if we have $\mathbb N\vDash\theta(\gn{\varphi})$. \emph{Remark:} It is instructive to consider two different solutions. First, one can exploit that truth for $\Sigma$-formulas coincides with provability in~$\mathsf Q$. Second, one can use Proposition~\ref{prop:bounded-form-prim-rec} to show that $\{\gn{\theta}\,|\,\theta\in X\}$ is primitive recursive when $X$ consists of the true bounded formulas.

(b) Prove that there is no $\Pi$-formula~$\theta(x)$ such that an arbitrary $\Sigma$-sentence~$\varphi$ is true precisely if we have $\mathbb N\vDash\theta(\gn{\varphi})$.

(c) Conclude that there is a $\Sigma$-formula~$\varphi(x)$ such that the set $\{a\in\mathbb N\,|\,\mathbb N\vDash\varphi(a)\}$ is not $\Delta$-definable.

(d) Show that the set $\{\gn{\varphi}\,|\,\vdash\varphi\}$ of formulas that are provable in predicate logic is not $\Pi$-definable.
\end{exercise}

The work of the previous subsections culminates in the following. Let us note that the restrictions to theories of first-order arithmetic is inessential. The following result is readily extended to axiom systems over other signatures, as long as the natural numbers are definable and one can interpret~$\mathsf Q$ in a suitable sense.

\begin{theorem}[G\"odel's first incompleteness theorem~\cite{goedel-incompleteness}]\label{thm:goedel-one}
Consider a theory~$\mathsf T\supseteq\mathsf Q$ of first order arithmetic such that $\{\gn{\theta}\,|\,\theta\in\mathsf T\}$ is $\Sigma$-definable. It $\mathsf T$ is consistent, there is a $\Pi$-sentence~$\psi$ such that we have both $\mathsf T\nvdash\psi$ and $\mathsf T\nvdash\neg\psi$.
\end{theorem}
\begin{proof}
Let us recall that $\mathsf T$ is $\Pi$-sound by Corollary~\ref{cor:con-pi-sound}. We first prove the result under the assumption that $\mathsf T$ is $\Sigma$-sound as well. From Exercise~\ref{ex:prov-Sigma} we obtain a $\Sigma$-formula~$\pr_{\mathsf T}(x)$ such that $\mathsf T\vdash\varphi$ holds precisely if we have~$\mathbb N\vDash\pr_{\mathsf T}(\gn{\varphi})$. We use the diagonal lemma to get a sentence~$\psi$ with
\begin{equation*}
\mathsf T\vdash\psi\leftrightarrow\neg\pr_{\mathsf T}(\overline{\gn{\psi}}).
\end{equation*}
By this very equivalence, we can assume that $\psi$ is a $\Pi$-sentence. Intuitively, $\psi$ says `I am not provable', which is a variant of the so-called liar paradox. If we did have $\mathsf T\vdash\psi$ and hence $\mathsf T\vdash\neg\pr_{\mathsf T}(\overline{\gn{\psi}})$, we could use $\Pi$-soundness to get $\mathbb N\vDash\neg\pr_{\mathsf T}(\gn{\psi})$. The latter amounts to $T\nvdash\psi$, which we have thus established. In particular, we have seen that the so-called G\"odel sentence~$\psi$ is true. The latter also entails $\mathsf T\nvdash\neg\psi$, given that $\mathsf T$ is $\Sigma$-sound.

We now remove the assumption that $\mathsf T$ is $\Sigma$-sound. This yields an improvement of G\"odel's original result, which is due to Rosser. In view of Exercise~\ref{ex:neg-prim-rec}, we have a $\Sigma$-formula~$\operatorname{Neg}(x,y)$ such that
\begin{equation*}
\mathsf T\vdash\forall z\left(\operatorname{Neg}(\gn{\varphi},z)\leftrightarrow z=\gn{\neg\varphi}\right)
\end{equation*}
holds for any formula~$\varphi$. Due to Exercise~\ref{ex:sigma-sigma_1}, we may assume that $\pr_{\mathsf T}(y)$ has the form $\exists x\,\prf_{\mathsf T}(x,y)$ for a bounded formula $\prf_{\mathsf T}$. A number~$p$ with $\mathbb N\vDash\prf_{\mathsf T}(p,\gn{\varphi})$ may be interpreted as a proof of~$\varphi$. The notion of Rosser provability is defined by the $\Sigma$-formula $\pr_{\mathsf T}^{\operatorname{R}}(y)$ that is given as
\begin{equation*}
\exists x,z\left(\operatorname{Neg}(y,z)\land\prf_{\mathsf T}(x,y)\land\forall x'<x\,\neg\prf_{\mathsf T}(x',z)\right).
\end{equation*}
Given that~$\mathsf T$ is consistent, Rosser provability and usual provability coincide in the sense that $\mathbb N\vDash\pr_{\mathsf T}^{\operatorname{R}}(\overline{\gn{\varphi}})$ is equivalent to~$\mathsf T\vdash\varphi$. Let us now redefine the formula~$\psi$ from above by stipulating~$\mathsf T\vdash\psi\leftrightarrow\neg\pr_{\mathsf T}^{\operatorname{R}}(\overline{\gn{\psi}})$. Intuitively, $\psi$ says `for any proof of me there is a shorter proof of my negation'. Essentially the same argument as before yields~$\mathsf T\nvdash\psi$. To complete the proof, we assume $\mathsf T\vdash\neg\psi$ and derive a contradiction. In view of $\Sigma$-completeness, we find a number~$p$ with
\begin{equation*}
\mathsf T\vdash\prf_{\mathsf T}(\overline p,\overline{\gn{\neg\psi}})\quad\text{and}\quad\mathsf T\vdash\forall x<\overline p\,\,\neg\prf_{\mathsf T}(x,\overline{\gn{\psi}}).
\end{equation*}
Part~(c) of Exercise~\ref{ex:Q} allows us to infer
\begin{equation*}
\mathsf T\vdash\forall x\left(\prf_{\mathsf T}(x,\overline{\gn{\psi}})\to\exists x'<x\,\prf_{\mathsf T}(x',\overline{\gn{\neg\psi}})\right).
\end{equation*}
The latter entails $\mathsf T\vdash\neg\pr_{\mathsf T}^{\operatorname{R}}(\overline{\gn{\psi}})$ and hence $\mathsf T\vdash\psi$, which was already refuted.
\end{proof}

We now present versions of L\"ob's theorem and the second incompleteness theorem that are abstract in the sense that relevant properties of provability are included as assumptions. Some justification for the latter can be found in the paragraph after Corollary~\ref{cor:G2} below.

\begin{theorem}[L\"ob's theorem~\cite{loebs-theorem}]\label{thm:loeb}
Assume that~$\mathsf T\supseteq\mathsf Q$ and~$\pr_{\mathsf T}(y)$ are a theory and a formula of first-order arithmetic that validate the following (so-called Hilbert-Bernays conditions):
\begin{enumerate}[label=(D\arabic*)]
\item If we have $T\vdash\varphi$, we get $\mathsf T\vdash\pr_{\mathsf T}(\overline{\gn{\varphi}})$.
\item We have $\mathsf T\vdash\pr_{\mathsf T}(\overline{\gn{\varphi}})\to\pr_{\mathsf T}\big(\overline{\gn{\pr_{\mathsf T}(\overline{\gn{\varphi}})}}\big)$.
\item We have $\mathsf T\vdash\pr_{\mathsf T}(\overline{\gn{\varphi\to\psi}})\land\pr_{\mathsf T}(\overline{\gn{\varphi}})\to\pr_{\mathsf T}(\overline{\gn{\psi}})$.
\end{enumerate}
Given $\mathsf T\vdash\pr_{\mathsf T}\big(\overline{\gn{\varphi}}\big)\to\varphi$, we then get $\mathsf T\vdash\varphi$, for any sentence~$\varphi$.
\end{theorem}
\begin{proof}
The diagonal lemma provides a sentence~$\psi$ with
\begin{equation*}
\mathsf T\vdash\psi\leftrightarrow\left(\pr_{\mathsf T}(\overline{\gn{\psi}})\to\varphi\right).
\end{equation*}
Due to condition~(D1), we obtain
\begin{equation*}
\mathsf T\vdash\pr_{\mathsf T}\left(\overline{\gn{\psi\to\left(\pr_{\mathsf T}(\overline{\gn{\psi}})\to\varphi\right)}}\right).
\end{equation*}
By two applications of~(D3), we can infer
\begin{equation*}
\mathsf T\vdash\pr_{\mathsf T}(\overline{\gn{\psi}})\to\left(\pr_{\mathsf T}\left(\overline{\gn{\pr_{\mathsf T}(\overline{\gn{\psi}})}}\right)\to\pr_{\mathsf T}(\overline{\gn{\varphi}})\right).
\end{equation*}
In view of condition~(D2), the latter amounts to~$\mathsf T\vdash\pr_{\mathsf T}(\overline{\gn{\psi}})\to\pr_{\mathsf T}(\overline{\gn{\varphi}})$. Given that we have $\mathsf T\vdash\pr_{\mathsf T}\big(\overline{\gn{\varphi}}\big)\to\varphi$, we obtain $\mathsf T\vdash\pr_{\mathsf T}(\overline{\gn{\psi}})\to\varphi$ and hence~$\mathsf T\vdash\psi$. Another application of~(D1) yields $\mathsf T\vdash\pr_{\mathsf T}(\overline{\gn{\psi}})$, so that we finally get~$\mathsf T\vdash\varphi$.
\end{proof}

In the following result, the formula~$\neg\pr_{\mathsf T}(\overline{\gn{0=1}})$ is an arithmetization of the statement~$\mathsf T\nvdash 0=S0$, which asserts that~$\mathsf T$ is consistent. Intuitively, the result says that~$\mathsf T$ cannot prove its own consistency.

\begin{corollary}[G\"odel's second incompleteness theorem~\cite{goedel-incompleteness}]\label{cor:G2}
Let~$\mathsf T\supseteq\mathsf Q$ and~$\pr_{\mathsf T}(y)$ be a theory and a formula of first-order arithmetic such that the Hilbert-Bernays conditions are satisfied. If $\mathsf T$ is consistent, then we have $\mathsf T\nvdash\neg\pr_{\mathsf T}(\overline{\gn{0=S0}})$.
\end{corollary}
\begin{proof}
If the result was false, we would get $\mathsf T\vdash\pr_{\mathsf T}(\overline{\gn{0=S0}})\to 0=S0$, so that L\"ob's theorem would yield $\mathsf T\vdash 0=S0$, which is incompatible with the assumption that~$\mathsf T$ is consistent and contains~$\mathsf Q$.
\end{proof}

Concerning the Hilbert-Bernays conditions, we note that~(D1) is a consequence of~$\Sigma$-completeness if $\pr_{\mathsf T}$ is a $\Sigma$-definition of provability in~$\mathsf T$. Condition~(D3) is readily established for standard arithmetizations of provability, since the inference from~$\varphi\to\psi$ and $\varphi$ to the conclusion~$\psi$ is a single application of the cut rule (see the work of Robert Jeroslow~\cite{jeroslow-goedel} for results on cut-free proofs). Finally, condition~(D2) is an internal version of~(D1). To show that~(D2) is satisfied, one can check that our proof of $\Sigma$-completeness can be formalized in the theory~$\mathsf T$. When the latter is~$\mathsf{I\Sigma}_1$ or Peano arithmetic, this is relatively straightforward but demands a lot of diligent technical work. A detailed proof of the Hilbert-Bernays conditions can be found, e.\,g., in the textbook by Wolfgang Rautenberg~\cite{rautenberg-introduction}.

\begin{exercise}\label{ex:Rosser-provability}
(a) Show that we have $\mathsf Q\vdash\neg\pr_{\mathsf Q}^{\operatorname{R}}(\overline{\gn{0=S0}})$, where $\pr_{\mathsf Q}^{\operatorname{R}}$ refers to the notion of Rosser provability from the proof of Theorem~\ref{thm:goedel-one}. Also give an informal explanation why condition~(D2) can be expected to fail for Rosser provability.

(b) Consider $\psi$ with $\mathsf{T}\vdash\psi\leftrightarrow\neg\pr_{\mathsf{T}}(\overline{\gn{\psi}})$, as in the proof of Theorem~\ref{thm:goedel-one}. Assuming the Hilbert-Bernays conditions, show that we have
\begin{equation*}
\mathsf{T}\vdash\psi\leftrightarrow\neg\pr_{\mathsf{T}}(\overline{\gn{0=S0}}), 
\end{equation*}
i.\,e., that the G\"odel sentence is equivalent to consistency.
\end{exercise}

Part~(a) of the exercise shows that the second incompleteness theorem does not apply to some non-standard arithmetizations of provability. To explain why this does not diminuish the importance of the theorem, we now discuss its role in the context of Hilbert's program (see~\cite{zach19} for more details and further references). Let us first point out that the incompleteness theorems extend to axiom systems like set theory, which allow for a definition of the natural numbers and an interpretation of~$\mathsf Q$ but can also express mathematical statements that are much more abstract and involve~complex~infinite objects. In the foundational debate around 1900, several famous mathematicians suggested that these abstract statements may lack a clear meaning and could endager the integrity of mathematics. Hilbert's program has been proposed as a way to contain such objections and to justify the use of abstract methods: The idea is that we do not actually need to ascribe meaning to abstract notions when we wish to secure their concrete consequences, where the latter are identified with the $\Pi$-sentences of first-order arithmetic. By Corollary~\ref{cor:con-pi-sound}, it suffices to ensure that our abstract methods form a consistent theory, which is a concrete condition about finite proofs.

To be more precise, the hope behind Hilbert's program was that a fairly benign theory such as Peano arithmetic would prove the consistency of a powerful and abstract theory such as Zermelo-Fraenkel set theory with choice ($\mathsf{ZFC}$), i.\,e., that we would have $\mathsf{PA}\vdash\neg\pr_{\mathsf{ZFC}}(\overline{\gn{0=S0}})$. If this had been the case, any $\Pi$-sentence that is provable in~$\mathsf{ZFC}$ would have been provable in~$\mathsf{PA}$ as well. Indeed, when we have $\mathsf{ZFC}\vdash\psi$, we get $\mathsf{PA}\vdash\pr_{\mathsf{ZFC}}(\overline{\gn{\psi}})$ by a version of~(D1). Furthermore, a version of~(D2) yields $\mathsf{PA}\vdash\neg\psi\to\mathsf{Pr}_{\mathsf{ZFC}}(\overline{\gn{\psi\to 0=S0}})$, given that $\neg\psi$ is a $\Sigma$-sentence. Using a version of~(D3), we can now conclude $\mathsf{PA}\vdash\neg\psi\to\mathsf{Pr}_{\mathsf{ZFC}}(\overline{\gn{0=S0}})$. If we did have $\mathsf{PA}\vdash\neg\pr_{\mathsf{ZFC}}(\overline{\gn{0=S0}})$, we could thus infer $\mathsf{PA}\vdash\psi$, as promised. The point is that this argument relies on versions of the Hilbert-Bernays conditions. This provides another form of justification for the inclusion of these conditions in the second incompleteness theorem: the latter is most relevant -- at least in the context of Hilbert's program -- when the conditions are satisfied. Conversely, the counterexample from part~(a) of Exercise~\ref{ex:Rosser-provability} is not particularly relevant, because Rosser provability was no suitable basis for Hilbert's program in the first place.

In a sense, we would not gain much from a theory that proves its own consistency. An equivalent but more compelling formulation of the second incompleteness theorem might say that no theory proves the consistency of a stronger one. In particular, we see that $\mathsf{I\Sigma}_1$ cannot prove the consistency of~$\mathsf{PA}$. Conversely, one can show that $\mathsf{PA}$ proves the consistency of~$\mathsf{I\Sigma}_1$, so that the inclusion $\mathsf{I\Sigma}_1\subseteq\mathsf{PA}$ is strict in a strong sense. The second incompleteness theorem can thus be used to compare the strength of theories.

Despite G\"odel's theorems, the goal of Hilbert's program has been realized to an astonishing extent: it has been confirmed that a great deal of mathematics can be done in theories that are no stronger than~$\mathsf{I\Sigma}_1$ or at least than~$\mathsf{PA}$ (see~\cite{kohlenbach-higher-types,simpson09}). In Section~\ref{sect:goodstein} we will see a notable exception in the form of a concrete mathematical theorem that cannot be proved in~$\mathsf{I\Sigma}_1$.

\section{Proof Theory}\label{sect:proof-theory}

This section offers a glimpse into proof theory, which is one of the main subfields of mathematical logic. As the name suggests, proof theory is directly concerned with proofs and their structure. We have already seen one phenomenon that falls within the scope of elementary proof theory, namely the connection between cut-elimination and the size of Herbrand disjunctions (see Exercise~\ref{ex:Herbrand-supexp} and the paragraph that precedes it).

On a more concrete level, our aim in the present section is to show that the theory $\mathsf{I\Sigma}_1$ (see Definition~\ref{def:Q-PA}) cannot prove a certain mathematical theorem about non-hereditary Goodstein sequences. This is interesting since it yields an example for G\"odel's incompleteness theorems. It shows that incompleteness is not just a theoretical possibility but can come up in everyday mathematics, even though this is rare. The theorem about Goodstein sequences asserts that a specific computer program terminates on every input (i.\,e., the program does not enter an infinite loop). In the author's view, it is fascinating that incompleteness occurs at such a concrete level and not just in the realm of infinite sets. It is a strength of proof theory that its tools are sensitive enough to detect complexity and incompleteness even in the finite domain. To avoid misunderstanding, we note that there are very fruitful connections with computability and set theory, which are often associated with the countably infinite and the uncountable, respectively (see Sections~\ref{sect:computability} and~\ref{sect:set-theory}).

In Subsection~\ref{subsect:prov-tot-ISigma_1} below, we introduce the concept of provably total function, which captures the $\Pi_2$-theorems of an axiom system (cf.~the paragraph after Definition~\ref{def:Sigma-formula}). We then show that the provably total functions of~$\mathsf{I\Sigma}_0$ are precisely the primitive recursive functions. In Subsection~\ref{sect:goodstein}, we prove that Goodstein sequences give rise to a function that grows faster than any primitive recursive one, which yields the desired independence result.

Once again, we conclude the introduction with some pointers to the literature. First, this section is a very selective introduction to proof theory. The survey article by Michael Rathjen and Wilfried Sieg~\cite{rathjen-sieg-stanford} provides a more complete picture. Secondly, our presentation may give the false impression that proof theory is only concerned with foundations. In fact, it has applications in various parts of mathematics, in particular via the use of G\"odel's so-called Dialectica interpretation in proof mining (see the textbook by Ulrich Kohlenbach~\cite{kohlenbach-proof-mining}). Finally, we point out that Goodstein sequences do not yield the most convincing example for G\"odel's theorems, because they do not seem to have many connections with other mathematical results. A much more convincing (but also more difficult) independence result is provided by the graph minor theorem, which is one of the towering achievements of modern mathematics and has important consequences in theoretical computer science. As shown by Harvey Friedman, Neil Robertson and Paul Seymour~\cite{friedman-robertson-seymour}, this theorem is independent of an axiom system denoted by $\Pi^1_1\text{-}\mathsf{CA}_0$, which is far stronger than~$\mathsf{I\Sigma}_1$. A main ingredient for this independence result is a proof-theoretic method known as ordinal analysis. For more information about the latter, we refer to another survey by Michael Rathjen~\cite{rathjen-realm} and to further lecture notes by the present author~\cite{first-course,second-course}.

\subsection{Provably Total Functions of $I\Sigma_1$}\label{subsect:prov-tot-ISigma_1}

In this section, we show that the computational content of the theory~$\mathsf{I\Sigma}_1$ is captured by the notion of primitive recursion. We know from Exercise~\ref{ex:ISigma_1-to-primrec} that any primitive recursive function is \mbox{$\mathsf{I\Sigma}_1$-}provably total in the sense of the following definition (except that some relevant facts about the encoding of sequences were not proved in the exercise). In the following we show that, conversely, any $\mathsf{I\Sigma}_1$-provably total function is primitive recursive.

\begin{definition}
Consider a theory~$\mathsf T$ over the signature of first-order arithmetic. A function $f:\mathbb N\to\mathbb N$ is called $\mathsf T$-provably total (or also $\mathsf T$-provably recursive) if we have $\mathsf T\vdash\forall x\exists y\,\varphi(x,y)$ for some $\Sigma_1$-formula~$\varphi$ that defines~$f$ (i.\,e., such that we have $f(a)=b$ precisely when $\mathbb N\vDash\varphi(a,b)$ holds).
\end{definition}

We point out that the definition would be trivial without the requirement that $\varphi$ is a $\Sigma_1$-formula (or in general: no more complex than the definition of~$f$ requires). Indeed, if $f$ is defined by~$\varphi(x,y)$, it is also defined by the formula $\psi(x,y)$ that is given as~$\forall u\exists w\,\varphi(u,w)\to\varphi(x,y)$. However, $\forall x\exists y\,\psi(x,y)$ is provable in any~theory~$\mathsf T$. Our definition is non-trivial since $\psi$ is no $\Sigma_1$-formula. The following shows that we could also require~$\mathsf T$ to prove unique existence (cf.~the paragraph before Exercise~\ref{ex:ISigma_1-to-primrec}).

\begin{exercise}
Assume that~$f$ is $\mathsf T$-provably total for a theory $\mathsf T$ that contains $\mathsf{I\Sigma}_1$ (even though less induction would also suffice). Show that we have $\mathsf T\vdash\forall x\exists! y\,\varphi(x,y)$ for some $\Sigma_1$-formula~$\varphi$ that defines~$f$. \emph{Hint:} Suppose we have $\mathsf T\vdash\forall x\exists y\exists z\,\theta(x,y,z)$ for some $\Sigma_1$-formula~$\exists z\,\theta(x,y,z)$ that defines~$f$. Let $\varphi(x,y)$ assert that $y$ is the first component of the minimal pair~$\langle y,z\rangle$ such that $\theta(x,y,z)$ holds.
\end{exercise}

In order to bound the $\mathsf{I\Sigma}_1$-provably total functions, we will use the following so-called asymmetric interpretation.

\begin{definition}\label{def:asym-int}
Given $m,n\in\mathbb N$ and a formula~$\varphi$ of first-order arithmetic, we define $\varphi^{m,n}$ as the formula that results from~$\varphi$ when all unbounded quantifiers~$\forall x$ and $\exists x$ are replaced by $\forall x\leq\overline m$ and $\exists x\leq\overline n$, respectively. Note that bounded quantifiers (cf.~Definition~\ref{def:Sigma-formula}) are unaffected, so that we get
\begin{alignat*}{3}
\varphi^{m,n}&=\forall x\left(x\leq t\to\psi^{m,n}\right)\quad&&\text{for}\quad\varphi=\forall x(x\leq t\to\psi)\text{ with }x\notin\fv(t),\\
\varphi^{m,n}&=\exists x\left(x\leq t\land\psi^{m,n}\right)\quad&&\text{for}\quad\varphi=\exists x(x\leq t\land\psi)\text{ with }x\notin\fv(t).
\end{alignat*}
Furthermore, we have
\begin{equation*}
(\forall x\,\varphi)^{m,n}=\forall x\leq\overline m\,\varphi^{m,n}\quad\text{and}\quad(\exists x\,\varphi)^{m,n}=\exists x\leq\overline n\,\varphi^{m,n}
\end{equation*}
in the case of unbounded quantifiers, i.\,e., when $\varphi$ is not of the form $x\leq t\to\psi$ or $x\leq t\land\psi$ with $x\notin\fv(t)$, respectively. In order to complete the recursive characterization of~$\varphi^{m,n}$, we note that $\varphi^{m,n}$ coincides with $\varphi$ when the latter is a literal and that we have
\begin{equation*}
(\varphi\land\psi)^{m,n}=\varphi^{m,n}\land\psi^{m,n}\quad\text{and}\quad(\varphi\lor\psi)^{m,n}=\varphi^{m,n}\lor\psi^{m,n}.
\end{equation*}
We write $\vDash^{m,n}\varphi$ when $\mathbb N,\eta\vDash\varphi^{m,n}$ holds for every~$\eta$ with $\eta(i)\leq m$ for all~$i\in\mathbb N$. Given $f:\mathbb N\to\mathbb N$, we write $\vDash^f\varphi$ if we have $\vDash^{m,f(m)}\varphi$ for all~$m\in\mathbb N$.
\end{definition}

The condition~$\eta(i)\leq m$ ensures that $\vDash^{m,n}\varphi$ is equivalent to~$\vDash^{m,n}\varphi'$ when $\varphi'$ is the universal closure of~$\varphi$. Let us also note that $\vDash^{m,n}\varphi$ entails $\vDash^{k,l}\varphi$ when we have $k\leq m$ and $l\geq n$. It follows that $\vDash^f\varphi$ entails $\vDash^g\varphi$ when $f$ is dominated by~$g$, i.\,e., when we have $f(m)\leq g(m)$ for all~$m\in\mathbb N$.

It will be shown that $\mathsf{I\Sigma}_1\vdash\psi$ entails $\vDash^g\psi$ for some primitive recursive~$g$. In~order to explain how this is applied, we now assume that $\psi$ has the form~$\exists y\,\varphi(x,y)$ for some $\Sigma$-formula~$\varphi$ that defines a function~$f$. Then each $m\in\mathbb N$ admits an $n\leq g(m)$ with $\vDash^{m,g(m)}\varphi(m,n)$. Since $\varphi$ is a $\Sigma_1$-formula, we can conclude $\mathbb N\vDash\varphi(m,n)$ and hence $f(m)=n\leq g(m)$. This does already show that $f$ is dominated by a primitive recursive function, which will be enough for the application in Section~\ref{sect:goodstein}. Due to
\begin{equation*}
f(m)=\min\left\{n\leq g(m)\,\left|\,\vDash^{m,g(m)}\varphi(m,n)\right.\right\},
\end{equation*}
we even learn that $f$ is primitive recursive, by Proposition~\ref{prop:bounded-min} and the following.

\begin{exercise}\label{ex:bounded-prim-rec}
For $c\in\mathbb N$ we write $c\vDash\varphi$ if we have $\mathbb N,\eta\vDash\varphi$ with $\eta(i):=(c)_i$ (using the sequence encoding from Definition~\ref{def:sequence-encoding}). Employ Proposition~\ref{prop:bounded-form-prim-rec}  to show that~$\{c\in\mathbb N\,|\,c\vDash\varphi\}$ is primitive recursive for each bounded formula~$\varphi$. \emph{Remark:} A somewhat more difficult argument yields the stronger result that
\begin{equation*}
\{(c,\ulcorner\varphi\urcorner)\in\mathbb N\,|\,\text{$\varphi$ is bounded and }c\vDash\varphi\}
\end{equation*}
is a primitive recursive relation.
\end{exercise}

In order to obtain a primitive recursive~$g$ with $\vDash^g\varphi$, we would like to use induction over a given proof~$\mathsf I\Sigma_1\vdash\varphi$. From Section~\ref{subsect:FO} we know that there is a cut-free proof~$\vdash_0\neg\theta_1,\ldots,\neg\theta_n,\varphi$ for suitable axioms~$\theta_i$ of~$\mathsf I\Sigma_1$. As we will see, one does obtain $\vDash^g\neg\theta_1\lor\ldots\lor\neg\theta_n\lor\varphi$ for some~$g$ that is primitive recursive (in fact even less complex). Given that the~$\theta_i$ are axioms of~$\mathsf I\Sigma_1$, one can also expect to come up with primitive recursive functions~$f_i$ that validate~$\vDash^{f_i}\theta_i$. Unfortunately, it does not seem possible to combine the given data into a primitive recursive~$h$ with $\vDash^h\varphi$, so that the argument breaks down at this point. In order to illustrate the difficulty in a somewhat simplified case, we assume that we have
\begin{equation*}
\vDash^f\forall x\exists y\,\theta\quad\text{and}\quad\vDash^g\varphi\lor\exists x\forall y\,\neg\theta
\end{equation*}
with bounded~$\theta$. Our aim is to deduce~$\vDash^h\varphi$ for a suitable~$h$. Assuming that~$\varphi$ is closed, the task is to show that $\mathbb N\vDash\varphi^{m,h(m)}$ holds for all~$m\in\mathbb N$. Since we have
\begin{equation*}
\mathbb N\vDash\varphi^{m,g(m)}\lor\exists x\leq g(m)\forall y\leq m\,\neg\theta,
\end{equation*}
an obvious idea is to take~$h=g$ and to aim at $\mathbb N\vDash\forall x\leq g(m)\exists y\leq m\,\theta$. However, the latter does not match the statements~$\mathbb N\vDash\forall x\leq n\exists y\leq f(n)\,\theta$ that are given (as we typically have $m<g(m)$ and $n<f(n)$ while we need $g(m)\leq n$ and $f(n)\leq m$). To circumvent the indicated issue, we now introduce a proof calculus that restricts the complexity of cut formulas and replaces the induction axioms by a rule.

\begin{definition}\label{def:vdashir}
For a sequent~$\Gamma$ over the signature of first-order arithmetic, we define $\vdashir\Gamma$ by the following recursive clauses (cf.~Definition~\ref{def:sequent-calc-FO}): 
\begin{description}[labelwidth=8ex,labelindent=\parindent,leftmargin=!,before={\renewcommand\makelabel[1]{(##1)}}]
\item[$\mathsf{Ax}$] When $\Gamma$ contains some literal and its negation, we have $\vdashir\Gamma$.
\item[$\land$] Given $\vdashir\Gamma,\varphi$ and $\vdashir\Gamma,\psi$, we get $\vdashir\Gamma,\varphi\land\psi$.
\item[$\lor$] From $\vdashir\Gamma,\psi$ we get both $\vdashir\Gamma,\psi\lor\theta$ and $\vdashir\Gamma,\varphi\lor\psi$.
\item[$\forall$] If we have $\vdashir\Gamma',\varphi[x/y]$ for some $\Gamma'\subseteq\Gamma$ and some~$y$ that satisfies the variable condition~$y\notin\fv(\Gamma',\forall x\,\varphi)$, then we get $\vdashir\Gamma,\forall x\,\varphi$.
\item[$\exists$] Given $\vdashir\Gamma,\varphi[x/t]$ for some term~$t$, we obtain $\vdashir\Gamma,\exists x\,\varphi$.
\item[$\Sigma\text{-}\mathsf{Cut}$] When have $\vdashir\Gamma,\varphi$ and $\vdashir\Gamma,\neg\varphi$ for some $\Sigma$-formula~$\varphi$, we get $\vdash\Gamma$.
\item[$\mathsf{IR}$] If we have $\vdashir\Gamma',\neg\varphi[x/y],\varphi[x/Sy]$ for $\Gamma'\subseteq\Gamma$, some $\Sigma$-formula~$\varphi$ and a variable $y\notin\fv(\Gamma',\forall x\,\varphi)$, we get $\vdashir\Gamma,\neg\varphi[x/0],\varphi[x/t]$ for any term~$t$.
\end{description}
\end{definition}

Given the induction rule, the following shows that we can derive $\vdashir\mathsf I\varphi$ for any $\Sigma$-formula~$\varphi$ (cf.~Exercise~\ref{ex:nega-pred}):

\begin{prooftree}
\AxiomC{$\varphi(x),\neg\varphi(x),\varphi(Sx)$}
\AxiomC{$\neg\varphi(Sx),\neg\varphi(x),\varphi(Sx)$}
\BinaryInfC{$\varphi(x)\land\neg\varphi(Sx),\neg\varphi(x),\varphi(Sx)$}
\UnaryInfC{$\exists x(\varphi(x)\land\neg\varphi(Sx)),\neg\varphi(x),\varphi(Sx)$}
\UnaryInfC{$\exists x(\varphi(x)\land\neg\varphi(Sx)),\neg\varphi(0),\varphi(x)$}
\UnaryInfC{$\exists x(\varphi(x)\land\neg\varphi(Sx)),\neg\varphi(0),\forall x\,\varphi(x)$}
\UnaryInfC{$\neg\varphi(0)\lor\exists x(\varphi(x)\land\neg\varphi(Sx))\lor\forall x\,\varphi(x)$}
\end{prooftree}
\vspace*{.5\baselineskip}

\noindent We would like to infer that $\mathsf{I\Sigma}_1\vdash\psi$ entails $\vdashir\neg\theta_1,\ldots,\neg\theta_n,\psi$ for suitable axioms~$\theta_i$ of Robinson arithmetic. In view of Theorem~\ref{thm:completeness-AL}, this could be achieved via cuts over the induction formulas~$\mathsf I\varphi$. However, these cuts are not permitted in the proof system from Definition~\ref{def:vdashir} (for reasons that were explained before that definition). One solution is to give a syntactic proof of cut-elimination (see, e.\,g.,~\cite{buss-introduction-98} for this classical approach). In our setting, it seems simpler to give a direct completeness proof, which amounts to another proof of cut-elimination by semantic methods.

Before we state the indicated completeness result, we note that it is unproblematic to have axioms that are $\Pi$-formulas, since their negations will be $\Sigma$-formulas, which are covered by the restricted cut rule~($\Sigma\text{-}\mathsf{Cut}$) from Definition~\ref{def:vdashir}. This applies to all axioms of Robinson's arithmetic~$\mathsf Q$ with the exception of
\begin{equation*}
\forall x,y\big(x\leq y\leftrightarrow\exists z\,z+x=y\big).
\end{equation*}
Let $\mathsf Q'$ be the theory that results from~$\mathsf Q$ when we replace this axiom by
\begin{equation*}
\forall x,y\big((x\leq y\rightarrow\exists z\leq y\,z+x=y)\land(\exists z\,z+x=y\to x\leq y)\big).
\end{equation*}
Note that the existential quantifier in the second conjunct occurs in the antecedent and will thus count as universal when the formula is brought into negation normal form. The point is that~$\mathsf Q'$ consists of true $\Pi$-sentences and entails~$\mathsf Q$.

\begin{theorem}\label{thm:compl-Sigma_1}
If we have $\mathsf{I\Sigma}_1\vdash\psi$, then we get $\vdashir\neg\theta_0,\ldots,\neg\theta_n,\psi$ for suitable axioms~$\theta_i$ of the theory~$\mathsf Q'$.
\end{theorem}
\begin{proof}
We adapt the completeness proofs that were given in Sections~\ref{subsect:complete-PL} and~\ref{subsect:FO}. As a preparation, we fix sequences~$\theta_0,\theta_1,\ldots$, $\mathsf I\varphi_0,\mathsf I\varphi_1,\ldots$ and $t_0,t_1,\ldots$ that enumerate all axioms of~$\mathsf Q'$, all instances of $\Sigma$-induction and all terms, respectively. Once again, we recursively define a binary tree~$T$ and an ordered sequent~$\Gamma(\sigma)$ for each~$\sigma\in T$. In the base case, we declare that we have $\langle\rangle\in T$ and $\Gamma(\langle\rangle)=\langle\psi\rangle$. For the recursion step, we distinguish cases according to the length~$l(\sigma)$ of a given $\sigma\in T$ such that $\Gamma(\sigma)$ is already determined. If we have $l(\sigma)=6n$, we declare that $\sigma\star i\in T$ holds for both $i<2$ and that we have~$\Gamma(\sigma\star0)=\Gamma(\sigma),\theta_n$ and $\Gamma(\sigma\star1)=\Gamma(\sigma),\neg\theta_n$, as in the proof of Theorem~\ref{thm:completeness-prop}. We also declare that $\sigma\star 0$ is a leaf of~$\mathsf T$ and that the recursion step at~$\tau:=\sigma\star 1$ is carried out as in the proofs of Propositions~\ref{prop:completeness} and~\ref{prop:completeness-FO}. More explicitly, this means that we first check whether~$\Gamma(\tau)$ contains some atomic formula together with its negation. If this is the case, we declare that $\tau$ is a leaf of~$T$. Otherwise, we write $\Gamma(\tau)=\varphi,\Delta$ and distinguish cases according to the form of~$\varphi$. As an example, we recall that $\varphi=\psi_0\land\psi_1$ leads to $\Gamma(\tau\star i)=\Delta,\varphi,\psi_i$ for both~$i<2$.

Let us now consider the recursive definition above a sequence~$\sigma\in T$ that has length~$l(\sigma)=6n+2$. Consider~$i,j,k$ such that we have $n=\pi(i,k)$ and $\forall v_{2j}\,\varphi_i$ is the conclusion of~$\mathsf I\varphi_i$ (cf.~Definition~\ref{def_fv-formula} and Exercise~\ref{ex:Cantor-pairing}). Let~$\varphi_i(t)$ denote~$\varphi_i[v_{2j}/t]$. The tree above~$\sigma$ is to look as follows, where~$y$ is some fresh variable:

\begin{prooftree}
\AxiomC{$\Gamma(\sigma),\varphi_i(0)$}
\AxiomC{$\Gamma(\sigma),\neg\varphi_i(0),\neg\varphi_i(t_k)$}
\AxiomC{$\Gamma(\sigma),\neg\varphi_i(y),\varphi_i(Sy)$}
\UnaryInfC{$\Gamma(\sigma),\neg\varphi_i(0),\varphi_i(t_k)$}
\BinaryInfC{$\Gamma(\sigma),\neg\varphi_i(0)$}
\BinaryInfC{$\Gamma(\sigma)$}
\end{prooftree}
\vspace*{.5\baselineskip}
More explicitly, we declare that $\sigma\star p\star q\star r\in T$ holds for all~$p,q,r<2$. Still for arbitrary~$q,r<2$, the associated sequents are given by
\begin{align*}
\Gamma(\sigma\star 0)=\Gamma(\sigma\star 0\star q)=\Gamma(\sigma\star 0\star q\star r)&=\Gamma(\sigma),\varphi_i(0),\\
\Gamma(\sigma\star 1)&=\Gamma(\sigma),\neg\varphi_i(0),\\
\Gamma(\sigma\star 1\star 0)=\Gamma(\sigma\star 1\star 0\star r)&=\Gamma(\sigma),\neg\varphi_i(0),\neg\varphi_i(t_k),\\
\Gamma(\sigma\star 1\star 1)&=\Gamma(\sigma),\neg\varphi_i(0),\varphi_i(t_k),\\
\Gamma(\sigma\star 1\star 1\star r)&=\Gamma(\sigma),\neg\varphi_i(y),\varphi_i(Sy).
\end{align*}
For definiteness, one may take~$y$ to be~$v_m$ for the smallest odd number~$m$ such that we have $v_m\notin\fv(\Gamma(\sigma),\forall v_{2j}\,\varphi)$. Let us note that we have covered all sequences~$\tau$ of length $l(\tau)\leq 6n+5$, so that our recursive definition is complete.

As in the previous completeness proofs, we now distinguish two cases. Let us first assume that~$T$ has a branch~$f:\mathbb N\to\{0,1\}$. By the proof of Proposition~\ref{prop:completeness-FO}, we obtain a model~$\mathcal M$ with universe~$M=\{t_k\,|\,k\in\mathbb N\}$ and an assignment~$\eta:\mathbb N\to M$ with $t^{\mathcal M,\eta}=t$ such that we have $\mathcal M,\eta\nvDash\rho$ for any formula~$\rho$ that occurs on~$f$. In particular, this applies when~$\rho$ is the formula~$\psi\in\Gamma(\langle\rangle)$ that is given in the present theorem. As in the proof of Theorem~\ref{thm:completeness-prop}, we also get~$\mathcal M,\eta\vDash\mathsf Q'$ and hence~$\mathcal M,\eta\vDash\mathsf Q$ (since each $\theta_k$ occurs at a leaf, so that~$\neg\theta_k$ must lie on~$f$). To reach a contradiction with the given assumption~$\mathsf{I\Sigma}_1\vdash\psi$, we now show that~$\mathcal M,\eta\vDash\mathsf I\varphi_i$ holds for any~$i\in\mathbb N$. We still write $\varphi_i(t)$ for $\varphi_i[v_{2j}/t]$, where $\forall v_{2j}\,\varphi_i$ is the conclusion of~$\mathsf I\varphi_i$. Let us assume that we have
\begin{equation*}
\mathcal M,\eta\vDash\varphi_i(0)\quad\text{and}\quad\mathcal M,\eta\vDash\forall x\,(\varphi_i(x)\to\varphi_i(Sx)).
\end{equation*}
To obtain the desired conclusion~$\mathcal M,\eta\vDash\forall x\,\varphi_i(x)$, we need to show that $\mathcal M,\eta_{2j}^{t_k}\vDash\varphi_i$ or equivalently~$\mathcal M,\eta\vDash\varphi_i(t_k)$ holds for arbitrary~$k$ (cf.~Exercise\ref{ex:substitution}). Let us consider the construction of~$T$ above the sequence~$f[6n+2]$ with $n=\pi(i,k)$. The following entails that~$\neg\varphi_i(t_k)$ occurs on~$f$, which is enough to conclude. First, the formula~$\varphi_i(0)$ cannot occur on~$f$, as this would yield~$\mathcal M,\eta\nvDash\varphi_i(0)$. Secondly, it is not possible that both~$\neg\varphi_i(y)$ and $\varphi_i(Sy)$ occur on the branch~$f$, since this would allow us to conclude~$\mathcal M,\eta\vDash\exists x(\varphi_i(x)\land\neg\varphi_i(Sx))$.

In the remaining case, the tree~$T$ is finite due to K\H{o}nig's lemma. Analogous to the proof of Theorem~\ref{thm:completeness-prop}, we consider the height function $h:T\to\mathbb N$ and pick a number~$n$ with $h(\langle\rangle)\leq 6n$. An induction over~$h(\sigma)$ shows $\vdashir\neg\theta_0,\ldots,\neg\theta_n,\Gamma(\sigma)$, which amounts to the conclusion of the theorem when we take~$\sigma=\langle\rangle$. Crucially, each negated axiom~$\neg\theta_k$ and each instance~$\varphi_i(t)$ is a $\Sigma$-formula, so that the required cuts are permissible according to Definition~\ref{def:vdashir}. For~$\sigma\in T$ with $l(\sigma)=6n+2$, the inference from $\Gamma(\sigma\star1\star1\star0)$ to $\Gamma(\sigma\star1\star1)$ is justified by our induction rule.
\end{proof}

The previous theorem is particularly relevant in view of the following result.

\begin{theorem}\label{thm:Sigma_1-prim-rec}
From $\vdashir\Gamma$ we get $\vDash^g\bigvee\Gamma$ for some primitive recursive $g:\mathbb N\to\mathbb N$.
\end{theorem}
\begin{proof}
We argue by induction over the given proof~$\vdashir\Gamma$, i.\,e., over the steps by which the latter is obtained according to the clauses from Definition~\ref{def:vdashir}. Let us first consider the case where $\vdashir\Gamma$ was derived by clause~($\mathsf{Ax}$), so that~$\Gamma$ contains some atomic formula~$\theta$ together with its negation. Our task is to find a primitive recursive~$g$ such that all $m\in\mathbb N$ and~$\eta:\mathbb N\to\{0,\ldots,m\}$ validate~$\mathbb N,\eta\vDash(\bigvee\Gamma)^{m,g(m)}$. The latter reduces to~$\mathbb N,\eta\vDash\theta^{m,g(m)}\lor(\neg\theta)^{m,g(m)}$. Since $\theta$ and~$\neg\theta$ are quantifier-free, they coincide with~$\theta^{m,g(m)}$ and $(\neg\theta)^{m,g(m)}$, respectively. In particular, it follows that the formulas $(\neg\theta)^{m,g(m)}$ and $\neg(\theta^{m,g(m)})$ coincide (which is not the case for formulas with quantifiers), so that the claim holds for arbitrary~$g$.

Next, we assume that $\vdashir\Gamma$ has been deduced by clause~($\land$), so that the sequent~$\Gamma$ is of the form~$\Delta,\varphi\land\psi$ and we have premises~$\vdashir\Delta,\varphi$ and $\vdashir\Delta,\psi$. The induction hypothesis yields primitive recursive functions~$f,g$ such that we have $\vDash^f\varphi\lor\bigvee\Delta$ and $\vDash^g\psi\lor\bigvee\Delta$. The function $h:\mathbb N\to\mathbb N$ with $h(n)=\max(f(n),g(n))$ is also primitive recursive. As noted after Definition~\ref{def:asym-int}, we obtain $\vDash^h\varphi\lor\bigvee\Delta$ and $\vDash^h\psi\lor\bigvee\Delta$. It is not hard to derive~$\vDash^h\bigvee\Gamma$, since the asymmetric interpretation distributes over conjunctions and disjunctions (see the relevant clauses in Definition~\ref{def:asym-int}). In the case of clause~($\lor$), this yields an even simpler reduction to the induction hypothesis.

Now assume that $\Gamma$ has the form~$\Delta,\forall v_i\,\varphi$ and was deduced by clause~($\forall$) with premise~$\vdashir\Delta',\varphi[v_i/v_j]$ for~$\Delta'\subseteq\Delta$ and $v_j\notin\fv(\Delta',\forall v_i\,\varphi)$. When the quantifier in~$\forall v_i\,\varphi$ is unbounded, we obtain a straightforward reduction to the induction hypothesis (similarly to the proof of Proposition~\ref{prop:soundness-FO}), since the definition of~$\vDash^{m,n}$ places the same bound on universal quantifiers and on the values of variable assignments. Perhaps surprisingly, the case of a bounded quantifier is the more interesting one. Here~$\varphi$ has the form~$v_i\leq t\to\psi$ with $v_i\notin\fv(t)$, which means that~$\forall v_i\,\varphi$~coincides with~$\forall v_i\leq t\,\psi$. The induction hypothesis yields a primitive recursive~$f$ that validates~$\vDash^f\Delta',\varphi[v_i/v_j]$. We may assume that~$f$ is increasing, as it is dominated by the increasing function $m\mapsto\sum_{i\leq m}f(i)$, which is still primitive recursive. Also, there is a primitive recursive function~$\tau:\mathbb N\to\mathbb N$ with $t^{\mathbb N,\eta}\leq\tau(m)$ for all $m\in\mathbb N$ and~$\eta:\mathbb N\to\{0,\ldots,m\}$, since each term is a polynomial in several variables. We may assume~$m\leq\tau(m)$. Consider the primitive recursive function~$g:\mathbb N\to\mathbb N$ that is given by~$g(m)=f(\tau(m))\geq f(m)$. In order to establish $\vDash^g\Gamma$, we consider arbitrary~$m\in\mathbb N$ and $\eta:\mathbb N\to\{0,\ldots,m\}$. If we have $\mathbb N,\eta\vDash\rho^{m,g(m)}$ for some~$\rho\in\Delta'\subseteq\Gamma$, then we are done. So we assume otherwise. Our task is to show~$\mathbb N,\eta\vDash\forall v_i\,\varphi^{m,g(m)}$. As in the proof of Proposition~\ref{prop:soundness-FO}, it suffices to establish~$\mathbb N,\eta_j^a\vDash\varphi^{m,g(m)}[v_i/v_j]$ for an arbitrary~$a\in\mathbb N$. It is straightforward to check that $\varphi^{m,g(m)}[v_i/v_j]$ coincides with $\varphi[v_i/v_j]^{m,g(m)}$. Also, the formula~$\varphi[v_i/v_j]$ coincides with~$v_j\leq t\to\psi[v_i/v_j]$, since we have $v_i\notin\fv(t)$ and hence $v_j\notin\fv(t)\subseteq\fv(\forall v_i\,\varphi)$. The latter does also yield~$t^{\mathbb N,\eta_j^a}=t^{\mathbb N,\eta}\leq\tau(m)$. It is thus enough to show that~$\mathbb N,\eta_j^a\vDash\varphi[v_i/v_j]^{m,g(m)}$ holds for arbitrary~$a\leq\tau(m)$. The result of the induction hypothesis entails
\begin{equation*}
\mathbb N,\eta_j^a\vDash\varphi[v_i/v_j]^{\tau(m),g(m)},\left(\bigvee\Delta'\right)^{\tau(m),g(m)}.
\end{equation*}
For $\rho\in\Delta'$ we have $v_j\notin\fv(\rho)$, so that $\mathbb N,\eta_j^a\vDash\rho^{\tau(m),g(m)}$ implies $\mathbb N,\eta\vDash\rho^{m,g(m)}$, against our assumption. So we indeed get $\mathbb N,\eta_j^a\vDash\varphi[v_i/v_j]^{m,g(m)}$.

We come to the case where $\Gamma$ has the form~$\Delta,\exists v_i\,\varphi$ and was deduced by clause~($\exists$) with premise~$\vdashir\Delta,\varphi[v_i/t]$. The induction hypothesis provides a primitive recursive function~$f$ that validates $\vDash^f\varphi[v_i/t]^{m,f(m)}\lor(\bigvee\Delta)^{m,f(m)}$. Since we can replace~$f$ by any function that dominates it, we may assume that~$t^{\mathbb N,\eta}\leq f(m)$ holds for any number~$m\in\mathbb N$ and all~$\eta:\mathbb N\to\{0,\ldots,m\}$. This ensures that $\mathbb N,\eta\vDash\varphi[v_i/t]^{m,f(m)}$ implies $\mathbb N,\eta\vDash(\exists v_i\,\varphi)^{m,f(m)}$, which is enough to secure the induction step.

In the penultimate case, the sequent~$\Gamma$ has been derived by clause~($\Sigma\text{-}\mathsf{Cut}$), which means that we have premises $\vdashir\Gamma,\varphi$ and $\vdashir\Gamma,\neg\varphi$ for some~$\Sigma$-formula~$\varphi$. By the induction hypothesis, there are primitive recursive functions~$f,g:\mathbb N\to\mathbb N$ that validate $\vDash^f\varphi\lor\bigvee\Gamma$ and $\vDash^g\neg\varphi\bigvee\Gamma$. We may assume that $n\leq f(n)$ and $n\leq g(n)$ holds for all~$n\in\mathbb N$. Our aim is to show that $\vDash^h\bigvee\Gamma$ holds for the primitive recursive function~$h:\mathbb N\to\mathbb N$ that is given by~$h(m)=g(f(m))$. Consider arbitrary~$m\in\mathbb N$ and $\eta:\mathbb N\to\{0,\ldots,m\}$. If we have $\mathbb N,\eta\vDash\rho^{m,f(m)}$ for some~$\rho\in\Gamma$, then we are done. Let us now assume~$\mathbb N,\eta\vDash\varphi^{m,f(m)}$. The crucial observation is that the formulas $\neg(\varphi^{m,f(m)})$ and~$(\neg\varphi)^{f(m),h(m)}$ coincide, as all unbounded quantifiers in~$\varphi$ and $\neg\varphi$ are existential and universal, respectively, so that the asymmetric interpretation restricts them to~$f(m)$ in both cases. Given that we have $\mathbb N,\eta\nvDash(\neg\varphi)^{f(m),h(m)}$, we can invoke~$\vDash^g\neg\varphi\lor\bigvee\Gamma$ to conclude that $\mathbb N,\eta\vDash\rho^{f(m),h(m)}$ holds for some~$\rho\in\Gamma$, as required to deduce~$\vDash^h\Gamma$.

Finally, we consider clause~($\mathsf{IR}$), so that $\Gamma$ has the form~$\Delta,\neg\varphi[v_i/0],\varphi[v_i/t]$ for some $\Sigma$-formula~$\varphi$. The rule has premise~$\vdashir\Delta',\neg\varphi[v_i/v_j],\varphi[v_i/Sv_j]$ with \mbox{$\Delta'\subseteq\Delta$} and $v_j\notin\fv(\Delta',\forall v_i\,\varphi)$. By the induction hypothesis, we have a primitive recursive function~$f$ that validates $\vDash^f\neg\varphi[v_i/v_j]\lor\varphi[v_i/Sv_j]\lor\bigvee\Delta'$. For an intuitive explanation of the proof idea, we temporarily assume that $\varphi(v_j)$ has the form $\exists x\,\theta(v_j,x)$ for a bounded formula~$\theta$. Given any $n\in\mathbb N$ and $\eta:\mathbb N\to n$, we then have
\begin{equation*}
\mathbb N,\eta\vDash\big(\exists x\leq n\,\theta(v_j,x)\to\exists x\leq f(n)\,\theta(S v_j,x)\big)\lor\left(\bigvee\Delta'\right)^{n,f(n)}.
\end{equation*}
Now the idea is to iterate~$f$ in order to bound the quantifier in~$\exists x\,\theta(v_j,x)$ for larger and larger~$v_j$. Back on the formal level, we find a primitive recursive~$\tau:\mathbb N\to\mathbb N$~such that $t^{\mathbb N,\eta}\leq\tau(m)$ holds for any~$m\in\mathbb N$ and $\eta:\mathbb N\to\{0,\ldots,m\}$. As before, we may assume that~$m\leq f(m)$ and~$m\leq\tau(m)$ holds for all~$m\in\mathbb N$. Let us consider the primitive recursive~$g:\mathbb N^2\to\mathbb N$ that is determined by
\begin{equation*}
g(m,0)=\tau(m)\quad\text{and}\quad g(m,n+1)=f(g(m,n)).
\end{equation*}
Our aim is to show that $\vDash^h\Gamma$ holds for the primitive recursive~$h:\mathbb N\to\mathbb N$ that is given by~$h(m)=g(m,\tau(m))$. Consider arbitrary~$m\in\mathbb N$ and $\eta:\mathbb N\to\{0,\ldots,m\}$. If we have $\mathbb N,\eta\vDash\rho^{m,h(m)}$ for some~$\rho\in\Delta'\subseteq\Delta$, then we are done. So we assume otherwise. Similarly, we may assume that we have~$\mathbb N,\eta\nvDash(\neg\varphi[v_i/0])^{m,h(m)}$. Given that $\varphi$ is a $\Sigma$-formula, we see that $(\neg\varphi[v_i/0])^{m,h(m)}$ coincides with $\neg(\varphi[v_i/0]^{k,m})$ for an arbitrary~$k\in\mathbb N$, which means that we obtain~$\mathbb N,\eta_i^0\vDash\varphi^{k,g(m,0)}$. By induction on~$n\leq\tau(m)$, we derive that $\mathbb N,\eta_i^n\vDash\varphi^{k,g(m,n)}$ holds for any~$k\in\mathbb N$. In the step, we first note that $n<\tau(m)$ entails
\begin{equation*}
\max(m,n)\leq\tau(m)\leq g(m,n)\leq f(g(m,n))=g(m,n+1)\leq h(m).
\end{equation*}
Given~$v_j\notin\fv(\Delta')$, we thus have $\mathbb N,\eta_j^n\nvDash\rho^{g(m,n),g(m,n+1)}$ for~$\rho\in\Delta'$, so that we~get
\begin{equation*}
\mathbb N,\eta_j^n\vDash(\neg\varphi[v_i/v_j])^{g(m,n),g(m,n+1)}\lor\varphi[v_i/S v_j]^{k,g(m,n+1)},
\end{equation*}
where~$k$ can again be arbitrary. As in the proof of Proposition~\ref{prop:soundness-FO}, one can~use Exercise~\ref{ex:substitution} to obtain $\mathbb N,\eta_i^n\nvDash\varphi^{g(m,n+1),g(m,n)}$ or $\mathbb N,\eta_i^{n+1}\vDash\varphi^{k,g(m,n+1)}$, as needed to complete the induction step. In view of $t^{\mathbb N,\eta}\leq\tau(m)$ and $g(m,n)\leq h(m)$, the result of the induction entails~$\mathbb N,\eta\vDash\varphi[v_i/t]^{m,h(m)}$, so that we indeed get~$\vDash^h\Gamma$.
\end{proof}

We now obtain the main result of this section.

\begin{corollary}[\cite{mints-PR,parsons70,takeuti-proof-theory-ed1}]\label{cor:ISigma1_prim-rec}
If a function is $\mathsf{I\Sigma}_1$-provably total, then it must be primitive recursive.
\end{corollary}
\begin{proof}
It suffices to combine Theorems~\ref{thm:Sigma_1-prim-rec} and~\ref{thm:compl-Sigma_1} with the observation that was made in the paragraph before Exercise~\ref{ex:bounded-prim-rec}.
\end{proof}

\subsection{Non-Hereditary Goodstein Sequences}\label{sect:goodstein}

As a concrete example for G\"odel's first incompleteness theorem, we present a concrete mathematical result that is unprovable in the theory~$\mathsf{I\Sigma}_1$. For natural numbers $b\geq 2$ and~$n>0$, we write
\begin{equation*}
n\nf{b}b^{e_0}\cdot c_0+\ldots+b^{e_k}\cdot c_k
\end{equation*}
if we have $e_1>\ldots>e_k$ and $0<c_i<b$ for all~$i\leq k$. Let us recall that any~$n>0$ has a unique base-$b$ normal form of this format.

\begin{definition}\label{def:Goodstein}
By the Goodstein sequence with start value~$n\in\mathbb N$, we mean the function~$G_n:\mathbb N\to\mathbb N$ that is determined by $G_n(0)=n$ and the recursive clause
\begin{equation*}
G_n(i+1)=\left\{\begin{aligned}
&(i+3)^{e_0}\cdot c_0+\ldots+(i+3)^{e_k}\cdot c_k-1\\
&\phantom{0}\qquad\text{if $0<G_n(i)\nf{(i+2)}(i+2)^{e_0}\cdot c_0+\ldots+(i+2)^{e_k}\cdot c_k$},\\[1ex]
&0\qquad\text{if $G_n(i)=0$}.
\end{aligned}\right.
\end{equation*}
We say that the Goodstein sequence terminates if~$G_n(i)=0$ holds for some~$i\in\mathbb N$.
\end{definition}

For example, the Goodstein sequence with start value~$42$ begins as follows:
\begin{alignat*}{3}
G_{42}(0)&\nf{2}2^5+2^3+2^1&&=42,\\
G_{42}(1)&=3^5+3^3+3^1-1\nf{3}3^5+3^3+3^0\cdot 2&&=272,\\
G_{42}(2)&=4^5+4^3+4^0\cdot 2-1\nf{4}4^5+4^3+4^0&&=1089,\\
G_{42}(3)&=5^5+5^3+5^0-1\nf{5}5^5+5^3&&=3250,\\
G_{42}(4)&=6^5+6^3-1\nf{6}=6^5+6^2\cdot 5+6^1\cdot 5+6^0\cdot 5&&=7991,\\
G_{42}(5)&=7^5+7^2\cdot 5+7^1\cdot 5+7^0\cdot 5-1&&=17091.
\end{alignat*}
The sequence of numerical values may suggest that~$G_{42}$ grows indefinitely. However, we will see that the Goodstein sequences terminate for all start values. Furthermore, we will show that this result is unprovable in~$\mathsf{I\Sigma}_1$ (for any suitable formalization).

We note that our definition deviates from a more common version of Goodstein sequences, which works with a hereditary normal form that requires the exponents to be written in normal form as well. Here the start value~$42$ would be written as
\begin{equation*}
2^{2^{2^1}+2^0}+2^{2 ^1+2^0}+2^1,
\end{equation*}
so that the next value in this version of the Goodstein sequences is
\begin{equation*}
3^{3^3+1}+3^{3+1}+3-1=3^{26}+3^4+2>10^{12}.
\end{equation*}
For this version, it is still true that all Goodstein sequences terminate, but that~fact is unprovable in full Peano arithmetic (see~\cite{kirby-paris82,rathjen-goodstein}). The independence result that we will prove for non-hereditary Goodstein sequences is mentioned in~\cite{cichon83}, but no detailed proof is given there. We will use the following to show that all Goodstein sequences in the sense of Definition~\ref{def:Goodstein} terminate.

\begin{definition}\label{def:omega^omega}
Let $<_2$ be the linear order on~$\mathbb N^2$ in which $(m,n)\leq(m',n')$ holds when we have either~$m=m'$ and $n<n'$ or~$m'<m$. We write $\omega^\omega$ for the set of all expressions $\omega^{e_0}\cdot c_0+\ldots+\omega^{e_{k-1}}\cdot c_{k-1}$ with $e_i,c_i\in\mathbb N$ and $e_0>\ldots>e_{k-1}$ (including the empty expression~$0\in\omega^\omega$ for~$k=0$). To turn~$\omega^\omega$ into a linear order, we declare that
\begin{equation*}
\omega^{e_0}\cdot c_0+\ldots+\omega^{e_{k-1}}\cdot c_{k-1}\prec\omega^{f_0}\cdot d_0+\ldots+\omega^{f_{l-1}}\cdot d_{l-1}
\end{equation*}
holds when we have $k<l$ and $(e_i,c_i)=(f_i,d_i)$ for all~$i<k$ and also when there is~a $j<\min(k,l)$ with $(e_j,c_j)<(f_j,d_j)$ and $(e_i,c_i)=(f_i,d_i)$ for all~$i<j$.
\end{definition}

One may view elements of~$\omega^\omega$ as base-$\omega$ normal forms for an infinite number~$\omega$. To make this precise, one defines an ordinal number as an isomorphism class of well orders (cf.~the following exercise). Each number~$n\in\mathbb N$ corresponds to the~ordinal that is represented by the finite order~$\{0,\ldots,n-1\}$. We take~$\omega$ to be the isomorphism class of~$\mathbb N$ with the usual order, since we wish to validate the normal form condition~$c_i<\omega$ precisely~for~$c_i\in\mathbb N$. It will be shown that~$\omega^\omega$ is also a well order, which represents an ordinal number that is strictly bigger than~$\omega$ (as the latter is isomorphic to a strict initial segment of~$\omega^\omega$). Interestingly, the termination of Good\-stein sequences is a result about~$\mathbb N$ that is naturally proved with the help~of infinite ordinal numbers. The aforementioned unprovability results show that infinite ordinals are indeed indispensable in a certain sense (see~\cite{afrw-goodstein} for a precise result).

\begin{exercise}
Show that the following characterizations of well orders are equivalent whenever~$(X,<)$ is a linear order:
\begin{enumerate}[label=(\roman*)]
\item There is no $f:\mathbb N\to X$ such that $f(i+1)<f(i)$ holds for all~$i\in\mathbb N$.
\item Each non-empty subset~$Z\subseteq X$ has an element~$z\in Z$ that is minimal in the sense that $x<z$ implies~$x\notin Z$.
\item For any subset~$P\subseteq X$ such that we get $z\in P$ whenever $x\in P$ is given for all~$x<z$, we have $P=X$.
\end{enumerate}
Note that statement~(iii) amounts to an induction principle if one thinks of~$P$ as the collection of all elements with some desired property. 
\end{exercise}

The following will connect~$\omega^\omega$ to our Goodstein sequences.

\begin{definition}
To define~$\Omega_b:\mathbb N\to\mathbb N$ for~$b\geq 2$, we stipulate
\begin{equation*}
\Omega_b(n)=\omega^{e_0}\cdot c_0+\ldots+\omega^k\cdot c_k\quad\text{for}\quad 0<n\nf{b}b^{e_0}\cdot c_0+\ldots+b^k\cdot c_k
\end{equation*}
and declare that $\Omega_b(0)\in\omega^\omega$ is the empty expression.
\end{definition}

By a familiar property of base-$b$ normal forms, we have
\begin{equation*}
m<n\quad\Leftrightarrow\quad\Omega_b(m)\prec\Omega_b(n)
\end{equation*}
for any~$b\geq 2$ and all~$m,n\in\mathbb N$.

\begin{lemma}\label{lem:Goodstein-ord-decent}
Given $G_n(i)\neq 0$, we get $\Omega_{i+3}G_n(i+1)\prec\Omega_{i+2}G_n(i)$.
\end{lemma}
\begin{proof}
For $G_n(i)\nf{(i+2)}(i+2)^{e_0}\cdot c_0+\ldots+(i+2)^{e_k}\cdot c_k$, the definition of Goodstein sequences yields
\begin{equation*}
G_n(i+1)+1\nf{(i+3)}(i+3)^{e_0}\cdot c_0+\ldots+(i+3)^{e_k}\cdot c_k.
\end{equation*}
We get $\Omega_{i+3}(G_n(i+1)+1)=\Omega_{i+2}(G_n(i))$ and conclude as~$\Omega_{i+3}$ is monotone.
\end{proof}

The previous lemma will be combined with the following result.

\begin{exercise}\label{ex:omega-omega}
(a) Prove that $\omega^\omega$ is a well order. \emph{Hint:} Use induction on~$n$ to show that an~$i\in\mathbb N$ with $f(i)\preceq f(i+1)$ exists for any~$f:\mathbb N\to\omega^\omega$ such that~$f(0)$ is of the form $\omega^{e_0}\cdot c_0+\ldots+\omega^k\cdot c_k$ with $e_0=n$.

(b) Show that Peano arithmetic proves the statement ``any number $m\in\mathbb N$~admits an $i\in\mathbb N$ with $f(m,i)\preceq f(m,i)$" for each externally given~$f:\mathbb N^2\to\omega^\omega$ with a definition in first-order arithmetic. \emph{Hint:} The induction from the hint in~(a) cannot be formalized literally, as it involves a quantification over arbitrary functions. To obtain an induction statement that can be expressed in first-order arithmetic, one observes that all relevant functions are straightforward modifications of the given~$f$. 
\end{exercise}

We now derive a result that was promised above.

\begin{proposition}\label{prop:Goodstein-terminates}
The Goodstein sequence for any start value terminates.
\end{proposition}
\begin{proof}
Towards a contradiction, we assume that we have an~$m\in\mathbb N$ with $G_m(i)\neq 0$ for all~$i\in\mathbb N$. By Lemma~\ref{lem:Goodstein-ord-decent}, the function $f:\mathbb N\to\omega^\omega$ with $f(i)=\Omega_{i+2}G_n(i)$ will then witness that~$\omega^\omega$ is no well order, against the previous exercise.
\end{proof}

The proposition ensures that the following is well-defined.

\begin{definition}\label{def:Ackermann}
In order to define~$A:\mathbb N\to\mathbb N$, we stipulate that $A(m)$ is the smallest number~$i\in\mathbb N$ with $G_m(i)=0$.
\end{definition}

It is not hard to see that the function~$(m,i)\mapsto G_m(i)$ is primitive recursive, which entails that~$A$ is $\Sigma$-definable. In the rest of this section, we show that~$A$ grows faster than any primitive recursive function. This means that~$A$, which is a variant of the so-called Ackermann function, cannot be primitive recursive itself. By the analysis of provably total functions that was given in the previous section, it will follow that Proposition~\ref{prop:Goodstein-terminates} cannot be proved in the theory~$\mathsf{I\Sigma}_1$. On the other hand, part~(b) of Exercise~\ref{prop:Goodstein-terminates} entails that Proposition~\ref{prop:Goodstein-terminates} can be formalized in Peano arithmetic. A careful solution of the exercise shows that it is enough to have induction for~$\Pi_2$-formulas (see the paragraph after Definition~\ref{def:Sigma-formula}).

\begin{definition}\label{def:mesh}
For $\alpha,\beta\in\omega^\omega$ we say that $\alpha$ and $\beta$ mesh if they are of the form
\begin{equation*}
\alpha=\omega^{e_0}\cdot c_0+\ldots+\omega^{e_k}\cdot c_k\quad\text{and}\quad\beta=\omega^{f_0}\cdot d_0+\ldots+\omega^{f_l}\cdot d_l
\end{equation*}
with $e_k\geq e_0$. In this case we set
\begin{equation*}
\alpha+\beta=\omega^{e_0}\cdot c_0+\ldots+\omega^{e_k}\cdot c_k+\omega^{f_0}\cdot d_0+\ldots+\omega^{f_l}\cdot d_l,
\end{equation*}
where $\omega^{e_k}\cdot c_k+\omega^{f_0}\cdot d_0$ is replaced by~$\omega^{e_k}\cdot(c_k+d_0)$ when we have~$e_k=f_0$. We also say that $\alpha$ and $\beta$ mesh when at least one of them is the empty expression~$0\in\omega^\omega$, and we declare $\alpha+0=\alpha$ and $0+\beta=\beta$.
\end{definition}

There are reasonable ways to extend addition to summands that do not mesh. In the present lecture notes, we implicitly assume that $\alpha$ and $\beta$ mesh whenever we write~$\alpha+\beta$. Let $\omega^e$ abbreviate~$\omega^e\cdot 1$. Each element~$\alpha\neq 0$ of~$\omega^\omega$ can be uniquely written as $\alpha=\beta+\omega^e$. Elements of the form~$\beta+\omega^0$ and $\beta+\omega^{e+1}$ are called successors and limits, respectively. We will also write~$1$ at the place of~$\omega^0$.

\begin{definition}
For each~$\alpha\in\omega^\omega$, we define a so-called fundamental sequence
\begin{equation*}
\mathbb N\ni n\mapsto\{\alpha\}(n)\in\omega^\omega
\end{equation*}
by setting $\{0\}(n)=0$ and $\{\beta+1\}(n)=\beta$ as well as $\{\beta+\omega^{e+1}\}(n)=\beta+\omega^e\cdot n$.
\end{definition}

The characteristic property of fundamental sequences is that they approximate limits in the following sense.

\begin{exercise}
If~$\alpha\in\omega^\omega$ is a limit, then $m<n$ entails $\{\alpha\}(m)\prec\{\alpha\}(n)\prec\alpha$ and each~$\gamma\prec\alpha$ admits an~$n\in\mathbb N$ with $\gamma\prec\{\alpha\}(n)$.
\end{exercise}

To connect with Goodstein sequences, we provide a variant of fundamental sequences that skips limits. Given a function~$f:X\to X$ on some set~$X$, we write~$f^i$ for the iterates that are defined by $f^0(x)=x$ and $f^{i+1}(x)=f(f^i(x))$. A straightforward induction shows $f^{i+j}=f^j\circ f^i$. Let us note that the following is well-defined since we have $\{\alpha\}(n)\prec\alpha$ for any~$\alpha\neq 0$.

\begin{definition}\label{def:modified-fund-seq}
Given $\alpha\in\omega^\omega$, we set $[\alpha](n)=\{\alpha\}^{i+1}(n)$ for the smallest~$i\in\mathbb N$ such that~$\{\alpha\}^i(n)$ is no limit.
\end{definition}

We can now state a more precise version of Lemma~\ref{lem:Goodstein-ord-decent}.

\begin{exercise}\label{ex:goodstein-fund-seq}
Prove $[\Omega_{i+2}(G_n(i))](i+3)=\Omega_{i+3}(G_n(i+1))$ for~$i,n\in\mathbb N$.
\end{exercise}

The following will allow us to iterate the result of the exercise. Let us note that the function~$H_\alpha$ is well-defined since we have~$[\gamma](m)\prec\gamma$ for~$\gamma\neq 0$.

\begin{definition}\label{def:Hardy}
For $\alpha\in\omega^\omega$ and $i,n\in\mathbb N$, we explain~$[\alpha]_i(n)$ recursively by
\begin{equation*}
[\alpha]_0(n)=\alpha\quad\text{and}\quad [\alpha]_{i+1}(n)=\big[[\alpha]_i(n)\big](n+i).
\end{equation*}
We then put $H_\alpha(n)=n+i$ for the smallest~$i\in\mathbb N$ with $[\alpha]_i(n)=0$.
\end{definition}

The functions $H_\alpha$ form the so-called Hardy hierarchy~\cite{wainer-hardy}. The latter is characterized by the recursive clauses in part~(a) of the following.

\begin{exercise}\label{ex:Hardy}
(a) Show that we have $H_0(n)=n$ and $H_{\alpha+1}(n)=H_\alpha(n+1)$ as well as $H_\lambda(n)=H_{\{\lambda\}(n)}(n)$ for limit $\lambda$. \emph{Hint:} For the successor case, use induction on~$i\in\mathbb N$ to prove $[\alpha+1]_{i+1}(n)=[\alpha]_i(n+1)$. In the limit case, observe that we have $[\lambda](n)=[\{\lambda\}(n)](n)$ according to Definition~\ref{def:modified-fund-seq}.

(b) Prove that we have $H_{\alpha+\beta}=H_\alpha\circ H_\beta$ when $\alpha$ and $\beta$ mesh. \emph{Hint:} Use~induction on~$\beta$ and employ the result of part~(a).
\end{exercise}

Let us come to the crucial connection with Goodstein sequences.

\begin{proposition}\label{prop:Goodstein-Hardy}
We have $H_{\Omega_2(n)}(3)=n+A(n)$ for any~$n\in\mathbb N$.
\end{proposition}
\begin{proof}
A straighforward induction based on Exercise~\ref{ex:goodstein-fund-seq} yields
\begin{equation*}
\Omega_{i+2}(G_n(i))=[\Omega_2(n)]_i(3).
\end{equation*}
Since~$G_n(i)=0$ in $\mathbb N$ is equivalent to~$\Omega_{i+2}(G_n(i))=0$ in $\omega^\omega$, the claim is now immediate by Definitions~\ref{def:Ackermann} and~\ref{def:Hardy}.
\end{proof}

To connect the Hardy hierarchy with primitive recursion, we introduce another hierarchy of functions. As before, we shall write $F_n^k$ for the~$k$-th iterate of~$F_n$.

\begin{definition}
In order to define functions~$F_n:\mathbb N\to\mathbb N$ for~$n\in\mathbb N$, we recursively stipulate~$F_0(b)=b+1$ and $F_{n+1}(b)=F_n^b(b)$.
\end{definition}

The functions~$F_n$ are part of the so-called fast-growing hierarchy~\cite{Grzegorczyk,schwichtenberg71,wainer70}. This name is justified by part~(a) of the following exercise and by Proposition~\ref{prop:pr-fast-growing} below.

\begin{exercise}\label{ex:fast-growing}
(a) Prove $F_2(n)=2^n\cdot n$ and~$F_3(3)>10^N$ for $N=10^6$.

(b) Show that we have $b<F_n(b)\leq F_{n+1}(b)$ for $b>0$ and even $F_n(b)<F_{n+1}(b)$ for~$b>1$. Also show that $b<c$ entails $F_n(b)<F_n(c)$.
\end{exercise}

We now connect the fast-growing and the Hardy hierarchy.

\begin{proposition}\label{prop:Hardy-fast-growing}
The functions $F_n$ and $H_{\omega^n}$ coincide for any~$n\in\mathbb N$.
\end{proposition}
\begin{proof}
We argue by induction on~$n$. In the base case, we compute
\begin{equation*}
H_1(n)=H_0(n+1)=n+1=F_0(n).
\end{equation*}
For the induction step, we use an auxiliary induction on~$i$ to show that $F_n^i$ coincides with $H_{\omega^n\cdot i}$. The step of the auxiliary induction is secured by Exercise~\ref{ex:Hardy}(b), which allows us to compute
\begin{equation*}
H_{\omega^n\cdot(i+1)}=H_{\omega^n}\circ H_{\omega^n\cdot i}=F_n\circ F_n^i=F_n^{i+1}.
\end{equation*}
Once the auxiliary induction is complete, we obtain
\begin{equation*}
H_{\omega^{n+1}}(k)=H_{\{\omega^{n+1}\}(k)}(k)=H_{\omega^n\cdot k}(k)=F_n^k(k)=F_{n+1}(k),
\end{equation*}
as needed to complete the step of the main induction.
\end{proof}

Let us note that~$F_n$ is primitive recursive for each~$n\in\mathbb N$. In view of the following result, the primitive recursive functions are exhausted by the given segment of the fast-growing hierarchy.

\begin{proposition}\label{prop:pr-fast-growing}
For each primitive recursive $f:\mathbb N^k\to\mathbb N$, there is an $n\in\mathbb N$ such that $f(a_0,\ldots,a_{k-1})\leq F_n(b)$ holds for any $b\geq 2$ with $a_i\leq b$ for all~$i<k$.
\end{proposition}
\begin{proof}
We argue by induction over the build-up of primitive recursive functions according to Definition~\ref{def:prim-rec}. In the case of a constant function with value~$n\in\mathbb N$, it suffices to note that Exercise~\ref{ex:fast-growing} yields $n\leq F_n(b)$ for any~$b\geq 2$. For projections and the successor function, the claim of the proposition holds with~$n=0$. Now assume that $f$ is a composition given by
\begin{equation*}
f(a_1,\ldots,a_k)=h\left(g_1(a_1,\ldots,a_k),\ldots,g_m(a_1,\ldots,a_k)\right).
\end{equation*}
Since $n\mapsto F_n(b)$ is increasing for~$b>0$, we may inductively assume that there is a single~$n\in\mathbb N$ such that we have $g_j(a_1,\ldots,a_k)\leq F_n(b)$ for all $b\geq 2$ with $a_i\leq b$ as well as $h(c_1,\ldots,c_m)\leq F_n(d)$ for all $d\geq 2$ with $c_j\leq d$. If we employ the latter with $d=F_n(b)>b$, we see that any $a_i\leq b$ with $b\geq 2$ validate
\begin{equation*}
f(a_1,\ldots,a_k)\leq F_n(F_n(b))\leq F_n^b(b)=F_{n+1}(b).
\end{equation*}
So the claim for~$f$ holds with $n+1$ at the place of~$n$. Finally, we consider the case where $f$ has been obtained by recursive clauses
\begin{align*}
f(a_1,\ldots,a_k,0)&=g(a_1,\ldots,a_k),\\
f(a_1,\ldots,a_k,c+1)&=h(a_1,\ldots,a_k,c,f(a_1,\ldots,a_k,c)).
\end{align*}
Again, we can invoke the induction hypothesis to obtain a single~$n\in\mathbb N$ such that we have $g(a_1,\ldots,a_k)\leq F_n(b)$ for all $b\geq 2$ with $a_i\leq b$ and $h(c_1,\ldots,c_{k+2})\leq F_n(d)$ for all $d\geq 2$ with $c_j\leq d$. Given $b\geq 2$ with $a_i\leq b$, a straightforward induction on~$c\leq b$ yields $f(a_1,\ldots,a_k,c)\leq F_n^{c+1}(b)$. Still for $c\leq b$, we thus get
\begin{equation*}
f(a_1,\ldots,a_k,c)\leq F_n^{b+1}(b)=F_n(F_{n+1}(b))\leq F_{n+1}^b(b)=F_{n+2}(b),
\end{equation*}
which completes the main induction step with $n+2$ at the place of~$n$.
\end{proof}

We are now able to confirm that our version of the Ackermann function~\cite{ackermann28} grows extremely fast.

\begin{corollary}
For each primitive recursive~$f:\mathbb N\to\mathbb N$, there is an $N\in\mathbb N$ such that we have $f(n)<A(n)$ for all~$n\geq N$.
\end{corollary}
\begin{proof}
By the proposition above, we get a~$k\geq 2$ such that $n+f(n)\leq F_k(n)$ holds for all~$n\geq 2$. We set $N=2^{k+1}$ and consider an arbitrary~$n\geq N$. Let $l\geq 3$ be determined by $2^{l+1}\leq n<2^{l+2}$. We can then write $\Omega_2(n)=\omega^{l+1}+\beta$ such that $\omega^{l+1}$ and $\beta$ mesh. An easy induction on~$\beta$ yields $i\leq H_\beta(i)$, so that we can invoke Propositions~\ref{prop:Goodstein-Hardy} and~\ref{prop:Hardy-fast-growing} as well as Exercises~\ref{ex:Hardy} and~\ref{ex:fast-growing} to get
\begin{equation*}
n+A(n)=H_{\Omega_2(n)}(3)=H_{\omega^{l+1}}(H_\beta(3))=F_{l+1}(H_\beta(3))\geq F_{l+1}(3)\geq F_k(F_l^2(3)).
\end{equation*}
Observe that $l\geq k\geq 2$ and $F_l(3)>l+2$ entail $F_l^2(3)>F_2(l+2)>2^{l+2}>n$, again by Exercise~\ref{ex:fast-growing}. We get $n+A(n)>F_k(n)\geq n+f(n)$ and hence $A(n)>f(n)$.
\end{proof}

The work of the previous and present section culminates in the following result. Let us note that the definition of Goodstein sequences is naturally represented by a $\Sigma$-formula, by the proof of Theorem~\ref{thm:prim-rec-sigma} and the remark after Definition~\ref{def:Ackermann}. We also point out that the termination of non-hereditary Goodstein sequences is provable in Peano arithmetic, which confirms that the latter is strictly stronger than the theory~$\mathsf{I\Sigma}_1$ (cf.~the paragraph before Definition~\ref{def:mesh}).

\begin{theorem}
Consider any $\Sigma$-formula~$\varphi(x,y,z)$ that defines the Goodstein sequences from Definition~\ref{def:Goodstein}, in the sense that all $i,k,n\in\mathbb N$ validate
\begin{equation*}
G_n(i)=k\quad\Leftrightarrow\quad\mathbb N\vDash\varphi(n,i,k).
\end{equation*}
We then have $\mathsf{I\Sigma}_1\nvdash\forall x\exists y\,\varphi(x,y,0)$, i.\,e., the termination of non-hereditary Goodstein sequences is unprovable in~$\mathsf{I\Sigma}_1$.
\end{theorem}
\begin{proof}
Aiming at a contradiction, we assume that we have $\mathsf{I\Sigma}_1\vdash\forall x\exists y\,\varphi(x,y,0)$ for $\varphi$ as indicated. By the proof of Corollary~\ref{cor:ISigma1_prim-rec}, we obtain a primitive recursive function~$f:\mathbb N\to\mathbb N$ such that each~$n$ admits an~$i\leq f(n)$ with~$\mathbb N\vDash\varphi(n,i,0)$ and hence~$G_n(i)=0$. So~$A:\mathbb N\to\mathbb N$ as in Definition~\ref{def:Ackermann} is dominated by~$f$. But this contradicts the previous corollary.
\end{proof}

\section{Computability}\label{sect:computability}

This section is a brief introduction to computability theory, which is also known as recursion theory. The theory of computation is fascinating both as a modern technical subject and for its connections to the development of the computer, in which logicians such as Alan Turing played a crucial role.

In Subsection~\ref{subsect:Kleene-Turing} below, we prove an equivalence between three possible definitions of computable functions. This provides support for the widely accepted Church-Turing thesis, which asserts that there is a single canonical notion of computability. We also consider computations with oracle and prove fundamental results such as Kleene's normal form theorem and the uncomputability of the halting problem. In Subsection~\ref{subsect:KL}, we use computability theory to analyze the complexity of K\H{o}nig's lemma. As we will see, the latter entails the existence of uncomputable sets. We also prove the low basis theorem, which shows that the weak version of K\H{o}nig's lemma for binary trees (see Theorem~\ref{thm:wkl}) is considerably simpler than the full version for trees with arbitrary finite branchings.

A more comprehensive introduction to computability theory can be found, e.\,g., in the textbook by Robert Soare~\cite{soare-computability}. Let us also point out that there is a close connection between computability and provability. The latter is explored, e.\,g., in the research program of reverse mathematics, which aims to determine the minimal axioms that are needed to prove a given mathematical theorem (see the founding article by Harvey Friedman~\cite{friedman-rm} and the textbook by Stephen Simpson~\cite{simpson09}).

\subsection{Kleene's Normal Form and the Halting Problem}\label{subsect:Kleene-Turing}

In the present section, we show that several definitions of computable function are equivalent. This suggests that there is a single and robust notion of computability. We study fundamental properties of this notion.

It is instructive to approach computability via a machine model, i.\,e., by a mathematical description of program evaluation on a computer. The Turing~machine is probably the most famous model in this context. We work with a particularly succinct model known as the register machine, which leads to an equivalent notion of computability (while comparisons of computational complexity can be subtle).

Let us begin with a somewhat informal description of register machines, which will be made formal later. By a program, we mean a finite sequence of instructions that have one of the forms
\begin{equation*}
r_i\leftarrow 0,\qquad r_i\leftarrow r_i+1,\qquad\texttt{if }r_i=r_j\texttt{ then }I_m\texttt{ else }I_n.
\end{equation*}
Any program refers to a finite number of so-called registers~$r_0,\ldots r_n$, each of which can store a natural number. The first two of the given instructions have the effect that the number in $r_i$ is replaced by $0$ or increased by one, respectively. Instructions are usually executed in order, except when we encounter an instruction of the third form, which we call a jump instruction. Here one continues with the $m$-th instruction of the program when the numbers in $r_i$ and $r_j$ are equal and with the $n$-th instruction otherwise. The evaluation of a program terminates when an instruction beyond the length of the program is called, either after the last instruction in the program or due to a jump instruction.

To perform a computation on given arguments, one places these in some registers before the program is executed. It is assumed that all other registers are initially set to zero. Which registers are used is inessential, as the following program allows us to copy the content of $r_i$ into~$r_j$, assuming that $i$ and~$j$ are different. We write each instruction in an individual line and label the $n$-th instruction by~$I_n$.
\begin{align*}
I_0\qquad& r_j\leftarrow 0\\
I_1\qquad& \texttt{if }r_i=r_j\texttt{ then }I_4\texttt{ else }I_2\\
I_2\qquad& r_j\leftarrow r_j+1\\
I_3\qquad& \texttt{if }r_i=r_j\texttt{ then }I_4\texttt{ else }I_2
\end{align*}
While the given program terminates on any input, there are programs that do not. An example of a program that never terminates is provided by
\begin{align*}
I_0\qquad& \texttt{if }r_0=r_0\texttt{ then }I_0\texttt{ else }I_1.
\end{align*}
We will later see that non-termination plays an indispensable role in the theory of computation (see Exercise~\ref{ex:computable-partial} and the paragraph that precedes it).

The following denotes a program which decreases the content of~$r_i$ by one unless it is already zero. We note that the line with label $I_2\text{-}I_5$ is no single instruction but refers to a subprogram that copies the content of~$r_i$ into~$r_{i+2}$. This subprogram coincides with the above program for $j=i+2$, except that each expression~$I_k$ is replaced by~$I_{2+k}$. In particular, the subprogram terminates by calling~$I_6$, which means that the main program continues with that instruction.
\begin{align*}
I_0\qquad& r_{i+1}\leftarrow 0\\
I_1\qquad& \texttt{if }r_i=r_{i+1}\texttt{ then }I_{12}\texttt{ else }I_2\\
I_2\text{-}I_5\qquad& r_{i+2}\leftarrow r_i\\
I_6\qquad& r_0\leftarrow 0\\
I_7\qquad& r_1\leftarrow r_1+1\\
I_8\qquad& \texttt{if }r_1=r_2\texttt{ then }I_{12}\texttt{ else }I_9\\
I_9\qquad& r_0\leftarrow r_0+1\\
I_{10}\qquad& r_2\leftarrow r_2+1\\*
I_{11}\qquad & \texttt{if }r_1=r_2\texttt{ then }I_{12}\texttt{ else }I_9
\end{align*}
Whenever a given program is used as a subprogram that spans the lines~$I_m,\ldots,I_n$ of a new program, we assume that each jump to~$I_k$ in the originally given program is tacitly replaced by a jump to~$I_{\min(m+k,n+1)}$ in the new program. When the mini\-mum is equal to~$m+k\leq n$, the jump goes to the intended line of the subprogram. The minimum is taken in order to ensure that a jump beyond the subprogram goes to the next instruction of the new program.

We will call a set $Y\subseteq\mathbb N$ computable when there is a program that computes its characteristic function~$\chi_Y:\mathbb N\to\{0,1\}$, in the sense that the program terminates with $\chi_Y(n)$ in~$r_0$ when it is evaluated starting with~$n$ in~$r_0$. It will also be of interest (cf.~Section~\ref{subsect:KL}) whether some $Y\subseteq\mathbb N$ is computable relative to another set~$Z\subseteq\mathbb N$. This can be made precise via register machines with a so-called oracle. In addition to the three types of instruction from above, these involve instructions of the form
\begin{equation*}
r_i\leftarrow\chi(r_i).
\end{equation*}
Programs are now executed relative to some set~$Z\subseteq\mathbb N$. When the given new instruction occurs while~$r_i$ contains~$n$, we replace the latter by~$\chi_Z(n)$. Computability without oracle can be recovered as the special case where we have~$Z=\emptyset$. It is common to work with oracles for sets rather than functions. The following shows that this is no real restriction (at least for oracles that never fail to respond). 

\begin{exercise}
Consider a function~$f:\mathbb N\to\mathbb N$ and its coded graph
\begin{equation*}
\mathcal G(f)=\{\pi(m,n)\,|\,f(m)=n\}\subseteq\mathbb N,
\end{equation*}
where $\pi(m,n)$ refers to the Cantor pairing function (see Exercise~\ref{ex:Cantor-pairing}). Find a program that computes~$f$ when it is executed relative to~$\mathcal G(f)$. \emph{Remark:} You may assume that a program to compute~$\pi$ is given (cf.~Theorem~\ref{thm:comp-equiv}).
\end{exercise}

We continue with a formal definition of program evaluation. The aim is not just to make things precise but also to show that evaluation is a primitive recursive process. For this reason, we define programs as sequences of natural numbers, which are coded according to Definition~\ref{def:seq-code}. The paragraph before Definition~\ref{def:binary-tree} explains basic notation for sequences.

\begin{definition}\label{def:RM-evaluation}
A program is a coded sequence $e=\langle e_0,\ldots,e_{l(e)-1}\rangle$ where each $e_k$ for~$k<l(e)$ codes a sequence of one of the following forms (with $i,j,m,n\in\mathbb N$):
\begin{equation*}
\langle 0,i\rangle,\qquad\langle 1,i\rangle,\qquad\langle 2,i,j,m,n\rangle,\qquad\langle 3,i\rangle.
\end{equation*}
Consider a program $e$, a sequence $\sigma\in\mathbb N^{<\omega}$ (`input') and a set $Z\subseteq\mathbb N$ (`oracle'). We employ recursion on~$k\in\mathbb N$ to define a number~$I^Z_{e,\sigma,k}\in\mathbb N$ (`current instruction') and, simultaneously, a function~$R^Z_{e,\sigma,k}:\mathbb N\to\mathbb N$ (`current register content'). Set
\begin{equation*}
I^Z_{e,\sigma,0}=0\qquad\text{and}\qquad R^Z_{e,\sigma,0}(p)=\begin{cases}
\sigma_p & \text{for }p<l(\sigma),\\
0 & \text{otherwise}.
\end{cases}
\end{equation*}
In the recursion step, we first declare that $I^Z_{e,\sigma,k}\geq l(e)$ entails $I^Z_{e,\sigma,k+1}=I^Z_{e,\sigma,k}$ as well as $R^Z_{e,\sigma,k+1}=R^Z_{e,\sigma,k}$. Now assume we have $I:=I^Z_{e,\sigma,k}<l(e)$. We then put
\begin{align*}
I^Z_{e,\sigma,k+1}&=\begin{cases}
I^Z_{e,\sigma,k}+1 & \text{if $e_I$ has the form $\langle 0,i\rangle$, $\langle 1,i\rangle$ or $\langle 3,i\rangle$},\\
m & \text{if $e_I=\langle 2,i,j,m,n\rangle$ and $R^Z_{e,\sigma,k}(i)=R^Z_{e,\sigma,k}(j)$},\\
n & \text{if $e_I=\langle 2,i,j,m,n\rangle$ and $R^Z_{e,\sigma,k}(i)\neq R^Z_{e,\sigma,k}(j)$},
\end{cases}\\
R^Z_{e,\sigma,k+1}(p)&=\begin{cases}
0 & \text{if $e_I=\langle 0,p\rangle$},\\
R^Z_{e,\sigma,k}(p)+1 & \text{if $e_I=\langle 1,p\rangle$},\\
\chi_Z\big(R^Z_{e,\sigma,k}(p)\big) & \text{if $e_I=\langle 3,p\rangle$},\\
R^Z_{e,\sigma,k}(p) & \text{in all other cases}.
\end{cases}
\end{align*}
For each program~$e$ and each set~$Z\subseteq\mathbb N$, we define
\begin{equation*}
K^Z_e=\left\{\sigma\in\mathbb N^{<\omega}\,|\,\text{there is a $k\in\mathbb N$ with $I^Z_{e,\sigma,k}\geq l(e)$}\right\}.
\end{equation*}
The function $\{e\}^Z:K^Z_e\to\mathbb N$ is determined by
\begin{equation*}
\{e\}^Z(\sigma)=R^Z_{e,\sigma,k}(0)\quad\text{for the smallest $k\in\mathbb N$ with $I^Z_{e,\sigma,k}\geq l(e)$}.
\end{equation*}
For an $e\in\mathbb N$ that is no program, we let $K^Z_e$ and $\{e\}^Z$ be the empty set and function. When we have $Z=\emptyset$, we write $K_e$ and $\{e\}$ at the place of~$K_e^Z$ and $\{e\}^Z$.
\end{definition}

In the context of computability, one writes $f:X\rightharpoonup Y$ (note the arrow head) and speaks of a partial~function in order to refer to a function $f:D\to Y$ that is defined on some set~$D\subseteq X$. When the latter is an equality, we say that~$f$ is total. To assert that we have $\mathbf a\in D$, one writes~$f(\mathbf a)\!\downarrow$ and says that $f(\mathbf a)$ is defined. In particular, we have $\{e\}^Z(\sigma)\!\downarrow$ precisely for~$\sigma\in K^Z_e$. For partial functions $f$ and~$g$, one writes $f(\mathbf a)\simeq g(\mathbf b)$ to express that $f(\mathbf a)\!\downarrow$ is equivalent to $g(\mathbf b)\!\downarrow$ and that we have $f(\mathbf a)=g(\mathbf b)$ if both values are defined. In particular, $f(\mathbf a)\simeq c$ denotes that~$f(\mathbf a)$ is defined with value~$c\in\mathbb N$.

The operator~$\mu$ of unbounded minimization or unbounded search (see Proposition~\ref{prop:bounded-min} for the bounded case) transforms a partial function $f:\mathbb N^{l+1}\rightharpoonup\mathbb N$ into the partial function $\mu f:\mathbb N^l\rightharpoonup\mathbb N$ that is given by
\begin{equation*}
\mu f(a_0,\ldots,a_{l-1})\simeq\min\{b\in\mathbb N\,|\,f(a_0,\ldots,a_{l-1},b)=0\}.
\end{equation*}
More precisely, $\mu f(\mathbf a,b)$ is defined with value~$b$ precisely if we have $f(\mathbf a,b')\!\downarrow$ for all~$b'\leq b$ and the number~$b$ is mimimal with $f(\mathbf a,b)=0$.

If an evaluation terminates, this is witnessed by a finite computation sequence that involves only finitely many oracle values. The following result provides a formal version of this simple but crucial observation. We recall that the notion of primitive recursion has been discussed in Section~\ref{sect:prim-rec}.

\begin{theorem}[Kleene Normal Form Theorem~\cite{kleene-nf}]\label{thm:Kleene-normal-form}
There are primitive recursive functions $T,U:\mathbb N^3\to\mathbb N$ such that all~$e\in\mathbb N$, $\sigma\in\mathbb N^{<\omega}$ and $Z\subseteq\mathbb N$ validate
\begin{align*}
\{e\}^Z(\sigma)\!\downarrow\quad&\Leftrightarrow\quad T(e,\sigma,\tau)=0\text{ for some }\tau\sqsubset\chi_Z,\\
T(e,\sigma,\tau)=0\text{ with }\tau\sqsubset\chi_Z\quad &\Rightarrow\quad \{e\}^Z(\sigma)\simeq U(e,\sigma,\tau).
\end{align*}
In the case without oracle, we obtain
\begin{equation*}
\{e\}(\sigma)\simeq U_0(\mu T_0(e,\sigma))
\end{equation*}
for primitive recursive functions $T_0:\mathbb N^2\to\mathbb N$ and $U_0:\mathbb N\to\mathbb N$.
\end{theorem}
\begin{proof}
Assuming that~$e$ is a program, we first determine a number~$N(e,\sigma)\geq l(\sigma)$ such that we have $i,j\leq N(e,\sigma)$ whenever $e_k$ with $k<l(e)$ is of the form $\langle 0,i\rangle$, $\langle 1,i\rangle$, $\langle 2,i,j,m,n\rangle$ or $\langle 3,i\rangle$. This entails that only the first~$N(e,\sigma)$ registers are relevant for the computation of~$\{e\}^Z(\sigma)$. For a coded sequence $\tau\in 2^{<\omega}$, we define
\begin{equation*}
i^\tau_{e,\sigma,k}\in\mathbb N\qquad\text{and}\qquad r^\tau_{e,\sigma,k}:\{0,\ldots,N(e,\sigma)\}\to\mathbb N
\end{equation*}
just like $I^Z_{e,\sigma,k}$ and $R^Z_{e,\sigma,k}$ from Definition~\ref{def:RM-evaluation}, except that the function $r^\tau_{e,\sigma,k}$ has the indicated finite domain and that any reference to $\chi_Z(q)$ is replaced by~$\tau_q$, where we agree on~$\tau_q=0$ for~$q\geq l(\tau)$. To keep track of the unintended case where the oracle is called for values beyond the length of~$\tau$, we define
\begin{equation*}
s^\tau_{e,\sigma,k}=\begin{cases}
0 & \text{if $e_I=\langle 3,p\rangle$ with $I:=i^\tau_{e,\sigma,k}<l(e)$ and $r^\tau_{e,\sigma,k}(p)\geq l(\tau)$},\\
1 & \text{in all other cases}.
\end{cases}
\end{equation*}
The set of secure approximations to some oracle is given by
\begin{equation*}
S_{e,\sigma,k}=\{\tau\in 2^{<\omega}\,|\,s^\tau_{e,\sigma,l}=1\text{ for all }l<k\}.
\end{equation*}
If we have~$\tau\sqsubset\chi_Z$ and $\tau\in S_{e,\sigma,k}$, an induction on~$k\in\mathbb N$ yields
\begin{equation*}
i^\tau_{e,\sigma,k}=I^Z_{e,\sigma,k}\qquad\text{and}\qquad r^\tau_{e,\sigma,k}(p)=R^Z_{e,\sigma,k}(p)\text{ for $p\leq N(e,\sigma)$}.
\end{equation*}
For arbitrary~$e,\sigma,k$ and~$Z$, one has $\tau\in S_{e,\sigma,k}$ when the length of $\tau\sqsubset\chi_Z$ exceeds all values $R^Z_{e,\sigma,l}(p)$ with~$l<k$ for which we have $e_I=\langle 3,p\rangle$ with $I=I^Z_{e,\sigma,l}$.

We now declare that $T(e,\sigma,\tau)=0$ holds precisely when~$e$ is a program and there is a~$k\leq\len(\tau)$ with $\tau\in S_{e,\sigma,k}$ and $i^\tau_{e,\sigma,k}\geq l(e)$. Based on the results of Section~\ref{sect:prim-rec}, one readily checks that the function
\begin{equation*}
(e,\sigma,\tau,k)\mapsto\langle i^\tau_{e,\sigma,k},r^\tau_{e,\sigma,k}(0),\ldots,r^\tau_{e,\sigma,k}(N(e,\sigma))\rangle
\end{equation*}
is primitive recursive, which entails that the same holds for~$T$. To verify the equivalence from the theorem, we first assume that we have~$\{e\}^Z(\sigma)\!\downarrow$. The latter allows us to pick a~$k\in\mathbb N$ with $I^Z_{e,\sigma,k}\geq l(e)$. As seen above, we can find a~$\tau\sqsubset\chi_Z$ such~that we have $\tau\in S_{e,\sigma,k}$ and $k\leq\len(\tau)$. In order to conclude that we have~$T(e,\sigma,\tau)=0$, it suffices to recall that $\tau\in S_{e,\sigma,k}$ entails $i^\tau_{e,\sigma,k}=I^Z_{e,\sigma,k}$. Before we prove the converse implication, we declare that~$U(e,\sigma,\tau)$ is defined as~$r^\tau_{e,\sigma,k}(0)$ for the smallest number~$k\leq\len(\tau)$ such that we have $\tau\in S_{e,\sigma,k}$ and $i^\tau_{e,\sigma,k}\geq l(e)$, if such a~$k$ exists. The function~$U$ is primitive recursive due to Proposition~\ref{prop:bounded-min}. Now assume that we have~$T(e,\sigma,\tau)=0$ with $\tau\sqsubset\chi_Z$. It follows that there is a~$k\leq\len(\tau)$ with $\tau\in S_{e,\sigma,k}$ and $i^\tau_{e,\sigma,k}\geq l(e)$. We obtain
\begin{equation*}
U(e,\sigma,\tau)=r^\tau_{e,\sigma,k}(0)=R^Z_{e,\sigma,k}(0)\qquad\text{and}\qquad I^Z_{e,\sigma,k}=i^\tau_{e,\sigma,k}\geq l(e)
\end{equation*}
for some such~$k$. This already yields~$\{e\}^Z(\sigma)\!\downarrow$, as needed for the open direction of our equivalence. When~$l$ is minimal with $I^Z_{e,\sigma,l}\geq l(e)$, Definition~\ref{def:RM-evaluation} yields
\begin{equation*}
\{e\}^Z(\sigma)\simeq R^Z_{e,\sigma,l}(0)=R^Z_{e,\sigma,l+1}(0)=\ldots=R^Z_{e,\sigma,k}(0),
\end{equation*}
which shows that we indeed have~$\{e\}^Z(\sigma)\simeq U(e,\sigma,\tau)$.

For the case without oracle, we note that $\tau\sqsubset\chi_\emptyset$ is a primitive recursive property of the coded sequence~$\tau\in 2^{<\omega}$. We now declare that~$T_0(e,\sigma,\rho)=0$ holds precisely if~$\rho=\langle e,\sigma,\tau,k\rangle$ codes a sequence with $\tau\sqsubset\chi_\emptyset$ and $T(e,\sigma,\tau)=0$ such that~$k\leq\len(\tau)$ is minimal with $\tau\in S_{e,\sigma,k}$ and $i^\tau_{e,\sigma,k}\geq l(e)$. We also put~$U_0(\langle e,\sigma,\tau,k\rangle)=r^\tau_{e,\sigma,k}(0)$. To reach the desired conclusion, it is enough to observe that $T_0(e,\sigma,\rho)=0$ holds for some~$\rho$ precisely if $T(e,\sigma,\tau)=0$ holds for some~$\tau\sqsubset\chi_\emptyset$, and that $T_0(e,\sigma,\rho)=0$ with $\rho=\langle e,\sigma,\tau,k\rangle$ entails~$U_0(\rho)=U(e,\sigma,\tau)$.
\end{proof}

When we consider the values~$\{e\}^Z(\sigma)$ of a program, we will mostly think of~$\sigma$ as the variable argument and of~$Z$ as some fixed parameter. However, one can also consider~$Z$ as an argument, which leads to a notion of computation with infinite data. The observation that each computation can only access finitely many oracle values amounts to a continuity principle that is explored in the following exercise. We note that part~(b) of the exercise does not reverse, i.\,e., that there are continuous functions that are not computable (cf.~Theorem~\ref{thm:halting}). Part~(a) shows why it can be necessary to consider partial functions. It is no real restriction to consider~$\{e\}^Z(\sigma)$ for~$\sigma=\langle\rangle$ only, as an arbitrary~$\sigma$ can be coded into~$Z$.

\begin{exercise}
(a) Prove that there is a program~$e$ with $\{e\}^Z(\langle\rangle)\simeq\mu\chi_Z$ (which yields the smallest~$n\notin Z$ when $Z\neq\mathbb N$) but no total extension, i.\,e., no program~$e'$ such that we have $\{e'\}^Z(\langle\rangle)\simeq\mu\chi_Z$ for $Z\neq\mathbb N$ as well as $\{e'\}^{\mathbb N}(\langle\rangle)\!\downarrow$. \emph{Hint:} You can conclude via part~(b). For a direct proof, assume that~$e'$ is as given and consider $T(e',\langle\rangle,\tau,\langle k,r\rangle)=0$ with~$\tau\sqsubset\chi_{\mathbb N}$. Then pick a $Z\neq\mathbb N$ with $\tau\sqsubset\chi_Z$ and $\mu\chi_Z\neq r$.

(b) Show that~$\mathbb N\supseteq Z\mapsto\{e\}^Z(\langle\rangle)\in\mathbb N$ is continuous for any program~$e$, where the domain carries the topology from Exercise~\ref{ex:compactness} and the topology on~$\mathbb N$ is discrete.

(c) Given~$e$ with $\{e\}^Z(\langle\rangle)\!\downarrow$ for all~$Z\subseteq\mathbb N$, show that there is a number~$N\in\mathbb N$ such that we have $\{e\}^Y(\langle\rangle)\simeq\{e\}^Z(\langle\rangle)$ whenever~$\chi_Y(n)=\chi_Z(n)$ holds for all~$n<N$. \emph{Hint:} In view of Exercise~\ref{ex:compactness}, the claim is an instance of the fact that a con\-tin\-u\-ous function on a compact space is uniformly continuous.
\end{exercise}

Our next aim is to show that several notions of computability are equivalent. We focus on the case without oracle in order to save on notation. It is straightforward to generalize our considerations to the case of computability with oracle. The~following definition seems natural in view of the Kleene normal form theorem. We note that the composition of partial functions is explained by
\begin{equation*}
h(g_1(\mathbf a),\ldots,g_n(\mathbf a))\simeq c\quad\Leftrightarrow\quad\text{there are~$b_i$ with $g_i(\mathbf a)\simeq b_i$ and $h(b_1,\ldots,b_n)\simeq c$}.
\end{equation*}
Hence $h(g_1(\mathbf a),\ldots,g_n(\mathbf a))$ is undefined if the same holds for some argument~$g_i(\mathbf a)$. Recursive definitions on partial functions have the same form as in clause~(v) of Definition~\ref{def:prim-rec} but with~$\simeq$ at the place of~$=$ (where~$f(\mathbf a,b+1)\simeq h(\mathbf a,b,f(\mathbf a,b))$ is undefined when the same holds for~$f(\mathbf a,b)$).

\begin{definition}
The class of $\mu$-recursive functions is the smallest class of partial functions~$\mathbb N^n\rightharpoonup\mathbb N$ that contains the constant functions, the projections and the successor function (i.\,e., the total functions from clauses~(i) to~(iii) of Definition~\ref{def:prim-rec}) and is closed under compositions, recursive definitions and the operation~$f\mapsto\mu f$.
\end{definition}

When~$n$ is clear from the context, we also write~$\{e\}$ for the function~$\mathbb N^n\rightharpoonup\mathbb N$ that is given by $\{e\}(a_1,\ldots,a_n)\simeq\{e\}(\langle a_1,\ldots,a_n\rangle)$. Condition~(c) in the following theorem amounts to an extension of Definition~\ref{def:sigma-definable} to the partial case. The equivalence of~(b) and~(c) can be seen as an optimal strengthening of Theorem~\ref{thm:prim-rec-sigma}.

\begin{theorem}\label{thm:comp-equiv}
For any partial function~$f:\mathbb N^n\rightharpoonup\mathbb N$, the following are equivalent:
\begin{enumerate}[label=(\alph*)]
\item We have $f\simeq\{e\}$ for some program~$e$.
\item The function~$f$ is $\mu$-recursive.
\item There is a $\Sigma$-formula $\varphi(x_1,\ldots,x_n,y)$ such that $f(a_1,\ldots,a_n)\simeq b$ is equivalent to $\mathbb N\vDash\varphi(a_1,\ldots,a_n,b)$ for all natural numbers~$a_i$ and~$b$.
\end{enumerate}
\end{theorem}
\begin{proof}
We first show that~(b) implies~(c). By the proof of Theorem~\ref{thm:prim-rec-sigma}, it suffices to check that~(c) holds for~$\mu f$ whenever it holds for~$f$. Now if $f(\mathbf a,b)=c$ is equivalent to~$\mathbb N\vDash\varphi(\mathbf a,b,c)$, then $\mu f(\mathbf a)=b$ is~equi\-valent to
\begin{equation*}
\mathbb N\vDash\varphi(\mathbf a,b,0)\land\forall b'<b\exists c\left(c\neq 0\land\varphi(\mathbf a,b,c)\right).
\end{equation*}
To see that~(c) implies~(b), we recall that~$\varphi(\mathbf x,y)$ is equivalent to~$\exists z\,\theta(\mathbf x,y,z)$ for some bounded formula~$\theta$, as seen in Exercise~\ref{ex:sigma-sigma_1}. We know from Exercise~\ref{ex:bounded-prim-rec} that $\mathbb N\vDash\theta(\mathbf a,b,c)$ is equivalent to~$g(\mathbf a,\pi(b,c))=0$ for some primitive recursive~$g$ (where~$\pi$ refers to Cantor pairing). Assuming that~$f$ is characterized by the equivalence from~(c), we now derive $f(\mathbf a)\simeq U(\mu g(\mathbf a))$ for~$U(\pi(b,c))=b$. First assume that $\mu g(\mathbf a)$ is defined with value~$\pi(b,c)$. Then~$c$ witnesses~$\mathbb N\vDash\exists z\,\theta(\mathbf a,b,z)$, so that we obtain~$f(\mathbf a)\simeq b\simeq U(\mu g(\mathbf a))$. Conversely, if~$f(\mathbf a)$ is defined, then there are~$b,c$ with $\mathbb N\vDash\theta(\mathbf a,b,c)$, which yields~$\mu g(\mathbf a)\!\downarrow$.

The Kleene normal form theorem ensures that~(a) implies~(b). For the converse, we argue my induction over the generation of~$\mu$-recursive functions according to the previous definition. Constant functions and the successor are readily implemented as programs. At the beginning of this section, we have seen how the content of one register can be copied into another, which accounts for the projections. Let us now consider a $\mu$-recursive function that arises as a composition
\begin{equation*}
f(a_1,\ldots,a_m)\simeq h(g_1(a_1,\ldots,a_m),\ldots,g_n(a_1,\ldots,a_m)).
\end{equation*}
The induction hypothesis provides programs to compute the~$g_i$ and~$h$. Let~\mbox{$N\geq m,n$} be so large that these given programs refer to the registers~$r_0,\ldots,r_{N-1}$ only.\nobreak\ A~pro\-gram that computes~$f$ can be described as follows: One first has~$m$ subprograms that copy each argument $a_{1+i}$ from~$r_i$ into~$r_{N+i}$. These are followed by~$n$ subprograms, each of which fills $r_0,\ldots,r_{m-1}$ with the arguments stored in~$r_N,\ldots,r_{N+i}$ (but leaves the latter unchanged for later re-use), sets $r_m,\ldots,r_{N-1}$ to zero, computes the respective value~$g_i(a_1,\ldots,a_m)$ via the given program (recall that instruction numbers need to be adapted consistently) and copies it from~$r_0$ into~$r_{N+m+i}$. Finally, one copies the results from~$r_{N+m},\ldots,r_{N+m+n-1}$ into~$r_0,\ldots,r_{n-1}$, sets the registers~$r_n,\ldots,r_{N-1}$ to zero and concludes with the program for~$h$. We now consider a function that is recursively determined by
\begin{equation*}
f(\mathbf a,0)\simeq g(\mathbf a)\quad\text{and}\quad f(\mathbf a,b+1)\simeq h(\mathbf a,b,f(\mathbf a,b)).
\end{equation*}
Given programs for~$g$ and~$h$, the values~$f(\mathbf a,b)$ can be computed as follows: First copy the arguments~$\mathbf a,b$ into registers from where they can be recovered at any~point. Then compute~$g(\mathbf a)$ and place it into a suitable register~$r_i$. Pick a fresh register~$r_j$ and set it to zero. Use a jump instruction to test whether the current value~$b'$ in~$r_j$ is equal~to~$b$. If this is not the case, take the current content~$c$ of~$r_i$, compute~$h(\mathbf a,b,c)$ and place the result into~$r_i$. Then increase~$r_j$ to~$b'+1$ and jump back to the jump instruction from before. Once~$r_j$ has been increased to~$b$, the content of~$r_i$ will be equal to~$f(\mathbf a,b)$, as one can see by induction. One then jumps to a subprogram that copies~$r_i$ to the output register~$r_0$. Finally, we have a function with specification
\begin{equation*}
f(\mathbf a)\simeq\mu g(\mathbf a).
\end{equation*}
To compute~$f$ based on a program for~$g$, one uses a jump instruction to fill some registers~$r_i$ and~$r_j$ with the numbers~$b=0,1,\ldots$ and the corresponding value~$g(\mathbf a,b)$ as long as the latter is positive. If it becomes zero, one copies its content~$\mu g(\mathbf a)\simeq b$ from~$r_i$ into the output register~$r_0$.
\end{proof}

The theorem suggests that there is a single robust notion of computability. This justifies the following terminology.

\begin{definition}
A partial function~$f:\mathbb N^n\rightharpoonup\mathbb N$ is called computable in or relative to a set $Z\subseteq\mathbb N$ if there is a program~$e$ such that we have $f\simeq\{e\}^Z$. We say that a set~$Y\subseteq\mathbb N$ is computable in~$Z$ when the same holds for its characteristic function. In this case we write $Y\leq_{\mathsf T}Z$. Let us also agree that $Y\equiv_{\mathsf T}Z$ denotes the conjunction of $Y\leq_{\mathsf T}Z$ and $Z\leq_{\mathsf T}Y$ while $Y<_{\mathsf T}Z$ means that we have $Y\leq_{\mathsf T}Z$ but $Z\not\leq_{\mathsf T}Y$. When a function or set is computable relative to the empty set, we simply say that it is computable.
\end{definition}

From the following exercise we learn that~$\equiv_{\mathsf T}$ is an equivalence relation. The associated equivalence classes are known as the Turing degrees.

\begin{exercise}
Show that $\leq_{\mathsf T}$ is a transitive relation on subsets of~$\mathbb N$.
\end{exercise}

As a basis for the following remark, we note that other suggested definitions of com\-put\-ability (e.\,g., via Turing machines or the lambda calculus of Alonzo Church) are also equivalent to the given one. 

\begin{remark}[Church-Turing Thesis]
According to the Church-Turing thesis (see~\cite{copeland-church-turing} for more details), the formal definition that we have given captures `the' informal notion of computability and covers any algorithm that can possibly be evaluated on something like a computer. The thesis admits no formal proof, since it explicitly refers to an informal notion. At the same time, it is widely accepted and receives strong support from the equivalence between the formal definitions that have been suggested, from experience with computers in practice and from the idea that a computation is a finite sequence of elementary manipulations of data (cf.~the proof of the Kleene normal form theorem). Some authors appeal to the Church-Turing thesis in order to argue that some function that is intuitively computable must also be computable according to the formal definition.
\end{remark}

Besides its role in connection with the Church-Turing thesis, Theorem~\ref{thm:comp-equiv} has concrete technical applications. When combined with Kleene's normal form theorem, it entails that any~$\mu$-recursive function has a definition in which the~\mbox{$\mu$-operator} occurs just once. The fact that programs can perform $\mu$-recursion is also used in the proof of the following result. We note that the latter can readily be extended to functions~$\{e\}^Z:\mathbb N^n\to\mathbb N$ with $n>1$.

\begin{theorem}\label{thm:univ-machine}
There is a so-called universal machine or universal program~$u$ such that all $e,n\in\mathbb N$ and $Z\subseteq\mathbb N$ validate
\begin{equation*}
\{u\}^Z(e,n)\simeq\{e\}^Z(n).
\end{equation*}
In other words, the function $(e,n)\mapsto\{e\}^Z(n)$ is computable uniformly in~$Z$.
\end{theorem}
\begin{proof}
In the case where~$Z$ is empty, we can argue that the function~$(e,n)\mapsto\{e\}(n)$ is $\mu$-recursive by Kleene's normal form theorem, so that it can be computed by a program due to Theorem~\ref{thm:comp-equiv}. To accommodate the oracle, we first recall that the length and entries of a coded finite sequence can be read off by primitive recursive functions (see Lemma~\ref{lem:seq-code}). In view of Theorem~\ref{thm:comp-equiv}, one can derive that there is a program~$u_0$ that does not depend on~$Z$ and validates
\begin{equation*}
\{u_0\}^Z(\tau)\simeq\begin{cases}
1 & \text{if $\tau$ is a coded sequence with $\tau\sqsubset\chi_Z$},\\
0 & \text{otherwise}.
\end{cases}
\end{equation*}
The desired program~$u$ makes use of~$u_0$ in an unbounded search for the smallest~$\tau$ that codes a sequence with~$T(e,\langle n\rangle,\tau)$ and~$\tau\sqsubset\chi_Z$, which is then used in the~compu\-tation of~$\{e\}^Z(n)\simeq U(e,\langle n\rangle,\tau)$, where~$T$ and~$U$ are the functions from Kleene's normal form theorem.
\end{proof}

We will see that the following provides an example of an uncomputable set.

\begin{definition}
The Turing jump or halting set of a set~$Z\subseteq\mathbb N$ is defined as
\begin{equation*}
Z'=\{\pi(e,n)\,|\,\{e\}^Z(n)\!\downarrow\}\subseteq\mathbb N,
\end{equation*}
where~$\pi$ is the Cantor pairing function.
\end{definition}

Let us note that we have used pairing in order to achieve~$Z'\subseteq\mathbb N$. In view of Exercise~\ref{ex:jump} below, this allows us to view the Turing jump as an operation on degrees. By the same exercise, other possible definitions of the Turing jump are equivalent to the given one.

\begin{theorem}[Halting problem]\label{thm:halting}
We have $Z'\not\leq_{\mathsf T}Z$ for any~$Z\subseteq\mathbb N$. In particular, the set $\emptyset'\subseteq\mathbb N$ is uncomputable.
\end{theorem}
\begin{proof}
Aiming at a contradiction, we assume that~$Z'$ is computable relative to~$Z$. Using Theorem~\ref{thm:univ-machine}, we conclude that the same holds for the total~$f:\mathbb N\to\mathbb N$~with
\begin{equation*}
f(e)=\begin{cases}
\{e\}^Z(e)+1 & \text{when $(e,e)\in Z'$},\\
0 & \text{otherwise}.
\end{cases}
\end{equation*}
But then there is a program~$e_0$ with~$\{e_0\}^Z(e)\simeq f(e)$. Given that the function~$f$ is total, we get $(e_0,e_0)\in Z'$ and hence
\begin{equation*}
\{e_0\}^Z(e_0)=f(e_0)=\{e_0\}^Z(e_0)+1,
\end{equation*}
which is impossible.
\end{proof}

In the following exercise, we show that the class of total computable functions admits no universal machine, which confirms the crucial role of partial functions in the theory of computation.

\begin{exercise}\label{ex:computable-partial}
Show that there is no total computable function~$f:\mathbb N^2\to\mathbb N$ such that~$f(e,n)$ is equal to~$\{u\}(e,n)\simeq\{e\}(n)$ whenever the function~$\{e\}:\mathbb N\to\mathbb N$ is total. \emph{Hint:} Assuming that~$f$ is as indicated, consider~$e_0$ with $\{e_0\}(e)=f(e,e)+1$, similarly to the proof of the previous theorem.
\end{exercise}

\begin{lemma}[S-m-n Theorem]
For each pair of natural numbers~$m$ and~$n$, there is a primitive recursive function~$S^m_n:\mathbb N^{1+m}\to\mathbb N$ such that
\begin{equation*}
\{e\}^Z(a_1,\ldots,a_m,b_1,\ldots,b_n)\simeq\{S^m_n(e,a_1,\ldots,a_m)\}^Z(b_1,\ldots,b_n)
\end{equation*}
holds for any program~$e$, oracle~$Z\subseteq\mathbb N$ and natural numbers~$a_i$ and~$b_j$.
\end{lemma}
\begin{proof}
The program~$S^m_n(e,\mathbf a)$ copies the arguments~$b_j$ from the registers~$r_0,\ldots,r_{n-1}$ into $r_m,\ldots,r_{m+n-1}$, then uses zero and successor instructions to fill $r_0,\ldots,r_{m-1}$ with the arguments~$a_i$ and finally employs the program~$e$. 
\end{proof}

The aim of the following exercise is to establish basic properties of the Turing jump, which were promised in the paragraph before Theorem~\ref{thm:halting}.

\begin{exercise}\label{ex:jump}
(a) Show that any~$X\subseteq\mathbb N$ validates
\begin{equation*}
X'\equiv_{\mathsf T}\{e\in\mathbb N\,|\,\{e\}^X(e)\!\downarrow\}\equiv_{\mathsf T}\{e\in\mathbb N\,|\,\{e\}^X(\langle\rangle)\!\downarrow\}.
\end{equation*}
\emph{Hint:} The s-m-n theorem yields~$\{e\}^X(n)\simeq\{S^1_0(e,n)\}^X(\langle\rangle)$.

(b) Prove that $X\leq_{\mathsf T}Y$ entails $X'\leq_{\mathsf T}Y'$. \emph{Hint:} Show that there is a primitive recursive function~$g:\mathbb N\to\mathbb N$ such that~$\{e\}^X(n)\simeq\{g (e)\}^Y(n)$ holds for all~$e,n\in\mathbb N$.

(c) Show that we have $X\leq_{\mathsf T}X'$ and hence~$X<_{\mathsf T}X'$. \emph{Hint:} Find a program~$e$ such that~$\{e\}^X(n)\!\downarrow$ holds precisely for~$n\in X$.
\end{exercise}

We conclude our brief introduction to computability with a notion that will be relevant in the next section.

\begin{definition}\label{def:ce}
A subset of~$\mathbb N$ is called computably enumerable (c.e.) if it is of the form~$\{n\in\mathbb N\,|\,\{e\}(n)\!\downarrow\}$ for some program~$e$.
\end{definition}

The terminology in the definition is motivated by part~(a) of the following.

\begin{exercise}\label{ex:ce}
(a) Show that a non-empty set~$X\subseteq\mathbb N$ is computably enumerable in the sense of Definition~\ref{def:ce} precisely if it has the form~$\{f(n)\,|\,n\in\mathbb N\}$ for some total computable~$f:\mathbb N\to\mathbb N$. Also prove that one can choose an injective~$f$ when the set~$X$ is infinite. \emph{Hint:} Fix some~$n_0\in X$. Set~$f(\pi(n,k))=n$ when $T_0(e,\langle n\rangle,k)=0$ and $f(\pi(n,k))=n_0$ otherwise, where~$T_0$ is as in Kleene's normal form theorem.

(b) Prove that $\{f(n)\,|\,n\in\mathbb N\}$ is computable for any total computable $f:\mathbb N\to\mathbb N$ that is strictly increasing.

(c) Show that a set $X\subseteq\mathbb N$ is computable when both~$X$ and~$\mathbb N\backslash X$ are computably enumerable. \emph{Hint:} Assume that $f_0$ and $f_1$ list the elements of $X$ and~$\mathbb N\backslash X$ as in part~(a). Consider a program that alternately computes~$f_0(k)$ and~$f_1(k)$. Note that any given~$n\in\mathbb N$ must occur as~$n=f_i(k)$ for some~$i<2$.

(d) Use part~(c) to give another proof of G\"odel's first incompleteness theorem. \emph{Hint:} Assuming that~$\mathsf T$ is a consistent theory that proves any true~$\Pi$-formula, we have $\pi(e,n)\in\mathbb N\backslash\emptyset'$ precisely if~$\mathsf T$ proves that~$\{e\}(n)$ is undefined.
\end{exercise}

Let us note that the sets $X$ and~$Y$ in the following result cannot be computable, as we could otherwise take~$S=X$ or~$S=\mathbb N\backslash Y$.

\begin{proposition}\label{prop:comp-insep}
There are computably enumerable and disjoint sets~$X,Y\subseteq\mathbb N$ that are computably inseparable, i.\,e., that admit no computable set~$S\subseteq\mathbb N$ such that we have $X\subseteq S$ and~$Y\subseteq\mathbb N\backslash S$.
\end{proposition}
\begin{proof}
Consider the sets $X=\{e\in\mathbb N\,|\,\{e\}(e)\simeq 0\}$ and $Y=\{e\in\mathbb N\,|\,\{e\}(e)\simeq 1\}$. These are computably enumerable since Theorem~\ref{thm:univ-machine} yields~$\{e\}(e)\simeq\{u\}(e,e)$. Now assume we have a computable separator~$S$. Pick a program~$e_0$ with $\{e_0\}\simeq\chi_S$. We then get the implications
\begin{alignat*}{3}
e_0\in S\quad&\Rightarrow\quad \{e_0\}(e_0)\simeq 1\quad&&\Rightarrow\quad e_0\in Y\quad&&\Rightarrow\quad e_0\in\mathbb N\backslash S,\\
e_0\in\mathbb N\backslash S\quad&\Rightarrow\quad\{e_0\}(e_0)\simeq 0\quad&&\Rightarrow\quad e_0\in X\quad&&\Rightarrow\quad e_0\in S,
\end{alignat*}
which are contradictory.
\end{proof}

If the scope of the present lecture had allowed it, we would have continued with a presentation of Kleene's recursion theorem. The interested reader will find this fundamental result in various textbooks on computability theory.

\subsection{An Analysis of K\H{o}nig's Lemma}\label{subsect:KL}

In this section, we employ computability theory to analyse two forms of K\H{o}nig's lemma. Let us first show that even the~weak K\H{o}nig's lemma (see Theorem~\ref{thm:wkl}) involves sets that are uncomputable. We recall that binary trees are realized as subtrees of~$2^{<\omega}$, which consists of the finite sequences with entries in $\{0,1\}$. Some notation for sequences has been explained in the paragraph before Definition~\ref{def:binary-tree}. Here and in the following, we invoke the sequence encoding given by Definition~\ref{def:seq-code} (see also Lemma~\ref{lem:seq-code}) to view a tree as a subset of~$\mathbb N$, to which the notions from computability theory can be applied.

\begin{theorem}\label{thm:wkl-uncomputable}
There is an infinite binary tree that is computable but does not have a computable branch.
\end{theorem}
\begin{proof}
From Proposition~\ref{prop:comp-insep} we have disjoint sets~$X_0$ and $X_1$ that are computably enumerable and cannot be computably separated. In view of Exercise~\ref{ex:ce}, we may write $X_i=\{f_i(k)\,|\,k\in\mathbb N\}$ for computable functions~$f_i:\mathbb N\to\mathbb N$. Here the latter are total, so that the relation $f_i(k)=n$ is computable. Let $T\subseteq 2^{<\omega}$ consist of all sequences $\langle\sigma_0,\ldots,\sigma_{l-1}\rangle$ such that, for any $n<l$ and $i<2$, we have $\sigma_n=i$ if there is a $k<l$ with $f_i(k)=n$. We point out that $\sigma_n$ may be arbitrary when no such~$k$ exists for either~$i<2$. It is readily seen that $T$ is a computable tree. We claim that a function $g:\mathbb N\to\{0,1\}$ is a branch of~$T$ precisely when $Y_g=\{n\in\mathbb N\,|\,g(n)=0\}$ separates $X_0$ and $X_1$, i.\,e., when we have $X_0\subseteq Y_g$ and $X_1\subseteq\mathbb N\backslash Y_g$. Once this is established, we learn that $T$ has a branch and is thus infinite (as we get~$Y_g=X_0$ when we stipulate that $g(n)=0$ holds precisely for~$n\in X$). We also learn that $Y_g$ and hence~$g$ is uncomputable whenever~$g$ is a branch.

To prove the open claim, we first assume $g$ is a branch of~$T$. Aiming at~$X_0\subseteq Y_g$, we consider an arbitrary element~$n\in X_0$. The latter can be written as $n=f_0(k)$. For $l=\max(k,n)+1$, the definition of~$T$ ensures that~$g[l]=\langle g(0),\ldots,g(l-1)\rangle\in T$ implies $g(n)=0$, which yields $n\in Y_g$. An analogous argument shows $X_1\subseteq\mathbb N\backslash Y_g$. For the remaining direction, we now assume that $Y_g$ separates $X_0$ and~$X_1$. To show that $g[l]\in T$ holds for any~$l\in\mathbb N$, we consider $k,n<l$ and $i<2$ with $f_i(k)=n$. Our task is to derive~$g(n)=i$. If we have~$i=0$, the latter holds due to $n\in X_0\subseteq Y_g$. To conclude for $i=1$, it suffices to note that $g(n)\neq 0$ follows from $n\in X_1\subseteq\mathbb N\backslash Y_g$.
\end{proof}

We now state the stronger version of K\H{o}nig's lemma that was alluded to above. Let $\mathbb N^{<\omega}$ denote the set of finite sequences with entries in~$\mathbb N$. Our previous notation for sequences in~$2^{<\omega}$ will also be used for these sequences. A tree is a set~$T\subseteq\mathbb N^{<\omega}$ such that $\sigma\sqsubset\tau\in T$ entails $\sigma\in T$. By a branch of a tree~$T$, we mean a function $f:\mathbb N\to\mathbb N$ with $f[n]=\langle f(0),\ldots,f(n-1)\rangle\in T$ for all~$n\in\mathbb N$. A tree~$T$ is called finitely branching if each $\sigma\in T$ admits an $N\in\mathbb N$ with $\sigma\star i\notin T$ for all $i\geq N$.

\begin{theorem}[K\H{o}nig's Lemma]\label{thm:kl}
Any tree that is both infinite and finitely branching does have a branch.
\end{theorem}
\begin{proof}
It is straightforward to adapt the proof of Theorem~\ref{thm:wkl}, since any finite union of finite sets is finite.
\end{proof}

The following exercise gives a first idea how~$2^{<\omega}$ and $\mathbb N^{<\omega}$ differ on questions of computability. We note that the upper bound in part~(a) will later be improved.

\begin{exercise}
(a) Show that any computable binary tree has a branch~$f\leq_{\mathsf T}\emptyset'$. \emph{Hint:} Devise an algorithm that halts if the subtree above a given sequence is finite.

(b) Extend the result of (a) to trees $T\subseteq\mathbb N^{<\omega}$ that are computably bounded, which means that they admit a computable $h:T\to\mathbb N$ with $\sigma\star i\notin T$ for~$i\geq h(\sigma)$.

(c) Show that any computable tree that is infinite and finitely branching has a branch~$f\leq_{\mathsf T}\emptyset''$. \emph{Hint:} There is a function~$h\leq_{\mathsf T}\emptyset'$ as in the hint for~(b). 
\end{exercise}

As the exercise suggests, the reference to binary trees is somewhat misleading. The crucial feature of our binary trees is that they are computably bounded, since we have agreed that they are realized as subtrees of~$2^{<\omega}$. On the other hand, the following proof involves a tree~$T$ that is not computably bounded and not binary in our sense, even though each~$\sigma\in T$ has at most two immediate successors~$\sigma\star i\in T$.

\begin{theorem}\label{thm:KL-to-jump}
There is a computable tree~$T$ that is finitely branching and has exactly one branch~$f$, for which we have $\emptyset'\leq_{\mathsf T}f$.
\end{theorem}
\begin{proof}
In view of Exercise~\ref{ex:ce}, we may write $\emptyset'=\{g(k)\,|\,k\in\mathbb N\}$ for a computable injection~$g:\mathbb N\to\mathbb N$. Let $T$ consist of all sequences $\langle\sigma_0,\ldots,\sigma_{l-1}\rangle\in\mathbb N^{<\omega}$ such that the following holds for all $n<l$:
\begin{enumerate}[label=(\roman*)]
\item If we have $\sigma_n=0$, then there is no~$k<l$ with $g(k)=n$.
\item If we have $\sigma_n>0$, then we get $g(\sigma_n-1)=n$.
\end{enumerate}
Given that~$f$ is injective, clause~(ii) can be validated by at most one number~$\sigma_n$. This entails that $T$ is finitely branching (with at most two immediate successors at each node). Consider the function $f:\mathbb N\to\mathbb N$ with
\begin{equation*}
f(n)=\begin{cases}
0 & \text{if $n\notin\emptyset'$},\\
k+1 & \text{if $n=g(k)$}.
\end{cases}
\end{equation*}
It is straightforward to see that~$f$ is a branch of~$T$. To establish uniqueness, we consider an arbitrary branch~$h$. Given~$n\in\emptyset'$, we set $l=\max(k,n)+1$ for the unique $k\in\mathbb N$ with~$n=g(k)$. As we have~$h[l]\in T$, clause~(i) above yields~$h(n)>0$. Due to clause~(ii), we then get $h(n)=k+1=f(n)$. For $n\notin\emptyset'$ we obtain~$h(n)=0=f(n)$ by a similar but simpler argument. Finally, we note that $\emptyset'\leq_{\mathsf T}f$ holds since $n\in\emptyset'$ is equivalent to $f(n)>0$.
\end{proof}

In order to show that the weak version of K\H{o}nig's lemma admits computationally simpler solutions, we need to find an uncomputable set that lies strictly below the Turing jump. The following is a source of such sets.

\begin{definition}
A set $X\subseteq\mathbb N$ is called low if we have $X'\leq_{\mathsf T}\emptyset'$.
\end{definition}

Let us make explicit that we get $X<_{\mathsf T}\emptyset'$ when~$X$ is low. As a consequence, the next result entails that Theorem~\ref{thm:KL-to-jump} becomes wrong when we demand~$T\subseteq 2^{<\omega}$. So indeed, weak K\H{o}nig's lemma is strictly weaker from a computational perspective.

\begin{theorem}[Low basis theorem~\cite{low-basis}]
Any computable tree $T\subseteq 2^{<\omega}$ that is infinite has a low branch.
\end{theorem}
\begin{proof}
In view of Theorem~\ref{thm:Kleene-normal-form}, we have a computable function $g:\mathbb N^3\to\{0,1\}$ such that any $Y\subseteq\mathbb N$ validates
\begin{equation*}
(e,n)\in Y'\quad\Leftrightarrow\quad\text{$g(e,n,\sigma)=1$ for some $\sigma\sqsubset\chi_Y$}.
\end{equation*}
Here we write $\chi_Y$ for the characteristic function and use coding to view $\sigma\in 2^{<\omega}$ as a natural number. For each $p\in\mathbb N$, which we consider as the Cantor pair with components $\pi_0(p)$ and $\pi_1(p)$, we form the tree
\begin{equation*}
S_p=\{\tau\in 2^{<\omega}\,|\,g(\pi_0(p),\pi_1(p),\sigma)=0\text{ for all }\sigma\sqsubseteq\tau\}.
\end{equation*}
To define a sequence of binary trees~$T_0\supseteq T_1\supseteq\ldots$, we recursively stipulate
\begin{equation*}
T_0=T\quad\text{and}\quad T_{p+1}=\begin{cases}
T_p\cap S_p&\text{if $T_p\cap S_p$ is infinite},\\
T_p & \text{otherwise}.
\end{cases}
\end{equation*}
The tree $T_\infty=\bigcap_{p\in\mathbb N}T_p$ must be infinite. If it was not, there would be an $n\in\mathbb N$ with $\sigma\notin T_\infty$ for all sequences~$\sigma\in 2^{<\omega}$ of length $l(\sigma)=n$. The number of these sequences is $2^n$ and in particular finite. Hence we would get a~$p\in\mathbb N$ with $\sigma\notin T_p$ for~$l(\sigma)=n$, against the fact that~$T_p$ is infinite. By K\H{o}nig's lemma, we now obtain a branch~$f$ of the tree $T_\infty\subseteq T$.

In order to show that our branch is low, we first recall that $f$ is identified with the coded graph $\mathcal G(f)=\{\langle n,i\rangle\,|\,f(n)=i\}$. Let us also consider $Y=\{n\in\mathbb N\,|\,f(n)=1\}$. We clearly have $Y\equiv_{\mathsf T}\mathcal G(f)$. Due to Exercise~\ref{ex:jump}, it is thus enough to show $Y'\leq_{\mathsf T}\emptyset'$. Given that $f$ is the characteristic function of~$Y$, we know that $(e,n)\in Y'$ holds precisely when we have $g(e,n,\sigma)=1$ for some~$\sigma\sqsubset f$. This entails
\begin{equation*}
(e,n)\in Y'\quad\Leftrightarrow\quad T_{\langle e,n\rangle}\cap S_{\langle e,n\rangle}\text{ is finite}.
\end{equation*}
Indeed, if the tree $T_{\langle e,n\rangle}\cap S_{\langle e,n\rangle}$ is finite, some $\sigma\sqsubset f$ is not contained in it. Given that $f$ is a branch of~$T_\infty\subseteq T_{\langle e,n\rangle}$, it follows that $\sigma$ does not lie in~$S_{\langle e,n\rangle}$. By the definition of the latter, we get $g(e,n,\sigma)=1$ and hence $(e,n)\in Y'$. On the other hand, if $T_{\langle e,n\rangle}\cap S_{\langle e,n\rangle}$ is infinite, then $f$ is a branch of~$T_{\langle e,n\rangle+1}\subseteq S_{\langle e,n\rangle}$. We thus obtain $g(e,n,\sigma)=0$ for any~$\sigma\sqsubset f$, as need to establish $(e,n)\notin Y'$.

It remains to show that the Turing jump can compute the set of indices $p\in\mathbb N$ such that $T_p\cap S_p$ is finite. Let us write
\begin{equation*}
T^\rho=T\cap\bigcap\left\{\left.S_{\rho(i)}\,\right|\,i<l\text{ with }\rho_i=1\right\}\quad\text{for}\quad\rho=\langle\rho_0,\ldots,\rho_{l-1}\rangle\in 2^{<\omega}.
\end{equation*}
Let $c$ be the code of a program that takes an input~$\rho$ and searches for an $n\in\mathbb N$ with $\sigma\notin T^\rho$ for all~$\sigma\in\ 2^{<\omega}$ of length~$l(\sigma)=n$. Since the computation of $\{c\}(\rho)$ terminates precisely when $T^\rho$ is finite, we get
\begin{equation*}
\left\{\left.\rho\in 2^{<\omega}\,\right|\,T^\rho\text{ is finite}\right\}\leq_{\mathsf T}\emptyset'.
\end{equation*}
By course-of-values recursion (see Proposition~\ref{prop:course-of-values}), we obtain a $\emptyset'$-computable function $h:\mathbb N\to\{0,1\}$ such that $h(p)=1$ holds precisely when~$T^{h[p]\star 1}$ is infinite. A straightforward induction yields $T^{h[p]}=T_p$ and hence $T^{h[p]\star 1}=T_p\cap S_p$. Finally, we can conclude $Y'=\{p\in\mathbb N\,|\,h(p)=0\}\leq_{\mathsf T}\emptyset'$.
\end{proof}

\section{Model Theory}\label{sect:model-theory}

This section presents some basic ideas from model theory. The latter is mostly concerned with theories that are more `tame' than those of first-order arithmetic or set theory. These theories are not intended to serve as foundation for a large part of mathematics. Instead, they axiomatize some particular structure, such as the class of vector spaces or algebraically closed fields. The natural numbers are usually not definable in the theories studied by model theory, so that G\"odel's incompleteness theorems do not apply. One often encounters theories that are complete and have models with a great deal of symmetry. Model theory has strong connections with other parts of mathematics, in particular with algebraic geometry.

In Subsection~\ref{subsect:ultraprod} below, we present two classical constructions that combine a collection of given models into a new model, namely products with respect to an ultrafilter and unions along chains. The constructions are useful in applications but also for foundational purposes. We employ them to derive completeness and compactness for uncountable signatures (recall that only the countable case was treated in Section~\ref{sect:fundamentals}). This will allow us to prove an upward L\"owenheim-Skolem theorem. Let us note that we do not obtain a downward L\"owenheim-Skolem theorem for uncountable languages. Nevertheless, we will be able to derive a version of the \L{}o\'s-Vaught test for complete theories, which is somewhat weaker than usual but sufficient for applications (see Theorem~\ref{thm:los-vaught-test} and the paragraph that precedes it). We will see one such application in the following Subsection~\ref{subsect:acf}, where we give a proof of the Ax-Grothendieck theorem that a polynomial function~$\mathbb C^n\to\mathbb C^n$ over the complex numbers is surjective if it is injective.

More information on research in model theory can be found in a survey by Wilfrid Hodges~\cite{hodges-stanford} (see in particular Section~5 of that article). We point out that the present section focuses on the model theory of first-order logic with potentially infinite structures. There are other important directions such as finite model theory, which has close connections with computer science~\cite{finite-model-theory-applic}.

\subsection{Ultraproducts and Elementary Extensions}\label{subsect:ultraprod}

As noted in the introduction to the present section, this subsection discusses ultraproducts and chains. These construction will be used to prove an upward L\"owenheim-Skolem theorem (cf.~Theorem~\ref{thm:loewenheim-skolem} above) and a version of the \L{}o\'s-Vaught test.

Given a family of sets~$A_i$ with $i$ from a so-called index set~$I$, we write~$\prod_{i\in I}A_i$ for the set of all sequences~$(a_i)_{i\in I}$ with~$a_i\in A_i$ for each~$i\in I$. When~$I$ is clear from the context, we sometimes abbreviate~$(a_i)_{i\in I}$ as~$(a_i)$. For a given family of $\sigma$-structures~$\mathcal M_i$ with a common signature~$\sigma$, we want to define a new structure~$\mathcal M$ with universe~$\prod_{i\in I}M_i$, where~$M_i$ is the universe of~$\mathcal M_i$. It is not obvious how the interpretation~$R^{\mathcal M}$ of a predicate symbol in the product is to be defined from the interpretations~$R^{\mathcal M_i}$ in the given structures. One intuitive idea, which we state for the case of a unary predicate, is that $(a_i)\in R^{\mathcal M}$ should hold when we have~$a_i\in R^{\mathcal M_i}$ for most~$i\in I$. The following offers a way to make this precise.

\begin{definition}\label{def:filter}
By a filter on a set~$I$ we mean a subset~$\mathcal F\subseteq\mathcal P(I)$ of the power set that satisfies the following conditions:
\begin{enumerate}[label=(\roman*)]
\item We have $I\in\mathcal F$ and~$\emptyset\notin\mathcal F$.
\item For any~$A,B\in\mathcal F$ we also have~$A\cap B\in\mathcal F$.
\item Given $A\in\mathcal F$ and $A\subseteq B\subseteq I$, we get $B\in\mathcal F$.
\end{enumerate}
A filter~$\mathcal F$ on~$I$ is called an ultrafilter if we have $A\in\mathcal F$ or $I\backslash A\in\mathcal F$ for each~$A\subseteq I$.
\end{definition}

One may think of a filter as a collection of subsets that are large in a certain sense. Conditions~(i) and~(iii) are clearly reasonable in this context, while condition~(ii) may seem counterintuitive. To make sense of it, one can think of~$A$ and~$B$ as \emph{so}~large that~\emph{even} $A\cap B$ is still large. An instructive example is the so-called Fr\'echet filter on~$\mathbb N$, which consists of all sets~$A\subseteq\mathbb N$ such that~$\mathbb N\backslash A$ is finite.

In condition~(ii) of Definition~\ref{def:filter}, the implication is actually an equivalence, as condition~(iii) ensures that $A\cap B\in\mathcal F$ entails~$A,B\in\mathcal F$. If~$\mathcal F$ is an ultrafilter on the set~$I$, then $A\in\mathcal F$ is equivalent to~$I\backslash A\notin\mathcal F$.

\begin{exercise}\label{ex:ultrafilter}
Show that a filter~$\mathcal F$ on~$I$ is an ultrafilter precisely if $A\cup B\in\mathcal F$ is equivalent to the disjunction of $A\in\mathcal F$ and~$B\in\mathcal F$. \emph{Hint:} Since $I\backslash(A\cup B)$ is equal to $(I\backslash A)\cap(I\backslash B)$, the claim is dual to one from the previous paragraph.
\end{exercise}

For any element~$i_0\in I$, the collection $\mathcal F=\{A\subseteq I\,|\,i_0\in A\}$ is an ultafilter on~$I$, as one readily checks. An ultrafilter of the indicated form is called principal. The following result entails that there are non-principal ultrafilters (consider an extension of the Fr\'echet filter). It cannot be proved in Zermelo-Fraenkel set theory without the axiom of choice~\cite{feferman-ultrafilter} (which is equivalent to Zorn's lemma).

\begin{proposition}\label{prop:ultrafilters-exist}
For any filter~$\mathcal F$ on a set~$I$, there is an ultrafilter~$\mathcal U$ on~$I$ such that we have $\mathcal F\subseteq\mathcal U$.
\end{proposition}
\begin{proof}
Consider the collection of all filters~$\mathcal G$ on~$I$ such that we have~$\mathcal F\subseteq\mathcal G$, ordered by inclusion. It is not hard to see that the union over any non-empty chain in this collection is again a filter. Thus Zorn's lemma provides a filter~$\mathcal U$ on~$I$ that is maximal with~$\mathcal F\subseteq\mathcal U$. In order to show that~$\mathcal U$ is an ultrafilter, we assume~$I\backslash A\notin\mathcal U$ and derive~$A\in\mathcal U$. The crucial point is that we have $A\cap B\neq\emptyset$ for any~$B\in\mathcal U$, as we would otherwise get~$B\subseteq I\backslash A$, against condition~(iii) of Definition~\ref{def:filter}. Consider
\begin{equation*}
\mathcal U'=\{C\subseteq I\,|\,A\cap B\subseteq C\text{ for some }B\in\mathcal U\}.
\end{equation*}
It is straightforward to conclude that $\mathcal U'$ is a filter. By maximality we get $\mathcal U=\mathcal U'$ and hence~$A\in\mathcal U$.
\end{proof}

The previous proof shows that any maximal filter on a set is an ultrafilter. Conversely, one readily shows that any ultrafilter is maximal. We now present the construction that was sketched in the paragraph before Definition~\ref{def:filter}, where the reader will also find relevant notation.

\begin{definition}\label{def:ultraproduct}
Consider a signature~$\sigma$, an ultrafilter~$\mathcal U$ on a set~$I$ and a family of $\sigma$-structures~$\mathcal M_i$ for~$i\in I$. The ultraproduct $\mathcal M=\prod^{\mathcal U}_{i\in I}\mathcal M_i$ is the $\sigma$-structure with universe~$\prod_{i\in I}M_i$ (where $M_i$ is the universe of~$\mathcal M_i$) and
\begin{align*}
f^{\mathcal M}\left((a^1_i)_{i\in I},\ldots,(a^n_i)_{i\in I}\right)&:=\left(f^{\mathcal M_i}(a^1_i,\ldots,a^n_i)\right)_{i\in I},\\
\left((a^1_i)_{i\in I},\ldots,(a^n_i)_{i\in I}\right)\in R^{\mathcal M}\,&:\Leftrightarrow\,\left\{i\in I\,\left|\,(a^1_i,\ldots,a^n_i)\in R^{\mathcal M_i}\right.\right\}\in\mathcal U,
\end{align*}
where~$f$ and~$R$ range over the $n$-ary function and relation symbols from~$\sigma$.
\end{definition}

With the given definition, the interpretation of equality in an ultraproduct may not coincide with actual equality, i.\,e., the structure $\prod^{\mathcal U}_{i\in I}\mathcal M_i$ may not be strict even when each component~$\mathcal M_i$ is (cf.~Definition~\ref{def:equality}). One can always obtain a strict model by taking a quotient as in the proof of Theorem~\ref{thm:strict-model}. Many authors define the ultraproduct as this quotient, but we find it simpler to separate the two steps. The central result about ultraproducts reads as follows.

\begin{theorem}[\L{}o\'s's Theorem~\cite{los-ultraproduct}]
In the situation of Definition~\ref{def:ultraproduct} we have
\begin{equation*}
\textstyle\prod^{\mathcal U}_{i\in I}\mathcal M_i\vDash\varphi\left((a^1_i)_{i\in I},\ldots,(a^n_i)_{i\in I}\right)\quad\Leftrightarrow\quad\left\{i\in I\,\left|\,\mathcal M_i\vDash\varphi(a^1_i,\ldots,a^n_i)\right.\right\}\in\mathcal U,
\end{equation*}
for any $\sigma$-formula~$\varphi(x_1,\ldots,x_n)$ and any elements $(a^j_i)_{i\in I}\in\prod_{i\in I}M_i$.
\end{theorem}
\begin{proof}
Let us write~$\mathcal M=\prod^{\mathcal U}_{i\in I}\mathcal M_i$. A straightforward induction shows that
\begin{equation*}
t^{\mathcal M}\left((a^1_i)_{i\in I},\ldots,(a^n_i)_{i\in I}\right)=\left(t^{\mathcal M_i}(a^1_i,\ldots,a^n_i)\right)_{i\in I}
\end{equation*}
holds for any term~$t(x_1,\ldots,x_n)$. Here we recall that $t^{\mathcal M}(b^1,\ldots,b^n)$ abbreviates~$t^{\mathcal M,\eta}$ for the assignment~$\eta$ that maps $x_j$ to~$b^j$. The statement of the theorem involves a similar abbreviation.

To establish the theorem, we argue by induction on the number of occurrences of the symbols~$\land,\lor$ and $\forall,\exists$ in the formula~$\varphi$. The point is that the induction hypothesis will be available not just for~$\varphi$ but also for~$\neg\varphi$ (as we work with formulas in negation normal form, cf.~Definitions~\ref{def:signature} and~\ref{def_fv-formula}). Let us first assume that~$\varphi$ is of the form~$Rt_1\ldots t_k$, where the variables of each term~$t_j$ are among~$x_1,\ldots,x_n$. For notational simplicity, we assume $k=n=1$ and write~$\varphi$ as $Rt$ with~$t=t(x)$. In view of Definition~\ref{def:ultraproduct}, we get
\begin{align*}
\mathcal M\vDash\varphi\left((a_i)_{i\in I}\right)\quad&\Leftrightarrow\quad\left(t^{\mathcal M_i}(a_i)\right)_{i\in I}=t^{\mathcal M}\left((a_i)_{i\in I}\right)\in R^{\mathcal M}\\
\quad&\Leftrightarrow\quad\{i\in I\,|\,\mathcal M_i\vDash\varphi(a_i)\}=\left\{i\in I\,\left|\,t^{\mathcal M_i}(a_i)\in R^{\mathcal M_i}\right.\right\}\in\mathcal U,
\end{align*}
as desired. To cover the case where~$\varphi$ is of the form~$\neg Rt_1\ldots t_k$, it suffices to note that we have
\begin{equation*}
\left\{i\in I\,\left|\,\mathcal M_i\vDash\varphi(a^1_i,\ldots,a^n_i)\right.\right\}\in\mathcal U\quad\Leftrightarrow\quad\left\{i\in I\,\left|\,\mathcal M_i\nvDash\varphi(a^1_i,\ldots,a^n_i)\right.\right\}\notin\mathcal U.
\end{equation*}
For this duality principle, it is crucial that~$\mathcal U$ is an ultrafilter and not just a filter.

Let us now consider a formula $\varphi(x)$ of the form~$\varphi_0(x)\land\varphi_1(x)$. Here we again assume that there is just one free variable, which simplifies the notation. Due to condition~(ii) of Definition~\ref{def:filter}, the induction hypothesis yields
\begin{align*}
\mathcal M\vDash\varphi\left((a_i)_{i\in I}\right)\quad\Leftrightarrow\quad\{i\in I\,|\,\mathcal M_i\vDash\varphi_0(a_i)\}\in\mathcal U\text{ and }\{i\in I\,|\,\mathcal M_i\vDash\varphi_0(a_i)\}&\in\mathcal U\\
\Leftrightarrow\quad\{i\in I\,|\,\mathcal M\vDash\varphi(a_i)\}=\{i\in I\,|\,\mathcal M_i\vDash\varphi_0(a_i)\}\cap\{i\in I\,|\,\mathcal M_i\vDash\varphi_0(a_i)\}&\in\mathcal U.
\end{align*}
When~$\varphi$ has the form~$\varphi_0\lor\varphi_1$, one argues similarly, using Exercise~\ref{ex:ultrafilter}. Alternatively, one can employ the above duality principle (as in the last case below).

Next, we assume that $\varphi(x)$ has the form~$\exists y\,\psi(x,y)$. If we have $\mathcal M\vDash\varphi((a_i)_{i\in I})$, we can pick an element $(b_i)_{i\in I}\in\prod_{i\in I}M_i$ with $\mathcal M\vDash\psi((a_i)_{i\in I},(b_i)_{i\in I})$. Due to the induction hypothesis, we obtain
\begin{equation*}
\left\{i\in I\,\left|\,\mathcal M_i\vDash\varphi(a_i)\right.\right\}\supseteq\left\{i\in I\,\left|\,\mathcal M_i\vDash\psi(a_i,b_i)\right.\right\}\in\mathcal U.
\end{equation*}
For the converse implication, we pick an element $(b_i)_{i\in I}\in\prod_{i\in I}M_i$ such that we have $\mathcal M_i\vDash\psi(a_i,b_i)$ for any~$i\in I$ with $\mathcal M_i\vDash\varphi(a_i)$. Given $\{i\in I\,|\,\mathcal M_i\vDash\varphi(a_i)\}\in\mathcal U$, the induction hypothesis yields~$\mathcal M\vDash\psi((a_i)_{i\in I},(b_i)_{i\in I})$ and hence~$\mathcal M\vDash\varphi((a_i)_{i\in I})$. Finally, we assume that $\varphi(x)$ has the form~$\forall y\,\psi(x,y)$. Let us point out that~$\neg\varphi(x)$ coincides with~$\exists y\,\psi(x,y)$, since we work with formulas in negation normal form (see Definition~\ref{def:nega-pred-log}). As noted above, the induction hypothesis is available not just for~$\psi$ but also for~$\neg\psi$. Thus the argument from the previous case yields
\begin{equation*}
\mathcal M\vDash\neg\varphi\left((a_i)_{i\in I}\right)\quad\Leftrightarrow\quad\{i\in I\,|\,\mathcal M_i\vDash\neg\varphi(a_i)\}\in\mathcal U.
\end{equation*}
If we negate both sides of this equivalence, we obtain the one from the theorem, due to the duality principle that was mentioned above.
\end{proof}

In Section~\ref{sect:fundamentals}, we have assumed that all signatures are countable. Due to this assumption, we could prove completeness directly for cut-free proofs, which allowed us to derive Herbrand's theorem and Beth's definability theorem as straightforward consequences. Using \L{}o\'s's theorem, we will now prove completeness and compactness in a more general setting. The following generalizes Theorem~\ref{thm:compactness}.

\begin{theorem}[Compactness for Uncountable Signatures]\label{thm:compactness-uncountable}
Consider a set~$\Theta$ of sentences over some signature, which may be uncountable. If any finite subset of~$\Theta$ has a model, then there is a single model for the entire set~$\Theta$.
\end{theorem}
\begin{proof}
Let~$I$ be the set of all finite subsets of~$\Theta$. For each~$\theta\in I$, the assumption of the theorem allows us to pick a model~$\mathcal M_\theta$ such that $\mathcal M_\theta\vDash\varphi$ holds for all~$\varphi\in\theta$. Also for~$\theta\in I$, we define $\theta\!\uparrow$ as the set of all~$\theta'\in I$ with $\theta\subseteq\theta'$. Then
\begin{equation*}
\mathcal F=\{A\subseteq I\,|\,{\theta\!\uparrow}\subseteq A\text{ for some }\theta\in I\}
\end{equation*}
is a filter on~$I$. Indeed, the crucial condition~(ii) from Definition~\ref{def:filter} is satisfied in view of~${\theta\!\uparrow}\cap{\rho\!\uparrow}={(\theta\cap\rho)\!\uparrow}$. Proposition~\ref{prop:ultrafilters-exist} provides an ultrafilter~$\mathcal U\supseteq\mathcal F$ on~$I$. For every~$\varphi\in\Theta$ we have
\begin{equation*}
\{\theta\in I\,|\,\mathcal M_\theta\vDash\varphi\}\supseteq{\{\varphi\}\!\uparrow}\in\mathcal U,
\end{equation*}
so that \L{}o\'s's theorem yields~$\prod^{\mathcal U}_{\theta\in I}\mathcal M_\theta\vDash\varphi$. 
\end{proof}

As indicated, we also obtain a generalization of Theorem~\ref{thm:completeness-AL}.

\begin{theorem}[Completeness for Uncountable Signatures]
We have
\begin{equation*}
\Theta\vDash\varphi\quad\Rightarrow\quad\Theta\vdash\varphi
\end{equation*}
whenever~$\Theta$ is a set of $\sigma$-formulas and~$\varphi$ is a $\sigma$-formula, where the signature~$\sigma$ may be uncountable.
\end{theorem}
\begin{proof}
Aiming at the contrapositive, we assume~$\Theta\nvdash\varphi$. For any finite set~$\Theta_0\subseteq\Theta$, we find a finite sub-signature~$\sigma_0$ such that~$\varphi$ and the elements of~$\Theta_0$ can be considered as~$\sigma_0$-formulas. In view of~$\Theta_0\nvdash\varphi$, the completeness theorem for countable signatures yields a $\sigma_0$-structure~$\mathcal M_0$ together with a variable assignment~$\eta$ such that we have $\mathcal M_0,\eta\vDash\Theta_0\cup\{\neg\varphi\}$. Here we may replace~$\mathcal M_0$ by any $\sigma$-structure that extents it, i.\,e., that has the same universe and interprets the functions and relation from~$\sigma_0$ in the same way.

Let us consider an extension~$\sigma'$ of~$\sigma$ that contains a new constant symbol~$c_i$ for each variable~$x_i$. Given a~$\sigma$-formula~$\psi$, we write $\psi'$ for the formula that results from~$\psi$ when each free variable~$x_i$ is replaced by the corresponding constant~$c_i$. The point is that $\psi'$ will always be a sentence. When~$\Psi$ is a set of $\sigma$-formulas, we write~$\Psi'$ for the set~$\{\psi'\,|\,\psi\in\Theta\}$. Given~$\mathcal M_0,\eta\vDash\Theta\cup\{\neg\varphi\}$, we get~$\mathcal M_0'\vDash\Theta'\cup\{\neg\varphi'\}$ for the extension~$\mathcal M'_0$ of~$\mathcal M_0$ that interprets~$c_i$ as the value of~$x_i$ under $\eta$.

The previous theorem on compactness yields a $\sigma'$-structure~$\mathcal M'\vDash\Theta'\cup\{\neg\varphi'\}$. Conversely to the previous paragraph, we get $\mathcal M,\eta\vDash\Theta\cup\{\neg\varphi\}$ for some variable assignment~$\eta$, where~$\mathcal M$ is the restriction of~$\mathcal M'$ to a $\sigma$-structure. This means that we have established~$\Theta\nvDash\varphi$, as needed for our proof of the contrapositive.
\end{proof}

Let us recall that we speak of a strict model when the symbol~$=$ is interpreted as actual equality (rather than some other equivalence relation).

\begin{convention}
From this point on, we assume that any signature includes a designated symbol for equality and that any structure is strict (cf.~Definition~\ref{def:equality}).
\end{convention}

In the following we study notions of substructure.

\begin{definition}\label{def:substructure}
For a signature~$\sigma$, an embedding between two $\sigma$-structures $\mathcal M$ and~$\mathcal N$ is given by a function $h:M\to N$ between their universes such that
\begin{equation*}
\begin{array}{rcl}
h\left(f^{\mathcal M}(a_1,\ldots,a_n)\right)&\!\!\!\!=\!\!\!\!&f^{\mathcal N}(h(a_1),\ldots,h(a_n)),\\
(a_1,\ldots,a_n)\in R^{\mathcal M}&\!\!\!\!\Leftrightarrow\!\!\!\!&(h(a_1),\ldots,h(a_n))\in R^{\mathcal N}
\end{array}
\end{equation*}
holds for any $a_i\in M$, where $f$ and~$R$ range over function and relation symbols from~$\sigma$. Such an~em\-bed\-ding is called elementary if we have
\begin{equation*}
\mathcal M\vDash\varphi(a_1,\ldots,a_n)\quad\Leftrightarrow\quad\mathcal N\vDash\varphi(h(a_1),\ldots,h(a_n))
\end{equation*}
for all $a_i\in M$ and any $\sigma$-formula~$\varphi(x_1,\ldots,x_n)$. By an isomorphism of $\sigma$-structures, we mean an embedding that is surjective. We write $\mathcal M\subseteq\mathcal N$ and say that $\mathcal M$ is a substructure of~$\mathcal N$ when we have an inclusion $M\subseteq N$ between the universes and the inclusion map is an embedding. If the latter is elementary, we write $\mathcal M\preceq\mathcal N$ and speak of an elementary substructure.
\end{definition}

Let us note that any embedding is injective, since equality is among the relation symbols that it must respect. We also point out that the universe of a substructure~$\mathcal M\subseteq\mathcal N$ must be closed under the interpretation of function symbols, i.\,e., that we must have~$f^{\mathcal N}(a_1,\ldots,a_n)\in M$ for any~$a_i\in M$. It is instructive to verify the following fundamental property.

\begin{exercise}
Show that any isomorphism of $\sigma$-structures is elementary.
\end{exercise}

The next notion will be used to find an extension of a given structure, i.\,e., to realize the latter as a substructure.

\begin{definition}
Consider a signature~$\sigma$ and a $\sigma$-structure~$\mathcal M$. By~$\sigma_{\mathcal M}$ we denote the extension of~$\sigma$ by a new constant~$c_a$ for each~$a\in M$. The elementary diagram of $\mathcal M$ is defined as the set~$\operatorname{Diag}_{\operatorname{el}}(\mathcal M)$ of all $\sigma_{\mathcal M}$-sentences that hold in~$\mathcal M$, where the latter is considered as a $\sigma_{\mathcal M}$-structure with~$c_a^{\mathcal M}=a$. The atomic diagram of $\mathcal M$ is the set $\operatorname{Diag}_{\operatorname{at}}(\mathcal M)$ that consists of all literals in~$\operatorname{Diag}_{\operatorname{el}}(\mathcal M)$.
\end{definition}

In the following lemma, the equivalence remains valid if one uses the atomic diagram in~(i) and settles for a not necessarily elementary embedding in~(ii), by essentially the same proof.

\begin{lemma}
Given a signature~$\sigma$ and a $\sigma$-structure~$\mathcal M$, the following are equivalent for any $\sigma_{\mathcal M}$-structure~$\mathcal N$:
\begin{enumerate}[label=(\roman*)]
\item We have $\mathcal N\vDash\operatorname{Diag}_{\operatorname{el}}(\mathcal M)$.
\item The function that is given by $M\ni a\mapsto c_a^{\mathcal N}\in N$ is an elementary embedding of $\sigma_{\mathcal M}$-structures (where we assume~$c_a^{\mathcal M}=a$ as in the previous definition).
\end{enumerate}
\end{lemma}
\begin{proof}
To see that~(ii) implies~(i), we note that $\mathcal M\vDash\operatorname{Diag}_{\operatorname{el}}(\mathcal M)$ holds due to the~definition of the elementary diagram, so that we get $\mathcal N\vDash\operatorname{Diag}_{\operatorname{el}}(\mathcal M)$ by the definition of elementary embeddings. Let us now verify the converse implication. We assume~(i) and write $h:M\to N$ for the map that is defined in~(ii). In order to see that the latter is injective, it suffices to note that the diagram of~$\mathcal M$ contains the formula~$c_a\neq c_b$ for any pair of distinct elements~$a,b\in M$. Given an $n$-ary function symbol~$f$ from~$\sigma$ and arbitrary elements~$a_i\in M$, we put $b:=f^{\mathcal M}(a_1,\ldots,a_n)$ and note that $f(c_{a_1},\ldots,c_{a_n})=c_b$ occurs in the diagram of~$\mathcal M$. In view of~$c_a^{\mathcal N}=h(a)$, we can thus invoke~(i) to get
\begin{equation*}
f^{\mathcal N}(h(a_1),\ldots,h(a_n))=h(b)=h\left(f^{\mathcal M}(a_1,\ldots,a_n)\right),
\end{equation*}
as required by Definition~\ref{def:substructure}. For the new constants or $0$-ary function symbols $c_a$ from $\sigma_{\mathcal M}$, we have $h(c_a^{\mathcal M})=h(a)=c_a^{\mathcal N}$. Now consider an $n$-ary relation symbol~$R$. If we have $(a_1,\ldots,a_n)\notin R^{\mathcal M}$, then $\neg Rc_{a_1}\ldots c_{a_n}$ lies in the diagram of~$\mathcal M$, so that we obtain $(h(a_1),\ldots,h(a_n))\notin R^{\mathcal N}$ due to~(i). Similarly, we see that $(a_1,\ldots,a_n)\in R^{\mathcal M}$ implies $(h(a_1),\ldots,h(a_n))\in R^{\mathcal N}$. Hence $h$ is an embedding. In order to show that it is elementary, we first assume~$\mathcal M\nvDash\varphi(a_1,\ldots,a_n)$. The latter means that $\mathcal M,\eta\vDash\varphi$ holds for the variable assignment~$\eta$ that sends~$x_i$ to~$a_i$. When we consider $\mathcal M$ as a $\sigma_{\mathcal M}$-structure with~$c_a^{\mathcal M}=a$, we obtain $\mathcal M\nvDash\varphi(c_{a_1},\ldots,c_{a_n})$ due to Exercise~\ref{ex:substitution}. So the diagram of~$\mathcal M$ contains $\neg\varphi(c_{a_1},\ldots,c_{a_n})$. We now get $\mathcal N\nvDash\varphi(h(a_1),\ldots,h(a_n))$ thanks to~(i). Similarly, $\mathcal M\vDash\varphi(a_1,\ldots,a_n)$ entails $\mathcal N\vDash\varphi(h(a_1),\ldots,h(a_n))$.
\end{proof}

While the downward L\"owenheim-Skolem theorem from Section~\ref{sect:cut-free-completeness} (see Theorem~\ref{thm:loewenheim-skolem}) guarantees the existence of models that are relatively small, the following result provides large models. The reader with some exposure to set theory will note that the set~$\kappa$ in the theorem plays the role of a cardinal number. We have chosen a formulation that does not require familiarity with cardinals.

\begin{theorem}[Upward L\"owenheim-Skolem Theorem]\label{thm:up-LS}
For any set~$\kappa$ and any infinite model~$\mathcal M$, there is a model~$\mathcal N$ such that we have $\mathcal M\preceq\mathcal N$ and there is an injection~$\iota:\kappa\to N$.
\end{theorem}
\begin{proof}
Extend the given signature by a constant~$d_\alpha$ for each~$\alpha\in\kappa$. The theory
\begin{equation*}
\operatorname{Diag}_{\operatorname{el}}(\mathcal M)\cup\{d_\alpha\neq d_\beta\,|\,\alpha,\beta\in\kappa\text{ with }\alpha\neq\beta\}
\end{equation*}
has a model~$\mathcal N$ by compactness (see Theorem~\ref{thm:compactness-uncountable} and also Theorem~\ref{thm:strict-model}), since each finite subtheory is satisfied by~$\mathcal M$ when the finitely many relevant constants are interpreted by different elements. In view of~$\mathcal N\vDash d_\alpha\neq d_\beta$, the desired injection can be given by $\iota(\alpha)=d_\alpha^{\mathcal N}$. The previous lemma ensures that $M\ni a\mapsto c_a^{\mathcal N}\in N$ defines an elementary embedding. To turn the latter into an inclusion map, we replace the substructure~$\{c_a^{\mathcal N}\,|\,a\in M\}\subseteq\mathcal N$ by the isomorphic structure~$\mathcal M$.
\end{proof}

The previous theorem implies, in particular, that any theory with an infinite model has models of arbitrarily large cardinality. Concerning the following exercise, we note that only the easier direction (i)$\Rightarrow$(ii) will be needed below. 

\begin{exercise}[\L{}o\'s-Tarski Preservation Theorem]\label{ex:Los-Tarski}
By a $\forall$-formula, we mean a formula of the form~$\forall x_1\ldots\forall x_n\,\theta$ with quantifier-free~$\theta$. Show that the following are equivalent for any theory~$\mathsf T$: 
\begin{enumerate}[label=(\roman*)]
\item There is a theory~$\mathsf T'$ that consists of $\forall$-formulas and is equivalent to~$\mathsf T$, in the sense that $\mathcal M\vDash\mathsf T$ and $\mathcal M\vDash\mathsf T'$ are equivalent for every structure~$\mathcal M$.
\item The theory~$\mathsf T$ is downward absolute, which means that we get $\mathcal M\vDash\mathsf T$ when\-ever we have $\mathcal N\vDash\mathsf T$ and $\mathcal M$ is a substructure of~$\mathcal N$.
\end{enumerate}
\emph{Hint:} For the crucial direction from~(ii) to~(i), let $\mathsf T'$ be the set of all $\forall$-sentences~$\varphi$ such that we have $\mathsf T\vDash\varphi$. In the proof that $\mathcal M\vDash\mathsf T'$ implies $\mathcal M\vDash\mathsf T$, use the previous lemma to reduce to the claim that $\mathsf T\cup\operatorname{Diag}_{\operatorname{at}}(\mathcal M)$ is consistent. To prove the latter, note that $\mathsf T\vDash\neg\theta(c)$ implies $\mathsf T\vDash\forall x\,\neg\theta(x)$ when the constant~$c$ does not occur in~$\mathsf T$.
\end{exercise}

The following is a simple but useful criterion for elementarity.

\begin{lemma}[Tarski-Vaught Test]\label{lem:Tarski-Vaught-test}
Consider structures~$\mathcal M\subseteq\mathcal N$ and assume that $\mathcal N\vDash\exists x\,\varphi(x,b_1,\ldots,b_n)$ implies~$\mathcal N\vDash\varphi(a,b_1,\ldots,b_n)$ for some~$a\in M$, for any formula~$\varphi$ and arbitrary parameters~$b_i\in M$. Then we have $\mathcal M\preceq\mathcal N$.
\end{lemma}
\begin{proof}
We use induction over~$\psi$ to show that $\mathcal M\vDash\psi(\mathbf b)$ is equivalent to~$\mathcal N\vDash\psi(\mathbf b)$ for any tuple of parameters~$\mathbf b$ from~$M$. When~$\psi$ is a literal, this holds since~$\mathcal M$ is a substructure of~$\mathcal N$. The quantifier step for conjunctions and disjunctions is~straight\-forward. Now assume $\psi(\mathbf b)$ has the form~$\forall x\,\varphi(x,\mathbf a)$. If we have $\mathcal N\nvDash\psi(\mathbf b)$ and hence $\mathcal N\vDash\exists x\,\neg\varphi(x,\mathbf b)$, the assumption of the lemma yields~$\mathcal N\vDash\neg\varphi(a,\mathbf b)$ for some~$a\in M$. By induction hypothesis, we now get $\mathcal M\nvDash\varphi(a,\mathbf b)$ and hence~$\mathcal M\nvDash\psi(\mathbf b)$. The implication from $\mathcal N\vDash\forall x\,\varphi(x,\mathbf a)$ to $\mathcal M\vDash\forall x\,\varphi(x,\mathbf a)$ is a direct consequence of the induction hypothesis, as any $a\in M$ lies in~$N$ (cf.~the previous exercise). A~similar argument covers the case of an existential quantifier. In effect we have shown that $\mathcal M$ satisfies the elementary diagram of~$\mathcal N$, but the reference to the diagram does not simplify the present proof.
\end{proof}

Let us recall that two sets are equinumerous when there is a bijection between them. The following concerns a fundamental property of this notion.

\begin{exercise}[Schr\"oder-Bernstein Theorem]\label{ex:schroeder-bernstein}
Consider sets $A$ and $B$ for which there are injections $f:A\to B$ and $g:B\to A$. Prove that $A$ and~$B$ are equinumerous. \emph{Hint:} Draw~$A$ as a circle and~$B$ as a square. Successive appli\-cations of~$f$ and~$g$ yield fractal nestings of circles and squares in~$A$ and~$B$. The drawings suggest how to define the desired isomorphism.
\end{exercise}

In the next exercise, we prove some basic facts about cardinality that will be needed later. The hints suggest proofs that only requires minimal prerequisites from set theory.

\begin{exercise}\label{ex:card-seqs}
(a) Prove that any infinite set~$\kappa$ is equinumerous to~$\mathbb N\times\kappa$. \emph{Hint:}~Apply Zorn's lemma to the collection of all injections $\mathbb N\times X\to X$ with $X\subseteq\kappa$, ordered by extension. Show that if $\mathbb N\times X\to X$ is maximal in this order, then $\kappa\backslash X$ is finite, so that we get an injection of~$\mathbb N\times\kappa$ into~$\mathbb N\times X$.

(b) Derive that any infinite~$\kappa$ is equinumerous to~$\kappa\times\kappa$. \emph{Hint:} Apply Zorn's lemma to the injections $X\times X\to X$ with $X\subseteq\kappa$. Assume that such an injection is maximal. If $X$ and $\kappa$ are equinumerous, we are done. Otherwise, there is an injection~$\iota:X\to\kappa\backslash X$ (if not, use Zorn's lemma to get a maximal injection~\mbox{$Y\to\kappa\backslash X$} with $Y\subseteq X$, note that the latter must be bijective with inverse~$j:\kappa\backslash X\to X$, and construct injections $\kappa\to\mathbb N\times X\to X$). Let $X'$ be the image of~$\iota$ and construct an injection~$(X\cup X')\times(X\cup X')\to(X\cup X')$ that extends our maximal one.

(c) Let $\kappa$ be the set of function symbols (including constants) of some signature~$\sigma$. Show that the set of $\sigma$-terms is equinumerous to~$\kappa$ when the latter is infinite and that it is countable otherwise.
\end{exercise}

The following strengthens Theorem~\ref{thm:loewenheim-skolem} in two respects. It removes the countability condition and yields a model that is elementary in a given one.

\begin{theorem}[Uncountable Downward L\"owenheim-Skolem Theorem]\label{thm:downward-LS-uncount}
Consider a $\sigma$-model~$\mathcal N$ and an infinite~$X\subseteq N$ that is at least as numerous as~$\sigma$, i.\,e., such that there is an injection from the set of $\sigma$-formulas into~$X$. Then there is an elementary submodel~$\mathcal M\preceq\mathcal N$ with $X\subseteq M$ such that $X$ and~$M$ are equinumerous.
\end{theorem}
\begin{proof}
For each~$\sigma$-formula~$\varphi(x,y_1,\ldots,y_n)$ we pick a function~$F_\varphi:N^n\to N$ such that $\mathcal N\vDash\exists x\,\varphi(x,b_1,\ldots,b_n)$ implies
\begin{equation*}
\mathcal N\vDash\varphi\big(F_\varphi(b_1,\ldots,b_n),b_1,\ldots,b_n\big).
\end{equation*}
We now put~$M_0:=X$ and recursively declare that~$M_{k+1}$ consists of the elements of the set~$M_k$ and all values~$f^{\mathcal N}(b_1,\ldots,b_n)$ and $F_\varphi(b_1,\ldots,b_n)$ for~$b_1,\ldots,b_n\in M_k$, where $\varphi$ and~$f$ range over $\sigma$-formulas and function symbols from~$\sigma$. Given that the union $M:=\bigcup_{k\in\mathbb N}M_k$ is closed under each function~$f^{\mathcal N}$, we obtain a unique substructure~$\mathcal M\subseteq\mathcal N$ with universe~$M$. Since~$M$ is closed under the functions~$F_\varphi$, the Tarski-Vaught test (Lemma~\ref{lem:Tarski-Vaught-test}) yields~$\mathcal M\preceq\mathcal N$. Each set~$M_{k+1}$ is equinumerous to~$M_k$ and hence to~$X$. This follows from the previous exercise, since there is an injection from~$M_{k+1}$ into a set of terms with function symbols~$f$ and~$F_\varphi$ as well as constants for elements of~$M_k$. We thus obtain an injection of~$M$ into~$\mathbb N\times X$ and hence into~$X$. To conclude, we invoke the Schr\"oder-Bernstein theorem.
\end{proof}

In the following, one should again think of~$\kappa$ as a cardinal number.

\begin{definition}
Assume that we have an isomorphism between two models $\mathcal M$ and~$\mathcal N$ of a theory~$\mathsf T$ whenever the universes~$M$ and~$N$ are equinumerous to a set~$\kappa$. Then we say that $\mathsf T$ is $\kappa$-categorical.
\end{definition}

In the next section, we will employ the following result to show that the theory of algebraically closed fields of fixed characteristic is complete.

\begin{theorem}[\L{}o\'s-Vaught Test~\cite{los-categorical,vaught-categorical}]\label{thm:los-vaught-test}
Consider a theory~$\mathsf T$ over some signature~$\sigma$. Assume that $\mathsf T$ has no finite models and that it is $\kappa$-categorical for some~$\kappa$ that is at least as numerous as the set of~$\sigma$-formulas. Then $\mathsf T$ is complete, i.\,e., we have $\mathsf T\vDash\varphi$ or $\mathsf T\vDash\neg\varphi$ for any $\sigma$-sentence~$\varphi$.
\end{theorem}
\begin{proof}
Aiming at a contradiction, we assume that we have two models~$\mathcal M\vDash\mathsf T\cup\{\varphi\}$ and $\mathcal N\vDash\mathsf T\cup\{\neg\varphi\}$. Note that these models must be infinite by assumption. By the upward L\"owenheim-Skolem theorem (see Theorem~\ref{thm:up-LS}), we find elementary extensions~$\mathcal M\preceq\mathcal M'$ and~$\mathcal N\preceq\mathcal N'$ that admit injections~$\kappa\to M'$ and $\kappa\to N'$. Now the downward L\"owenheim-Skolem theorem (see Theorem~\ref{thm:downward-LS-uncount}) yields elementary substructures $\mathcal M''\preceq\mathcal M'$ and $\mathcal N''\preceq\mathcal N'$ such that $M''$ and~$N''$ are equinumerous to~$\kappa$. Given that~$\mathsf T$ is $\kappa$-categorical, we can infer that~$\mathcal M''$ and~$\mathcal N''$ are isomorphic. This, however, is impossible in view of~$\mathcal M''\vDash\varphi$ and~$\mathcal N''\vDash\neg\varphi$.
\end{proof}

To conclude this section, we discuss chains of models. Note that $f^{\mathcal M}$ is well-defined in the following construction as $\mathcal M_i\subseteq\mathcal M_j$ entails that all $a_1,\ldots,a_n\in M_i$ validate $f^{\mathcal M_i}(a_1,\ldots,a_n)=f^{\mathcal M_j}(a_1,\ldots,a_n)$. A~similar observation applies in the case of relation symbols.

\begin{definition}\label{def:chain}
A sequence~$\mathcal M_0,\mathcal M_1,\ldots$ of $\sigma$-structures for a common signature~$\sigma$ is called a chain if $\mathcal M_i\subseteq\mathcal M_{i+1}$ holds for all~$i\in\mathbb N$. The union $\mathcal M:=\bigcup_{i\in I}\mathcal M_i$ over such a chain is defined as the $\sigma$-structure with universe~$\bigcup_{i\in\mathbb N}M_i$ such that
\begin{equation*}
\begin{array}{rcl}
f^{\mathcal M}(a_1,\ldots,a_n)&\!\!\!\!=\!\!\!\!&f^{\mathcal M_i}(a_1,\ldots,a_n),\\
(a_1,\ldots,a_n)\in R^{\mathcal M}&\!\!\!\!\Leftrightarrow\!\!\!\!&(a_1,\ldots,a_n)\in R^{\mathcal M_i}
\end{array}
\end{equation*}
holds whenever we have $a_1,\ldots,a_n\in M_i$, where $f$ and~$R$ range over the $n$-ary function and relation symbols from~$\sigma$. A chain $\mathcal M_0,\mathcal M_1,\ldots$ is called elementary if we have $\mathcal M_i\preceq\mathcal M_{i+1}$ for every~$i\in\mathbb N$.
\end{definition}

By a $\forall\exists$-formula, we mean a formula that has the form~$\forall x_1\ldots\forall x_m\exists y_1\ldots\exists y_n\,\theta$ for quantifier-free~$\theta$. The following yields fundamental properties of chains.

\begin{proposition}\label{prop:chains}
For any chain $\mathcal M_0\subseteq\mathcal M_1\subseteq\ldots$, we have the following:
\begin{enumerate}[label=(\alph*)]
\item Given a $\forall\exists$-formula $\varphi$ with $\mathcal M_i\vDash\varphi$ for all~$i\in\mathbb N$, we get $\bigcup_{i\in I}\mathcal M_i\vDash\varphi$.
\item If the given chain is elementary, then we have $\mathcal M_j\preceq\bigcup_{i\in I}\mathcal M_i$ for all~$j\in I$.
\end{enumerate}
\end{proposition}
\begin{proof}
(a) Write $\varphi$ as $\forall\mathbf x\exists\mathbf y\,\theta(\mathbf x,\mathbf y)$ for quantifier-free~$\theta$. To show~$\bigcup_{i\in I}\mathcal M_i\vDash\varphi$, we consider arbitrary elements $\mathbf a\subseteq\bigcup_{i\in\mathbb N}M_i$. Pick~$i\in\mathbb N$ so large that we have $\mathbf a\subseteq M_i$. Given $\mathcal M_i\vDash\varphi$, we get $\mathcal M_i\vDash\exists\mathbf y\,\theta(\mathbf a,\mathbf y)$. Now $\mathcal M_i$ is a substructure of~$\bigcup_{i\in\mathbb N}\mathcal M_i$, by construction of the latter. We thus get $\bigcup_{i\in\mathbb N}\mathcal M_i\vDash\exists\mathbf y\,\theta(\mathbf a,\mathbf y)$ due to the implication from~(i) to~(ii) in Exercise~\ref{ex:Los-Tarski} (as the negation of $\exists\mathbf y\,\theta(\mathbf a,\mathbf y)$ is a $\forall$-formula).

(b) By induction on the build-up of~$\varphi$, we establish that $\mathcal M_j\vDash\varphi$ is equivalent to $\mathcal M:=\bigcup_{i\in I}\mathcal M_i\vDash\varphi$ for any $\sigma_{\mathcal M_j}$-sentence~$\varphi$. When $\varphi$ is a literal, the claim holds since $\mathcal M_j$ is a substructure of~$\mathcal M$. Now assume that $\varphi$ has the form~$\forall x\,\psi(x)$. Given the induction hypothesis, it is straightforward to show that $\mathcal M\vDash\varphi$ entails~$\mathcal M_j\vDash\varphi$ (argue as in the proof that (i) implies~(ii) in Exercise~\ref{ex:Los-Tarski}). To prove the converse, we assume $\mathcal M_j\vDash\varphi$. Aiming at $\mathcal M\vDash\varphi$, we consider an arbitrary~$b\in\bigcup_{i\in\mathbb N}M_i$. Pick a~$k\geq j$ with $b\in M_k$. Given $\mathcal M_j\preceq\mathcal M_k$, we get $\mathcal M_k\vDash\varphi$ and hence $\mathcal M_k\vDash\psi(b)$. The induction hypothesis yields $\mathcal M\vDash\psi(b)$, as required. For $\varphi$ of the form~$\exists x\,\psi(x)$, one concludes by a dual argument (prove both directions of our equivalence by contra\-position). The induction step is easy when~$\varphi$ is a conjunction or disjunction.
\end{proof}

Part~(a) of the proposition is sharp in the sense that any~$\mathsf T$ that is preserved under unions of chains (i.\,e., such that we get $\bigcup_{i\in\mathbb N}\mathcal M_i\vDash\mathsf T$ for any $\mathcal M_0\subseteq\mathcal M_1\subseteq\ldots$ with~$\mathcal M_i\vDash\mathsf T$) is equivalent to a set of $\forall\exists$-sentences. This result, which strikes a similar note as Exercise~\ref{ex:Los-Tarski}, is known as the Chang-\L{}o\'s-Suszko theorem.

\subsection{An Application to Algebraically Closed Fields}\label{subsect:acf}

In this section, we use tools from mathematical logic to prove the Ax-Grothendieck theorem, which asserts that any injective polynomial $p:\mathbb C^n\to\mathbb C^n$ is surjective (where $\mathbb C$ is the field of complex numbers). One can generalize the theorem to algebraic varieties, but we will not do so in the present lecture notes.

\begin{definition}
By the language of rings, we mean the signature~$\sigma_r$ that consists of constants~$0,1$ and binary function symbols $+,-,\times$ as well as the binary relation symbol~$=$ for equality. The theory~$T_f$ of fields is the $\sigma_r$-theory that consists of the equality axioms (cf.~Definition~\ref{def:equality}) and the universal closures of the formulas
\begin{gather*}
\begin{aligned}
x+(y+z)&=(x+y)+z,\qquad &x\times(y\times z)&=(x\times y)\times z,\\
x+y&=y+x,\qquad &x\times y&=y\times x,\\
x+0&=x,\qquad &x\times 1&=x,
\end{aligned}\\
x-x=0,\qquad x\neq 0\to\exists y\,(x\times y=1),\qquad 0\neq 1,\\
(x+y)\times z=(x\times z)+(y\times z).
\end{gather*}
In order to define a $\sigma_r$-term $\overline n$ for each~$n\in\mathbb N$, we write $\overline 0$ for~$0$ and recursively declare that $\overline{n+1}$ coincides with $\overline{n}+1$. For each prime number~$p$, we define~$T_{f,p}$ as the extension of~$T_f$ by the axioms~$\overline n\neq 0$ for~$n<p$ and the axiom~$\overline p=0$. We also define $T_{f,0}$ as the extension of~$T_f$ by the axioms $\overline n\neq 0$ for all~$n\in\mathbb N$.
\end{definition}

Note that a field is essentially the same as a $\sigma_r$-structure that satisfies~$T_f$. A field satisfies~$T_{f,p}$ precisely if it has characteristic~$p$ (where $p$ must be zero or prime).

\begin{proposition}\label{prop:pos-to-zero}
If a $\sigma_r$-theory is satisfied by fields of arbitrarily large positive characteristic, then it is satisfied by some field of characteristic zero.
\end{proposition}
\begin{proof}
Write~$T$ for the theory in question. We must show that $T_{f,0}\cup T$ is satisfi\-able. By assumption, $T_{f,p}\cup T$ is satisfiable for arbitrarily large~$p$. We can thus conclude by compactness (see Theorem~\ref{thm:compactness}).
\end{proof}

With respect to the following definition, it is instructive to observe that a single sentence can refer to all polynomials of some fixed degree. At the same time, we cannot quantify over the exponents of a polynomial, since our variables range over field elements rather than natural numbers. As usual, we agree that $\times$ binds stronger than~$+$, and we omit parentheses in sums with several summands.

\begin{definition}\label{def:acf}
For a $\sigma_r$-term~$t$ and $n\in\mathbb N$, we declare that $t^n$ denotes~$1$ when we have $n=0$ and that it coincides with $t^{n-1}\times t$ otherwise. Let $\mathsf{ACF}$ be the $\sigma_r$-theory that extends $T_f$ by the sentences
\begin{equation*}
\forall a_0\ldots\forall a_n\exists x\,\left(x^{n+1}+a_n\times x^n+\ldots+a_0\times x^0=0\right)
\end{equation*}
for all~$n\in\mathbb N$. When~$p$ is zero or prime, we let $\mathsf{ACF}_p$ be the union of~$T_{f,p}$ and $\mathsf{ACF}$.
\end{definition}

We point out that a field satisfies~$\mathsf{ACF}$ precisely if it is algebraically closed, i.\,e., if every non-constant polynomial has a root in the field. Even though the following proof relies on some algebra, its general idea should be understandable without many prerequisites.

\begin{theorem}
The theory $\mathsf{ACF}_p$ is complete for each~$p$ that is prime or zero.
\end{theorem}
\begin{proof}
Let us first note that no finite field~$k$ is algebraically closed, because the~poly\-nomial~$1+\prod_{a\in k}(x-a)$ cannot have a solution in~$k$. Due to the \L{}o\'s-Vaught test (see Theorem~\ref{thm:los-vaught-test}), is suffices to show that two algebraically closed fields~$K_0$ and $K_1$ of characteristic~$p$ are isomorphic when they are uncountable and equinumerous.

Let $\mathbb F_p$ stand for~$\mathbb Q$ when $p$ is zero and for the unique field with $p$ elements when $p$ is prime. Just like any field of characteristic~$p$, each of our fields~$K_i$ is a field extension of~$\mathbb F_p$. We pick a transcendence basis $B_i\subseteq K_i$ of each extension. This means, first, that no polynomial equation with coefficients in~$\mathbb F_p$ and multiple variables has a solution in~$B_i$. Secondly, it means that any element of~$K_i$ is a zero of some non-zero polynomial with coefficients~in~$\mathbb F_p(B_i)$. Since each polynomial does only have countably many zeros, $K_i$ is equinumerous to~$\mathbb F_p(B_i)$, by Exercise~\ref{ex:card-seqs}. The elements of~$\mathbb F_p(B_i)$ are quotients of polynomials, which can be seen as terms over a signature with the elements of $\mathbb F_p\cup B_i$ as constants. It follows that $B_i$ is equinumerous to~$\mathbb F_p(B_i)$ and hence to~$K_i$ (note that $B_i$ cannot be finite as~$K_i$~is uncountable). We can conclude that $B_0$ and $B_1$ are equinumerous. The desired isomorphism between $K_0$ and $K_1$ is now obtained via a theorem of Steinitz, which says that an algebraically closed field is uniquely determined by its characteristic and transcendence degree~\cite{steinitz} (see also~\cite[Section~VI.1]{hungerford-algebra}).
\end{proof}

In the following exercie, the familiar notion of basis for vector spaces replaces the notion of transcendence basis that occurs in the previous proof.

\begin{exercise}
Axiomatize the theory of infinite $k$-vector spaces in a language that contains (amongst others) a unary function symbol~$m_\lambda$ for scalar multiplication with each element~$\lambda$ of the field~$k$. Then show that this theory is complete.
\end{exercise}

We derive the following transfer principle.

\begin{corollary}
If a $\sigma_r$-theory is satisfied by algebraically closed fields of arbitrarily large positive characteristic, it is satisfied by the field~$\mathbb C$ of complex numbers.
\end{corollary}
\begin{proof}
Write $\mathsf T$ for the theory in question. Given that $\mathsf{ACF}_p\cup T$ is satisfiable for arbitrarily large~$p$, we can use Proposition~\ref{prop:pos-to-zero} to infer that $\mathsf{ACF}_0\cup T$ is satisfiable. So we have $\mathsf{ACF}_0\nvDash\neg\varphi$ for any formula~$\varphi$ in~$\mathsf T$. Since $\mathsf{ACF}_0$ is complete by the previous theorem, we get $\mathsf{ACF}_0\vDash\mathsf T$. The result follows due to~$\mathbb C\vDash\mathsf{ACF}_0$.
\end{proof}

Any function~$f:K^m\to K^n$ corresponds to an $n$-tuple of functions $f_i:K^m\to K$ with $f(\mathbf a)=(f_1(\mathbf a),\ldots,f_n(\mathbf a))$. We say that~$f$ is polynomial if each~$f_i$ is given by a polynomial in $m$ variables. As mentioned before, the following result can be generalized from~$\mathbb C^n$ to algebraic varieties.

\begin{theorem}[Ax-Grothendieck Theorem~\cite{ax-finite-fields,grothendieck-EGA4.3}]
Let $f:\mathbb C^n\to\mathbb C^n$ be a polynomial function for some~$n\in\mathbb N$. If $f$ is injective, then it is surjective.
\end{theorem}
\begin{proof}
For each~$m\in\mathbb N$, there is a $\forall\exists$-sentence~$\varphi_m$ over the signature~$\sigma_r$ (cf.~the paragraph before Proposition~\ref{prop:chains}) such that a field~$K$ satisfies~$\varphi_m$ precisely if every injective function~$f:K^n\to K^n$ that is given by polynomials of total degree at most~$m$ is also surjective. To see this, note that we can quantify over the poly\-nomials of fixed degree by quantifying over coefficients, as in Definition~\ref{def:acf}. If we write $p_i(x_1,\ldots,x_n)$ for suitable polynomial terms (with additional free variables as coefficients), the sentence~$\varphi_m$ can be given as the universal closure of
\begin{multline*}
\forall x_1\ldots\forall x_n\forall y_1\ldots\forall y_n\big(\textstyle\bigwedge_{i=1}^n p_i(x_1,\ldots,x_n)=p_i(y_1,\ldots,y_n)\to\textstyle\bigwedge_{i=1}^n x_i=y_i\big)\\
\to\forall z_1\ldots\forall z_n\exists x_1\ldots\exists x_n\big(\textstyle\bigwedge_{i=1}^n z_i=p_i(x_1,\ldots,x_n)\big).
\end{multline*}
We now consider the theory~$T=\{\varphi_m\,|\,m\in\mathbb N\}$. Our task is to show that this theory is satisfied by~$\mathbb C$. Due to the previous corollary, it is enough if we prove that it is satisfied by algebraically closed fields of positive characteristic.

Let us recall that there is a (necessarily unique) field~$\mathbb F_q$ with $q\in\mathbb N$ elements precisely if~$q$ has the form~$p^k$ for a positive integer~$k$ and a prime~$p$, which must be the characteristic of~$\mathbb F_q$. We note that $\mathbb F_q$ is not given by $\mathbb Z/q\mathbb Z$ with addition and multiplication modulo~$q$ in general, even though it is when~$q$ itself is prime. The field $\mathbb F_{p^k}$ is isomorphic to a (necessarily unique) subfield of~$\mathbb F_{p^m}$ precisely if $k$ divides~$m$. We shall assume that $F_{p^{k!}}$ is always realized as a subfield of~$F_{p^{(k+1)!}}$. The algebraic closure of~$\mathbb F_p$ can then be given as the increasing union
\begin{equation*}
\overline{\mathbb F_p}=\textstyle\bigcup_{k\in\mathbb N\backslash\{0\}}\mathbb F_{p^{k!}}.
\end{equation*}
We note that this is a union over a chain in the sense of Definition~\ref{def:chain}. The~above theory~$T$ is satisfied by each of the finite fields~$\mathbb F_{p^{k!}}$, because an injection from a finite set to itself must clearly be surjective. Given that $T$ consists of $\forall\exists$-sentences, we can use Proposition~\ref{prop:chains} to conclude that~$T$ is satisfied by the algebraically closed field~$\overline{\mathbb F_p}$ of characteristic~$p$, as needed.
\end{proof}

In the previous proof, we have exploited the fact that any injection from a finite set to itself must be surjective. Of course, it is also true that any surjection from a finite set to itself must be injective. As the following exercise shows, however, we do not get a corresponding result about polynomials over~$\mathbb C$.

\begin{exercise}
Find a polynomial function $f:\mathbb C\to\mathbb C$ that is surjective but not injective. Which part of the previous proof breaks down when injectivity and surjectivity are interchanged?
\end{exercise}

\section{Epilogue: Set Theory}\label{sect:set-theory}

This final section gives a brief overview of set theory. In the preceding sections on proof theory, computability and model theory, the guiding principle was to present one reasonably advanced and interesting result from each area in detail. To represent set theory, the author would have liked to prove the relative consistency of the axiom of choice, but this turned out to go beyond the scope of this lecture. In the view of the author, it would not do set theory justice to present some basic constructions without a proper application. For this reason, we content ourselves with a short overview and reserve details for a separate lecture on set theory.

Very informally, the universe of sets is generated by the recursive clause that any collection of sets is a set. This clause is problematic since it is not quite clear what a collection should be if it is not a set, in which case the clause is circular. Nevertheless, the given clause can guide our intuition to a certain extent. In particular, it suggests that there is an empty set~$\emptyset$, which then gives rise to another set~$\{\emptyset\}$. More generally, we can represent the natural numbers as sets by stipulating~$0:=\emptyset$ and $n+1:=n\cup\{n\}$ (observe $1=\{\emptyset\}$). Let us also note that the ordered pair of sets $x$ and $y$ can be represented by the set $\langle x,y\rangle:=\{\{x\},\{x,y\}\}$, as one can show that $\langle x,y\rangle=\langle x',y'\rangle$ entails $x=x'$ and~$y=y'$. By the familiar constructions, we obtain sets that represent the integers, the rationals and the reals. Ultimately, all the usual objects of mathematics can be construed as sets. We note that everything is built from the empty set, i.\,e., that no urelements are needed.

One might think that it should be possible to form the set~$r:=\{x\,|\,x\notin x\}$, which is supposed to consist of all sets that are not members of themselves. However, the definition of~$r$ entails that $r\in r$ is equivalent to~$r\notin r$. This is a contradiction, which is known as Russel's paradox. To avoid the latter and to make the notion of set precise, mathematicians have developed axiom systems for set theory. By far the best-known is Zermelo-Fraenkel set theory ($\mathsf{ZF}$). One writes $\mathsf{ZFC}$ for the extension of~$\mathsf{ZF}$ by the axiom of choice~($\mathsf{AC}$), which asserts that any set~$x$ with $\emptyset\notin x$ admits a function~$f$ with domain~$x$ such that we have $f(v)\in v$ for all~$v\in x$. To explain how Russel's paradox is excluded, we first note that $\forall x\,(x\notin x)$ is provable in~$\mathsf{ZF}$. Hence the~$r$ from the paradox would be the collection of all sets. A possible analysis of the paradox is that this collection is too big to be a set. One usually says that it is a proper class. In general, the collection~$\{x\,|\,\varphi(x)\}$ of all sets~$x$ that validate some formula~$\varphi$ is not a set according to~$\mathsf{ZFC}$. Instead, the latter allows to form~$\{x\in a\,|\,\varphi(x)\}$ as a subset of a given set~$a$. For $r_0:=\{x\in a\,|\,x\notin x\}$, the previous argument does not lead to a contradiction but only to the conclusion that we have~$r_0\notin a$. In view of G\"odel's incompleteness theorems, we cannot really prove the consistency of~$\mathsf{ZFC}$. Nevertheless, extensive scrutiny by logicians and 100 years of successful use in mathematics have led to a wide consensus that~$\mathsf{ZFC}$ provides a trustworthy foundation.

In comparison with first-order arithmetic (see Section~\ref{subsect:PA}), the notion of set is far more powerful in two respects. First, set theory includes infinite objects. Secondly, it seems very hard to find a sentence in the language of first-order arithmetic that is unprovable in Peano arithmetic~($\mathsf{PA}$) and has the intuitive appeal of an axiom (cf.~Isaacson's thesis in~\cite{incurvati-isaacson,smith-isaacson}). On the other hand, $\mathsf{ZF}$ is much stronger than~$\mathsf{PA}$ and consists of axioms (organised in just eight axiom schemes) that are rather plausible explications of the idea of a set. It is also attractive that the signature of~$\mathsf{ZF}$ contains only one symbol besides equality, namely the binary relation symbol~$\in$ for set membership. Whether the enormous strength of $\mathsf{ZF}$ is only an advantage can be debated. The point here is that $\mathsf{ZF}$ is much stronger than needed for most of mathematical practice, so that other frameworks will often allow for a finer analysis (see the results on Goodstein sequences in Section~\ref{sect:goodstein} and the discussion of reverse mathematics in the introduction to Section~\ref{sect:computability}). In any case, it is an extremely impressive and important achievement that $\mathsf{ZFC}$ provides a single standard of proof that is accepted across mathematics.

It is common in set theory to write $\omega$ for~$\mathbb N$, which is considered as the smallest infinite number (cf.~the paragraph after Definition~\ref{def:omega^omega}). If we define the next number as $\omega+1:=\omega\cup\{\omega\}$ and continue in a similar fashion, we obtain the so-called ordinals. In the context of ordinals, we also write $<$ for~$\in$, which is a well order (and in particular linear) on each ordinal and on the class of ordinals as a whole. A cardinal is defined as an ordinal that is not equinumerous to any smaller ordinal. For example, $\omega$ is a cardinal while $\omega+1$ is not. In $\mathsf{ZF}$ one can show that any well order is isomorphic to a unique ordinal. Under the axiom of choice, any set~$x$ is equinumerous to a (clearly unique) cardinal~$|x|$. A set $x$ can never be equinumerous to its power set~$\mathcal P(x)$, i.\,e., to the set of all subsets of~$x$. To see this, note that $z:=\{y\in x\,|\,y\notin f(y)\}$ cannot lie in the image of~$f:x\to\mathcal P(x)$, since $f(y)=z$ would make $y\in z$ and $y\notin z$ equivalent. It follows that the cardinals form an unbounded subclass of the ordinals. The theory of well orders admits an extremely elegant formulation in terms of ordinals. Arguably, our formulation of the \L{}o\'s-Vaught test (see Theorem~\ref{thm:los-vaught-test}) can be made more natural with the help of cardinals.

The continuum hypothesis~($\mathsf{CH}$) is the statement that there is no cardinal~$\kappa$ such that we have $\omega<\kappa<|\mathcal P(\omega)|$, i.\,e., that~$\mathbb R$ is equinumerous with any uncountable subset of itself. It was proposed by Georg Cantor, who is commonly credited with the creation of set theory. The first of the 23 problems that David Hilbert famously presented at the International Congress of Mathematicians in~1900 asks for a proof or a refutation of the continuum hypothesis. We now know that neither is possible on the basis of~$\mathsf{ZFC}$. Indeed, G\"odel showed in~1938 that we have $\mathsf{ZFC}\nvdash\neg\mathsf{CH}$ as well as $\mathsf{ZF}\nvdash\neg\mathsf{AC}$, assuming that~$\mathsf{ZF}$ is consistent. Under the same assumption, Paul Cohen proved in 1963 that we have $\mathsf{ZFC}\nvdash\mathsf{CH}$ and $\mathsf{ZF}\nvdash\mathsf{AC}$. Cohen received a fields medal for this work, in which the method of forcing is introduced. The latter has been extremely important for the further development of set theory. To name one early application of mathematical interest, we note that Kaplansky's conjecture on homomorphisms between Banach algebras is independent of~$\mathsf{ZFC}$ (see~\cite{dales-woodin-kaplansky}).

To argue for the truth or falsity of statements that are independent of~$\mathsf{ZFC}$, set theorists have looked for additional axioms. The latter can have an intrinsic or an extrinsic justification, i.\,e., they can be convincing because they explicate our intuitive conception of sets or because they have fruitful and unifying consequences for the field. Some of the fascinating results of this search for new axioms are discussed in Peter Koellner's survey on `Large Cardinals and Determinacy'~\cite{koellner-determinacy}. We also refer to survey articles by Joan Bagaria~\cite{bagaria-intro} (for a more comprehensive introduction to set theory and further pointers to the literature), Laura Crosilla~\cite{crosilla-constructive-zf} (for constructive set theory) and M.~Randall Holmes~\cite{holmes-alternative-set-theories} (for other set theories).

\bibliographystyle{amsplain}
\bibliography{Intro-Math-Log_Freund}

\end{document}